\documentclass[12pt]{amsart}
\usepackage{dsfont}            
\usepackage{amssymb}  
\usepackage[utf8]{inputenc}
\usepackage{mathtools}
\usepackage{amsthm}  
\usepackage{indentfirst}  
\usepackage{amsmath}
\usepackage{graphicx}
\usepackage{times}
\usepackage{color}
\usepackage[shortlabels]{enumitem}
\usepackage[colorlinks=true]{hyperref}

\numberwithin{equation}{section}

\newtheorem{theorem}{Theorem}[section]
\newtheorem{proposition}[theorem]{Proposition}  
  
\newtheorem{lemma}[theorem]{Lemma}

\newtheoremstyle{claim} % name
{1em}                        % Space above
{1em}                        % Space below
{}                           % Body font
{}                           % Indent amount
{\bfseries}                  % Theorem head font
{.}                          % Punctuation after theorem head
{.5em}                       % Space after theorem head
{}  % Theorem head spec (can be left empty, meaning ÔnormalÕ)

\theoremstyle{claim}

\theoremstyle{remark}

\newtheorem*{remark*}{Remark}

\theoremstyle{definition}
\newtheorem{definition}[theorem]{Definition}

\DeclareMathOperator{\Arg}{Arg} % argument (phase) of complex number
\begin{document}

\newcommand{\secref}[1]{\S\ref{#1}}
\newcommand{\figref}[1]{Figure~\ref{#1}}
\newcommand{\appref}[1]{Appendix~\ref{#1}}
\newcommand{\tabref}[1]{Table~\ref{#1}}
\newcommand{\thmref}[1]{Theorem~\ref{#1}}
\newcommand{\propref}[1]{Proposition~\ref{#1}}
\newcommand{\defref}[1]{Definition~\ref{#1}}
\newcommand{\lemref}[1]{Lemma~\ref{#1}}

\def\be{\begin{equation}}
\def\ee{\end{equation}}
\def\bear{\begin{eqnarray}}
\def\eear{\end{eqnarray}}
\def\nn{\nonumber}

\newcommand{\C}{\mathbb{C}}
\newcommand{\R}{\mathbb{R}}
\newcommand{\Z}{\mathbb{Z}}
\newcommand{\Q}{\mathbb{Q}}
\newcommand{\N}{\mathbb{N}}

\def\sZ{{s}} % Zero of $\zeta(s)$
\def\sA{{s}} % Argument of $\zeta(s)$ (as a function, not as a complex number)
\newcommand\innerP[2]{{\left\langle{#1},{#2}\right\rangle}} % zeta-inner product

\def\xb{{{u}}} % elements
\def\xc{{{w}}} % elements
\def\xd{{{v}}} % elements

\def\lBanach{{{\ell}}}

\def\matAX{{{A}}} % anti hermitian matrix with extra 1/2

\def\genFb{{\Xi}} % generating function for $\xb_k$
\def\genFd{{\Theta}} % generating function for $\xd_k$
\def\genFc{{\Upsilon}} % generating function for $\xc_k$
\def\tF{{{t}}} % argument of generating function
\def\btF{{\overline{\tF}}} % argument of generating function - cc
\def\uF{{{\rho}}} % function of \tF
\def\buF{{\overline{\uF}}} % argument of generating function - cc
\def\vF{{{\gamma}}} % argument of generating function
\def\ConstF{{{C}}} % Constant of integration
\def\uMF{{{\omega}}} % log of rational expression in $\sqrt{\tF}$
\def\rMF{{{\tilde{\uMF}}}} % another log of rational expression in $\sqrt{\tF}$

\newcommand\Fs[1]{{{\mathfrak{F}_{{#1}}}}} % functional on polynomials

\def\aPoly{{{\mathbf a}}}
\def\baPoly{{\overline{\mathbf a}}} % complex conjugate
\def\bPoly{{{\mathbf b}}}

\def\PolyP{{{\mathbf P}}} % any polynomial
\def\PolyPid{{{\mathbf S}}} % Pidduck polynomial
\def\PolyOrth{{{\mathbf Q}}} % our orthogonal polynomials
\def\PolyML{{{\mathbf M}}} % Mittag-Leffler polynomial

\def\zA{{{\mathbf{z}}}}
\def\eqdef{{\,{\stackrel{\text{\tiny def}}{=}}\,}}
\def\eqnotreally{{\,{\stackrel{\text{?}}{=}}\,}}
\def\genFQk{{\widetilde{\Phi}}} % Generating function for orthogonal polynomials
\def\DilOp{{{\mathcal D}}} % dilatation operator

\def\Bernoulli{{{\mathbf B}}} % Bernoulli number

\def\genFMLk{{\widetilde{\Psi}}} % Generating function for Mittag-Leffler

\def\xvar{{\mathbf{x}'}}

\def\nI{{{n}}} % index
\def\mI{{{m}}} % index
\def\jI{{{j}}} % index
\def\kI{{{k}}} % index
\def\lI{{{l}}} % index

\def\sqt{{\alpha}} % $\sqrt{\tF}$

\def\ContourC{{{\mathcal C}}} % contour
\def\ContourH{{{\mathcal C_H}}} % contour
\def\ContourCirc{{{\widetilde{\mathcal C}}}} % contour
\def\ConstC{{{\mathbf C}}} % constant

\def\PathP{{{\mathcal P}}} % path
\def\SemiC{{{\mathcal S}}} % semicircle
\def\SemiRad{{{\varepsilon}}} % semicircle radius
\def\betaC{{{\beta}}}

\def\xiC{{\xi}}

\def\func{{{f}}} % some function

\def\zC{{\mathfrak{z}}}

\def\vecxb{{{\mathbf{u}}}} % vector of xb coefficients
\def\vecxd{{{\mathbf{v}}}} % vector of xd coefficients
\def\vecxdt{{{\tilde{\mathbf{v}}}}} % vector of tilde-xd coefficients
\def\bvecxb{{{\overline{\mathbf{u}}}}} % c.c. of vecxb
\def\matPX{{{{P}}}}

\def\intXi{{\Xi}}

\def\seqxb{{\left(\xb_{\kI}\right)_{\kI=1}^\infty}} % sequence of $\xb$'s
\def\seqxc{{\left(\xc_{\kI}\right)_{\kI=1}^\infty}} % sequence of $\xc$'s
\def\seqxd{{\left(\xd_{\kI}\right)_{\kI=0}^\infty}} % sequence of $\xd$'s
\def\seqxbz{{\left(\xb_{\kI}\right)_{\kI=0}^\infty}} % sequence of $\xb$'s

\def\seqPolyOrth{{\left(\PolyOrth_\kI\right)_{\kI=0}^\infty}} % sequence of $\PolyOrth$

\def\seqV{{\mathbf{V}}}

\def\genFcA{{\widetilde{\Upsilon}}} % \genFc minus a constant

\def\ReuF{{\mathfrak{x}}} % real part of $\uF$.
\def\RevF{{\mathfrak{y}}} % real part of $\vF$.

\def\DomainD{{{\mathcal D}}}
\def\factorL{{{\ell}}}

\def\vecf{{\mathbf{f}}}
\def\veck{{\mathbf{k}}}

\def\genP{{\mathcal{P}}} % generating function of two variables
\def\tG{{\tau}} % variable of generating function
\def\pG{{\sigma}} % variable of generating function
\def\genf{{\mathcal{F}}} % generating function for a vector f

\def\xG{{\mathbf{x}}} % new variable
\def\yG{{\mathbf{y}}} % new variable
\def\zG{{\mathbf{z}}} % new variable

\def\xygeng{{\mathcal{G}}} % g - variable change

\def\genFcD{{\Omega}} % generating function for $\xc_k$, derivative of $\genFc$

\def\matK{{{K}}} % matrix used for non-simple zeros section

\def\matU{{{B}}} % matrix used for non-simple zeros section
\def\matL{{{W}}} % matrix used for non-simple zeros section

\def\OpWLT{{\matL}} % operator on Hilbert space
\def\OpWUT{{\matU}} % operator on Hilbert space

\def\tH{{{\kappa}}} % new variable
\def\pH{{{\varpi}}} % new variable

\def\ConstFunctionX{{\mathcal{W}}}

\def\SqP{{{\mathbf{R}}}}

\def\gM{{\nu}} % shift to 1/2+\gM+it
\def\genPM{{\widetilde{\genP}}}
\def\matPXM{{\widetilde{\matPX}}}

\def\qG{{{\mathbf{q}}}}

\title{
On Pidduck polynomials \\ and zeros of the Riemann zeta function
}


%    author one information
\author{Ori J. Ganor}
\address{366 Physics North MC 7300, University of California, Berkeley, CA 94720, U.S.A.}
\curraddr{}
\email{ganor@berkeley.edu}
\thanks{}

%    author two information
%\author{}
%\address{}
%\curraddr{}
%\email{}
%\thanks{}

%    \subjclass is required.
\subjclass[2022]{Number Theory, Riemann Hypothesis}

\date{September 28, 2022}

\dedicatory{}

%    Abstract is required.
\begin{abstract}
For $1<p<\infty$, we prove that a necessary and sufficient condition for $\sZ$ to be a zero of the Riemann zeta function in the strip $0<\Re\sZ<1$ is that
$$
\begin{pmatrix}
1 & \frac{1}{3} & \frac{1}{5} & \frac{1}{7} &  \frac{1}{9} & \cdots \\
 & & & & & \\
-\frac{\sA}{3} &  1 & \frac{1}{3} & \frac{1}{5} & \frac{1}{7} & \cdots \\
 & & & & & \\
-\frac{\sA}{5} & -\frac{\sA}{5}  &  1 & \frac{1}{3} & \frac{1}{5} & \cdots \\
 & & & & & \\
-\frac{\sA}{7} &-\frac{\sA}{7}  & -\frac{\sA}{7}  &  1 &  \frac{1}{3} & \cdots \\
 & & & & & \\
-\frac{\sA}{9} & -\frac{\sA}{9} & -\frac{\sA}{9} & -\frac{\sA}{9} &  1 & \cdots \\
% & & & & & \\
\vdots &\vdots & \vdots &\vdots & \vdots & \ddots\\
\end{pmatrix}
\begin{pmatrix}
\xd_0 \\ \\ \xd_1 \\ \\ \xd_2 \\ \\ \xd_3 \\ \\ \vdots \\  \vdots \\
\end{pmatrix}
=0
$$
has a nontrivial solution $\seqxd\in\lBanach^p$.
A similar matrix equation was discovered by K.~M.~Ball in 2017, but the current paper offers a different (and independent) perspective. In this paper an explicit formula for $\xd_\kI$ is constructed in terms of Pidduck polynomials. In the process, it is also shown that Pidduck polynomials form an orthogonal basis with respect to an inner product of polynomials $f,g$ whereby we replace in a formal expression ``$\sum_{n=1}^\infty (-1)^{n+1}n \overline{f(n^2)} g(n^2)$'' the divergent sums ``$\sum_{n=1}^\infty (-1)^{n+1}n^{1+2k}$'' with their zeta-function regularized values. We also discuss the modification for possible non-simple zeros and conclude with applications to the question of the simplicity of the zeros and a relation to the Hilbert-P\'olya program.

\end{abstract}

\maketitle


% ======================================================================
\section{Introduction}

The Hilbert-P\'olya approach to the Riemann Hypothesis is a quest for an anti-Hermitian operator whose eigenvalues $\lambda$ correspond to zeros $\sZ=\lambda+\frac{1}{2}$ of the Riemann zeta function $\zeta(\sZ)$. Berry and Keating \cite{BerryKeating1,BerryKeating2} matched the asymptotic distribution of eigenvalues of the requisite operator (if it exists) with a semiclassical calculation of the phase-space area (with an appropriate cut-off) for the dilatation operator $x\frac{d}{dx}$, and Bender, Brody and M\"uller \cite{Bender:2016wob} constructed an explicit operator, albeit not anti-Hermitian, with the requisite spectrum. Operators related to dilatation, acting on various spaces of functions, were also proposed in \cite{Berry:1986,Connes:1998,Meyer:2005,Sierra:2011tb,Sierra:2016rgn,Bishop:2018mgy}.
(See \cite{Schumayer:2011yp,Sands:2022} for a recent review of additional related approaches.)
Connections between the Riemann Hypothesis and hermiticity or unitarity in higher dimensional quantum systems have also recently been proposed in \cite{Remmen:2021zmc,Bianchi:2022mhs,Benjamin:2022pnx}, and a connection between Robin's criterion and bounds on multiplicities of states in certain gauge theories was suggested in \cite{Honda:2022hvy}.)

The goal of this paper is to present an alternative, perhaps simpler, linear criterion for $\sA$ to be a zero of $\zeta(\sA)$ in terms of existence of a solution to a sequence of linear equations in the Banach space $\lBanach^p$ (for any $1< p<\infty$).

We will show that for fixed $p>1$ (and $p<\infty$), if $\Re\sA>0$ then $\zeta(\sA)=0$ if and only if
\be\label{eqn:start}
\begin{pmatrix}
1 & \frac{1}{3} & \frac{1}{5} & \frac{1}{7} &  \frac{1}{9} & \frac{1}{11} &\cdots \\
 & & & & & & \\
-\frac{\sA}{3} &  1 & \frac{1}{3} & \frac{1}{5} & \frac{1}{7} & \frac{1}{9} &\cdots \\
 & & & & & & \\
-\frac{\sA}{5} & -\frac{\sA}{5}  &  1 & \frac{1}{3} & \frac{1}{5} & \frac{1}{7} &\cdots \\
 & & & & & & \\
-\frac{\sA}{7} &-\frac{\sA}{7}  & -\frac{\sA}{7}  &  1 &  \frac{1}{3} & \frac{1}{5} &\cdots \\
 & & & & & & \\
-\frac{\sA}{9} & -\frac{\sA}{9} & -\frac{\sA}{9} & -\frac{\sA}{9} &  1 & \frac{1}{3} &\cdots \\
 & & & & & & \\
-\frac{\sA}{11} & -\frac{\sA}{11} & -\frac{\sA}{11} & -\frac{\sA}{11} &  -\frac{\sA}{11} &  1  &\cdots \\
% & & & & & \\
\vdots & \vdots &\vdots & \vdots &\vdots & \vdots & \ddots\\
\end{pmatrix}
\begin{pmatrix}
\xd_0 \\ \\ \xd_1 \\ \\ \xd_2 \\ \\ \xd_3 \\ \\ \xd_4 \\ \\ \vdots \\  \vdots \\
\end{pmatrix}
=0
\ee
has a nontrivial solution $\seqxd\in\lBanach^p$.

One way to arrive at \eqref{eqn:start} is through a construction of a series of orthogonal polynomials with respect to the inner product $\innerP{\PolyP_1}{\PolyP_2}$ defined by a regularized version of $\sum_{n=1}^\infty(-1)^n n\PolyP_1(n^2)\PolyP_2(n^2)$, whereby the divergent sum has to be expanded in terms of the formal sums $\sum_{n=1}^\infty (-1)^n n^{2k+1}$ (for $k\in\mathbb{N}$) and then those sums have to be replaced by the analytic continuation of $\sum_{n=1}^\infty (-1)^n n^{-\sA}$ at $\sA=-2k-1$. (Such a regularization is ubiquitous in String Theory. See for instance \cite{Sonnenschein:2018aqf,Sonnenschein:2020jbe} for a recent application.) It turns out that these orthogonal polynomials can be expressed in terms of Pidduck polynomials, which are closely related to Mittag-Leffler polynomials. (Coefficients of Mittag-Leffler polynomials also appear in a recently discovered series expansion for the zeta function \cite{Rzadkowski:2012}.) The linear criterion for $\sZ$ to be a zero of $\zeta(\sZ)$ arises by using the inner product to formally express $n^{-(\sZ+1)/2}$ in the basis of the orthogonal polynomials and then requiring it to be an eigenfunction of the dilatation operator. We will present the heuristic argument in more details below, but also subsequently provide a formal proof. To pursue that route, it will be convenient to recast the main result \eqref{eqn:start} in a different way, stated below.

% -----------------------------------------------------------------------------------------
\subsection{Main result}

Let $\zeta(\sA)$ be the Riemann zeta-function. We will prove the following necessary and sufficient condition for $\sA$ to be a zero. [The condition is easily shown to be equivalent to \eqref{eqn:start}, but is more convenient to work with.]

\begin{theorem}[]\label{thm:main}
Suppose $\zeta(\sA)=0$ and $\sA$ is not an even negative integer. Then, there exists an infinite sequence $\seqxb$ with $\xb_k\in\C$ and
\begin{itemize}
\item
$\xb_1\neq 0$, 
\item
$|\xb_{\kI+1}-\xb_\kI|=O(\frac{1}{\kI}\log\kI)$ as $\kI\rightarrow\infty$,
%[and hence $|\xb_\kI|=O((\log\kI)^2)$],
\end{itemize}
and such that the following matrix identity holds:
\be\label{eqn:OpW}
\begin{pmatrix}
\frac{1}{1\cdot 3} & \frac{1}{3\cdot 5} & \frac{1}{5\cdot 7} & \frac{1}{7\cdot 9} & \cdots \\
 & & & & \\
-\frac{\sA}{2\cdot 3}-\frac{1}{2} &  \frac{1}{1\cdot 3} & \frac{1}{3\cdot 5} & \frac{1}{5\cdot 7} & \cdots \\
 & & & & \\
0 & -\frac{\sA}{2\cdot 5}-\frac{1}{2}  &  \frac{1}{1\cdot 3} & \frac{1}{3\cdot 5} & \cdots \\
 & & & & \\
0 &0  & -\frac{\sA}{2\cdot 7}-\frac{1}{2}  &  \frac{1}{1\cdot 3} & \cdots \\
 & & & & \\
0 & 0 & 0 & -\frac{\sA}{2\cdot 9}-\frac{1}{2} & \cdots \\
 & & & & \\
\vdots &\vdots & \vdots &\vdots & \ddots\\
\end{pmatrix}
\begin{pmatrix}
\xb_1 \\ \\ \xb_2 \\ \\ \xb_3 \\ \\ \xb_4 \\ \\ \vdots \\ \\ \vdots \\
\end{pmatrix}
=0,
\ee
where \eqref{eqn:OpW} is to be read as a sequence of convergent series:
\be\label{eqn:OpWseq}
-\frac{1}{2}\left(\frac{s}{2\kI-1}+1\right)\xb_{\kI-1}
+\sum_{n=0}^\infty\frac{\xb_{\kI+n}}{(2n+1)(2n+3)}=0,
\ee
for $k=1,2,\dots$, with $\xb_0\eqdef 0$.

%Conversely, if $0<\Re\sA<1$, $\xb_j\neq0$ for at least one $j$, and the sequence $\seqxb$ is such that the sequence of differences behaves asymptotically as $|\xb_{\kI+1}-\xb_\kI|=O(\frac{1}{\kI}\log\kI)$, and the sequence itself satisfies \eqref{eqn:OpW}, then $\zeta(\sA)=0$.

Conversely, if $0<\Re\sA<1$, and there exists a sequence $\seqxb$ with $\xb_1\neq 0$ that satisfies \eqref{eqn:OpW}, and the sequence of differences
\be\label{eqn:xdDef}
\xd_\kI\eqdef\xb_{\kI+1}-\xb_\kI,\qquad\kI=1,2,\dots,\qquad(\xb_0\eqdef 0),
\ee
satisfies $\seqxd\in\lBanach^p$ for some $1< p<\infty$ (i.e., $\sum_{\kI=0}^\infty|\xd_\kI|^p<\infty$), then $\zeta(\sA)=0$.
\end{theorem}

% -----------------------------------------------------------------------------------------
\subsection{Sketch of proof}

The necessary condition for $\sA$ to be a nontrivial zero is shown in \secref{sec:necpf}, and the sufficient condition is shown in \secref{sec:sufpf}.
The general idea for the proof of the sufficient condition is to find a sequence $\seqxc$, with $\xc_\kI\in\C$, such that
\bear
\lefteqn{
\begin{pmatrix}
1 & 0 & 0 & 0 & 0 & \cdots \\
\end{pmatrix} =
}\nn\\ &&
\begin{pmatrix}
\xc_1 & \xc_2 & \xc_3 & \xc_4 & \xc_5 & \cdots \\
\end{pmatrix}
\begin{pmatrix}
\frac{1}{1\cdot 3} & \frac{1}{3\cdot 5} & \frac{1}{5\cdot 7} & \frac{1}{7\cdot 9} & \cdots \\
 & & & & \\
-\frac{\sA}{2\cdot 3}-\frac{1}{2} &  \frac{1}{1\cdot 3} & \frac{1}{3\cdot 5} & \frac{1}{5\cdot 7} & \cdots \\
 & & & & \\
0 & -\frac{\sA}{2\cdot 5}-\frac{1}{2}  &  \frac{1}{1\cdot 3} & \frac{1}{3\cdot 5} & \cdots \\
 & & & & \\
0 &0  & -\frac{\sA}{2\cdot 7}-\frac{1}{2}  &  \frac{1}{1\cdot 3} & \cdots \\
 & & & & \\
0 & 0 & 0 & -\frac{\sA}{2\cdot 9}-\frac{1}{2} & \cdots \\
 & & & & \\
\vdots &\vdots & \vdots &\vdots & \ddots\\
\end{pmatrix}.
%\begin{pmatrix}
%1 \\ \\ 0 \\ \\ 0 \\ \\ 0 \\ \\ 0 \\ \\ \vdots \\
%\end{pmatrix}.
\nn\\ &&
\label{eqn:xcW0}
\eear
Then, if $|\xc_\kI|$ falls off fast enough as $\kI\rightarrow\infty$, we can combine \eqref{eqn:xcW0} with \eqref{eqn:OpW} to argue that $\xb_1=0$. We will see that precisely if $\zeta(\sA)\neq 0$, a solution with the required large $\kI$ behavior does indeed exist. The technical details are left for \secref{sec:sufpf}.

As for the necessary condition for $\sA$ to be a zero of $\zeta$, in \secref{sec:necpf} we will explicitly construct $\seqxb$ with expressions for the $\xb_\kI$'s taking the form
$$
\xb_\kI=\sum_{j=0}^\kI\aPoly_{\kI,j}(1-2^{1+2j-\sA})\zeta(\sA-2j),\qquad
\kI=0,1,\dots,
$$
where the $\aPoly_{\kI,j}$'s are rational coefficients of a series of polynomials 
$$
\PolyOrth_\kI(x)\eqdef\sum_{j=0}^\kI\aPoly_{\kI,j} x^j
$$
to be defined shortly. The general motivation for this construction is presented next, while the formal proof is deferred to \secref{sec:necpf}.

% -----------------------------------------------------------------------------------------
\subsection{Motivation}

The $\xb_\kI$'s that solve \eqref{eqn:OpW} are defined as follows. First, instead of the Riemann zeta function, it is more convenient to work with the Dirichlet eta function,
\be\label{eqn:DirichletEta}
\eta(\sA) = (1-2^{1-\sA})\zeta(\sA),
\ee
and zeros of $\zeta(\sA)$ correspond to zeros of $\eta(\sA)$.
For $\Re\sA>0$, we have  $\eta(\sA) = \sum_{n=1}^\infty\frac{(-1)^{n+1}}{n^\sA}$.
Next, for $\sA\in\C$ that is not a positive odd integer, define (the umbral operator) $\Fs{\sA}:\C[x]\rightarrow\C$ by
$$
\Fs{\sA}(\PolyP)\eqdef\sum_{j=0}^n\aPoly_j\eta(\sA-2j)
\qquad
\text{for a polynomial $\PolyP(x)=\sum_{j=0}^n \aPoly_j x^j$.}
$$
$\Fs{\sA}(\PolyP)$ can be regarded as a regularized version of the formal, divergent, expression ``$\sum_{n=0}^\infty (-1)^{n+1}n^{-\sA}\PolyP(n^2)$'', whereby we replace the formal sum ``$\sum(-1)^{n+1}n^{2j-\sA}$'' [that would appear when expanding ``$\sum_{n=0}^\infty (-1)^{n+1}n^{-\sA}\PolyP(n^2)$''] with $\eta(\sA-2j)$.
We then define an inner product on $\C[x]$ by
\be\label{eqn:innerPDef}
\innerP{\PolyP_1}{\PolyP_2}\eqdef \Fs{-1}(\PolyP_1\overline{\PolyP}_2),\qquad
\text{for $\PolyP_1,\PolyP_2\in\C[x]$.}
\ee
This turns out to be a positive definite inner product, as will be shown in \secref{sec:innerp}.
Next, we identify a series of orthogonal polynomials $(\PolyOrth_\kI)_{\kI=0}^\infty$ with respect to $\innerP{\cdot}{\cdot}$, such that $\PolyOrth_\kI(0)=1$ and $\kI=\deg\PolyOrth_\kI$ for $\kI=0,1,2,\dots$. We will show that for $\sA$ a nontrivial zero of the zeta function, the series defined by 
\be\label{eqn:xbkdef}
\xb_\kI\eqdef\Fs{\sA}(\PolyOrth_\kI),\qquad\text{for $\kI=1,2,\dots$}
\ee
satisfies \eqref{eqn:OpW}. Note that if $\sA$ is an even negative integer, $\xb_\kI=0$ for all $\kI=1,2,\dots$, hence the requirement in \thmref{thm:main} for $\sA$ to be a nontrivial zero.

The proof of \thmref{thm:main} uses the generating function for $(\PolyOrth_\kI)_{\kI=0}^\infty$, which turns out to be
\be\label{eqn:genFQkDef}
\genFQk(x,\tF)\eqdef
\sum_{\kI=0}^\infty\PolyOrth_\kI(x)\tF^\kI
=\frac{1}{(1-\tF)}\cosh\left\lbrack
\sqrt{x}\log\left(\frac{1+\sqrt{\tF}}{1-\sqrt{\tF}}\right)
\right\rbrack.
\ee
Note that, despite the appearance of $\sqrt{x}$ and $\sqrt{\tF}$ in \eqref{eqn:genFQkDef}, $\genFQk(x,\tF)$ is a power series in $\Q[x][[\tF]]$, i.e., with only integer powers of $x$ and $\tF$, and rational coefficients.

The heuristic motivation for \eqref{eqn:OpW} is that \eqref{eqn:xbkdef} can be regarded (informally!) as an expansion of the noninteger power $x^{-(\sA+1)/2}$ in the basis $(\PolyOrth_\kI)_{\kI=0}^\infty$. [See \eqref{eqn:xbkdefAlt} below.] The function $x^{-(\sA+1)/2}$ is an eigenfunction of the dilatation operator
\be\label{eqn:dilatOp}
\DilOp\eqdef x\frac{d}{dx}
\ee
and equation \eqref{eqn:OpW}, as we will see, represents that in the basis $(\PolyOrth_\kI)_{\kI=0}^\infty$. In \secref{sec:necpf} we will present a formal proof using the generating function \eqref{eqn:genFQkDef}.

To obtain \eqref{eqn:start} from \eqref{eqn:OpW}, we note that for $N>0$ we have
\bear
\lefteqn{
\sum_{n=0}^N\frac{\xb_{\kI+n}}{(2n+1)(2n+3)} =
\frac{1}{2}\sum_{n=0}^N\left(\frac{1}{2n+1}-\frac{1}{2n+3}\right)\xb_{\kI+n}
}\nn\\
&=&
\frac{1}{2}\xb_\kI
+\frac{1}{2}\sum_{n=1}^N\frac{1}{2n+1}\left(\xb_{\kI+n}-\xb_{\kI+n-1}\right)
-\frac{\xb_{\kI+N}}{2(2N+3)}\,,
\nn
\eear
and since by \eqref{eqn:xdDef} and \thmref{thm:main},
\be\label{eqn:xbfromxd}
\xb_\kI=\sum_{\nI=0}^{\kI-1}\xd_\nI=O([\log\kI]^2),
\ee
we can take the limit $N\rightarrow\infty$ and rewrite \eqref{eqn:OpWseq} as
\bear
\qquad0&=&
-\frac{\sA}{2(2\kI-1)}\xb_{\kI-1}
+\frac{1}{2}\sum_{n=0}^\infty\frac{\xd_{\kI+n-1}}{2n+1}
\label{eqn:OpWxd}\\
&=&
-\frac{\sA}{2(2\kI-1)}\sum_{\nI=0}^{\kI-2}\xd_{\nI}
+\frac{1}{2}\sum_{n=0}^\infty\frac{\xd_{\kI+n-1}}{2n+1}
\,,\qquad
\kI=1,2,3,\dots
\label{eqn:OpWxdnob}
\eear
The expression \eqref{eqn:OpWxdnob} can be visualized in matrix form as \eqref{eqn:start}.
The expression \eqref{eqn:OpWxd} will be useful in \secref{sec:sufpf}.

% -----------------------------------------------------------------------------------------
\subsection{Paper structure}

The paper is organized as follows. 
%Sections \secref{sec:innerp}-\sec{sec:OrthogonalPolynomials} are independent of the proof of the main theorem, and are provided for motivation.
We begin in \secref{sec:innerp} with basic properties of the inner product \eqref{eqn:innerPDef}, including the proof that it is positive definite. In \secref{sec:OrthogonalPolynomials} we study the orthogonal polynomials with respect to that inner product, and we derive the generating function \eqref{eqn:genFQkDef}, and hence the relation with Pidduck and Mittag-Leffler polynomials. At the end of that section we derive the matrix elements of the dilatation operator, which correspond to the matrix elements of the matrix that appears in \eqref{eqn:OpW}.
\secref{sec:necpf}-\secref{sec:sufpf} present the proof of the main theorem.
The bulk of these sections is a technical analysis of the growth rate of the coefficients $\xb_\kI$ and $\xb_{\kI+1}-\xb_{\kI}$, as well as the coefficients $\xc_\kI$ from \eqref{eqn:xcW0}. This analysis is crucial for the convergence of the series \eqref{eqn:OpWseq}, as well as the validity of certain steps later on, where the order of sums are interchanged. In \secref{sec:NonSimpleZeros} we present a characterization of (potential) zeros $\sZ$ of higher order, as a condition to be added to \eqref{eqn:start}. In \secref{sec:HilbertPolya}, we consider \eqref{eqn:start} in connection with the Hilbert-P\'olya program. We attempt to recast \eqref{eqn:start} as a spectral problem for a self-adjoint operator -- an attempt that (not surprisingly) fails, albeit in a somewhat technical way. In the process, we discover an additional infinite sequence of constant ($\sZ$-independent) functionals that annihilate $\seqxd\in\lBanach^p$. In principle, they could be added to \eqref{eqn:start}, although we don't have a form as simple as \eqref{eqn:start} for those other functionals. [See \eqref{eqn:SqPGenFun} and \eqref{eqn:sumSqPxd}.] We conclude in \secref{sec:Discussion} with a summary and suggestions for further study.

% -----------------------------------------------------------------------------------------
\subsection*{Note added}
After this work was completed, the author was informed of previous work by Ball \cite{Ball:2017,Ball:2019} that arrived at a very similar result by proving a rational approximation to $\zeta(\sA)$ whose numerator is a determinant of a matrix similar to a truncated version of the one appearing in \eqref{eqn:start}.

% ======================================================================
%\section{Motivation}\label{sec:Motiv}

% ======================================================================
\section{The inner product}\label{sec:innerp}

Let $\PolyP_1(x)=\sum_{j=0}^n \aPoly_j x^j$ and $\PolyP_2(x)=\sum_{j=0}^m \bPoly_j x^j$ be polynomials in $\C[x]$. Equation \eqref{eqn:innerPDef} defines an inner product
\be\label{eqn:innerPrepeat}
\innerP{\PolyP_1}{\PolyP_2}
=\sum_{j=0}^n\sum_{k=0}^m\eta(-1-2j-2k)\aPoly_j\overline{\bPoly}_k
%=\sum_{j=0}^n\sum_{k=0}^m
%\aPoly_j\bPoly_k(4^{1+j+k}-1)\zeta(-1-2j-2k)\aPoly_j\overline{\bPoly}_k
%=\sum_{j=0}^n\sum_{k=0}^m
%\frac{(1-4^{1+j+k})\Bernoulli_{2j+2k+2}}{2j+2k+2}\aPoly_j\overline{\bPoly}_k
=\sum_{k=0}^{n+m}
\frac{(4^{1+k}-1)\Bernoulli_{2k+2}}{2k+2}\left(\sum_{j=0}^k\aPoly_j\overline{\bPoly}_{k-j}\right)
\ee
where $\Bernoulli_n$ is a Bernoulli number.

Let
$$
\genFQk_1(x,\tF)=\sum_{k=0}^\infty\PolyP_{1,k}(x)\tF^k,\qquad
\genFQk_2(x,\tF)=\sum_{k=0}^\infty\PolyP_{2,k}(x)\tF^k
$$
be formal power series in $\tF$ with coefficients that are polynomials in $x$. We define
$$
\innerP{\genFQk_1(x,\tF_1)}{\genFQk_2(x,\tF_2)}\eqdef
\sum_{k,m=0}^\infty\innerP{\PolyP_{1,k}}{\PolyP_{1,m}}\tF_1^k\btF_2^m
\in\C[[\tF_1,\btF_2]].
$$

\begin{lemma}
Viewing $\cosh(\sqrt{x}\tF)=\sum_{n=0}^\infty\frac{1}{(2n)!}x^n\tF^{2n}$ as a formal power series in $\C[x][[\tF]]$, we have the identity
\be\label{eqn:innerPcoshcosh}
\innerP{\cosh(\sqrt{x}\tF_1)}{\cosh(\sqrt{x}\tF_2)}=
\frac{1+\cosh\tF_1\cosh\btF_2}{(\cosh\tF_1+\cosh\btF_2)^2}
\ee
\end{lemma}

\begin{proof}
\bear
\lefteqn{
\innerP{\cosh(\sqrt{x}\tF_1)}{\cosh(\sqrt{x}\tF_2)}=
\sum_{k,m=0}^\infty \frac{\tF_1^{2k}\btF_2^{2m}}{(2k)!(2m)!}\innerP{x^k}{x^m}
}\nn\\ &=&
\sum_{k,m=0}^\infty \frac{\tF_1^{2k}\btF_2^{2m}}{(2k)!(2m)!}
\frac{(4^{1+k+m}-1)\Bernoulli_{2k+2m+2}}{2k+2m+2}
\nn\\
%&=&
%\sum_{n=0}^\infty\frac{(4^{1+n}-1)\Bernoulli_{2n+2}}{2n+2}
%\left(\sum_{k=0}^n\frac{\tF_1^{2k}\btF_2^{2n-2k}}{(2k)!(2n-2k)!}\right)
%\nn\\
&=&
\sum_{n=0}^\infty\frac{(4^{1+n}-1)\Bernoulli_{2n+2}}{2(2n+2)(2n)!}
\left\lbrack
(\tF_1+\btF_2)^{2n}+(\tF_1-\btF_2)^{2n}
\right\rbrack
\nn\\
&=&
\frac{1}{8\cosh^2(\frac{\tF_1+\btF_2}{2})}
+\frac{1}{8\cosh^2(\frac{\tF_1-\btF_2}{2})}
=
\frac{1+\cosh\tF_1\cosh\btF_2}{2(\cosh\tF_1+\cosh\btF_2)^2}\,.
\nn
\eear
%where we used
%$$
%\sum_{m=0}^\infty\frac{\Bernoulli_{2m+2}\tF^{2m}}{(2m+2)(2m)!}
%=\frac{1}{4\sinh^2(\frac{\tF}{2})}-\frac{1}{\tF^2}
%$$
\end{proof}

An integral representation for the inner product is provided as follows.

\begin{proposition}
Let $\PolyP_1(x)$ and $\PolyP_2(x)$ be polynomials. Then,
\be\label{eqn:measure}
\innerP{\PolyP_1}{\PolyP_2} = \int_{-\infty}^0
\frac{\PolyP_1(x)\overline{\PolyP}_2(x)dx}{2\sinh\left(\pi\sqrt{-x}\right)}
\ee
\end{proposition}

\begin{proof}
Using a well-known identity for the integral of $x^{-\sA}/\sinh x$ (see, e.g., \S25.5.8 of \cite{NIST}), as well as the definition of the Dirichlet eta function \eqref{eqn:DirichletEta}, and the functional equation of the Riemann zeta function, we can write for $\Re\sA<0$,
\be\label{eqn:intxsAsinh}
\int_0^\infty\frac{x^{-\sA}dx}{\sinh x}
=
2(2^{\sA-1}-1)\Gamma(1-\sA)\zeta(1-\sA)
=-\frac{\pi^{1-\sA}\eta(\sA)}{\sin\left(\frac{\pi\sA}{2}\right)}\,.
\ee
Thus, for integer $\jI,\kI\ge 0$,
$$
\int_{-\infty}^0
\frac{x^{\jI+\kI}dx}{2\sinh\left(\pi\sqrt{-x}\right)}
%=
%\int_0^\infty(-1)^{\jI+\kI}\frac{y^{2\jI+2\kI}y dy}{\sinh\left(\pi y\right)}
%=\int_0^\infty\pi^{-2-2\jI-2\kI}(-1)^{\jI+\kI}\frac{y^{2\jI+2\kI+1} dy}{\sinh y}
%=\int_0^\infty\pi^{-2-2\jI-2\kI}(-1)^{\jI+\kI}\frac{y^{2\jI+2\kI+1} dy}{\sinh y}
=\eta(-1-2\jI-2\kI)
=\innerP{x^\jI}{x^\kI}.
$$
Equation \eqref{eqn:measure} follows by linearity.
\end{proof}

% ======================================================================
\section{Orthogonal polynomials}
\label{sec:OrthogonalPolynomials}

Given the inner product \eqref{eqn:innerPrepeat}, we define a sequence of orthogonal polynomials in the standard way.
% inductively, by setting $\PolyOrth_0=1$ and defining (for $k=1,2,\dots$) $\PolyOrth_k$ to be the unique polynomial of degree $k$ that satisfies $\PolyOrth_k(0)=1$ and $\innerP{\PolyOrth_k}{\PolyOrth_j}=0$ for $0\le j\le k-1$.

\begin{definition}\label{defn:defQk}
The polynomials $\seqPolyOrth$ are defined inductively, starting with $\PolyOrth_0(x)=1$, by the following three conditions (for $k=1,2,\dots$):
\begin{enumerate}
\item[(a)]
$\deg\PolyOrth_\kI=\kI$;
\item[(b)]
$\PolyOrth_\kI(0)=1$;
\item[(c)]
$\innerP{\PolyOrth_\kI}{\PolyOrth_\jI}=0$ for $0\le\jI\le\kI-1$.
\end{enumerate}
\end{definition}

For example, the first four are
$$
\begin{array}{ll}
\PolyOrth_0(x)=1, &
\PolyOrth_1(x) = 1+2x,\\
\PolyOrth_2(x)=1+\tfrac{10}{3}x+\tfrac{2}{3}x^2, &
\PolyOrth_3(x)=1+\tfrac{196}{45}x+\tfrac{14}{9}x^2+\tfrac{4}{45}x^3.
\\
\end{array}
$$
The polynomials $\PolyOrth_k$ turn out to be closely related to Pidduck polynomials \cite{Pidduck}, the latter being related to the perhaps more familiar Mittag-Leffler polynomials $\left(\PolyML_\kI\right)_{\kI=0}^\infty$ whose generating function can be expressed as 
\be\label{eqn:genFMLk}
\genFMLk(x,\tF)=\sum_{\nI=0}^\infty\PolyML_\nI(x)\tF^\nI = 
\exp\left\lbrack
x\log\left(\frac{1+\tF}{1-\tF}\right)\right\rbrack.
\ee
Given the generating function \eqref{eqn:genFQkDef}, which we will prove below, we have
\be\label{eqn:PolyOrthML}
\PolyOrth_\kI(x)=\sum_{\jI=0}^\kI\PolyML_{2\jI}(\sqrt{x}).
\ee
Most of the properties of $\seqPolyOrth$ that are described below can be derived as direct consequences of known properties of the Mittag-Leffler polynomials. (See, e.g., \cite{Lomont:2001} for a comprehensive discussion.) Nevertheless, for completeness, we derive them explicitly from the generating function $\genFQk$.

\begin{proposition}
The polynomials $\seqPolyOrth$ satisfy the following:
\begin{enumerate}
\item
Normalization:
\be\label{eqn:NormPk}
\innerP{\PolyOrth_\kI}{\PolyOrth_\kI}=\tfrac{1}{4}(2\kI+1)
\ee
\item
Leading term:
\be\label{eqn:LeadingTerm}
\PolyOrth_\kI = \frac{(4x)^\kI}{(2\kI)!}+\text{(polynomial of degree $\le\kI-1$).}
\ee
\item
Recursion relation:
\be\label{eqn:xPk}
\left(2\kI+1+\frac{2}{2\kI+1}x\right)\PolyOrth_\kI=(\kI+1)\PolyOrth_{\kI+1}+\kI\PolyOrth_{\kI-1}
\qquad
\text{for $\kI\ge 0$, with $\PolyOrth_{-1}\eqdef 0$.}
\ee

\item
The generating function for the series $\seqPolyOrth$ is given by \eqref{eqn:genFQkDef}.

\end{enumerate}
\end{proposition}

\begin{proof}
We will first show that the coefficients $\PolyOrth_\kI(x)$ in \eqref{eqn:genFQkDef} are the same as $\PolyOrth_\kI(x)$ defined in \defref{defn:defQk}. Let $\PolyOrth_\kI(x)$ be the polynomial in \eqref{eqn:genFQkDef}, i.e., the coefficient of $\tF^\kI$ in the expansion of
$$
\frac{1}{(1-\tF)}\cosh\left\lbrack
\sqrt{x}\log\left(\frac{1+\sqrt{\tF}}{1-\sqrt{\tF}}\right)
\right\rbrack.
$$
It is clear that $\deg\PolyOrth_\kI\le\kI$ and that $\PolyOrth_\kI(0)=1$.
To check that $\innerP{\PolyOrth_\kI}{\PolyOrth_\mI}=\frac{1}{4}(2\kI+1)\delta_{\kI\mI}$, we define
$$
\uF_\jI\eqdef\log\left(\frac{1+\sqrt{\tF_\jI}}{1-\sqrt{\tF_\jI}}\right),\qquad
\text{for $\jI=1,2$},
$$
so that $\cosh\uF_\jI=(1+\tF_\jI)/(1-\tF_\jI)$. We then use \eqref{eqn:innerPcoshcosh} to compute
\bear
\lefteqn{
\innerP{\genFQk(x,\tF_1)}{\genFQk(x,\tF_2)} =
\frac{\innerP{\cosh(\sqrt{x}\uF_1)}{\cosh(\sqrt{x}\uF_2)}}{(1-\tF_1)(1-\btF_2)}
}\nn\\
&=&
\frac{1+\cosh\uF_1\cosh\buF_2}{2(1-\tF_1)(1-\btF_2)(\cosh\uF_1+\cosh\buF_2)^2}
=\frac{1+\tF_1\btF_2}{4(1-\tF_1\btF_2)^2}
=\sum_{\kI=0}^\infty\left(\frac{2\kI+1}{4}\right)\tF_1^\kI\btF_2^\kI.
\nn
\eear
This completes the proof that the $\PolyOrth_\kI(x)$'s defined in \eqref{eqn:genFQkDef} are the same as those defined in \defref{defn:defQk}, and also proves \eqref{eqn:NormPk}.

To show that \eqref{eqn:genFQkDef} implies \eqref{eqn:xPk}, we multiply \eqref{eqn:xPk} by $(2\kI+1)\tF^\kI$ and sum over $\kI$. We find that the recursion relation is equivalent to the differential equation
\be\label{eqn:genFQkPDE}
\tF(\tF-1)^2\partial_\tF^2\genFQk
+\tfrac{1}{2}(\tF-1)(7\tF-1)\partial_\tF\genFQk
+\tfrac{1}{2}(3\tF-1)\genFQk
-x\genFQk = 0,
\ee
and it is straightforward to check that \eqref{eqn:genFQkDef} satisfies \eqref{eqn:genFQkPDE} with the appropriate boundary conditions at $\tF=0$.
%The leading term \eqref{eqn:LeadingTerm} immediately follows from \eqref{eqn:xPk}, by induction, and so does the assertion $\PolyOrth_k(0)=1$.
Finally, \eqref{eqn:LeadingTerm} follows from the expansion
\bear
\frac{1}{(1-\tF)}\cosh\left\lbrack
\sqrt{x}\log\left(\frac{1+\sqrt{\tF}}{1-\sqrt{\tF}}\right)
\right\rbrack
&=&(1+O(\tF))\cosh\left\lbrack
\sqrt{x}\left(2\sqrt{\tF}+O(\tF)\right)
\right\rbrack
\nn\\
&=&
\sum_{\kI=0}^\infty\frac{(4x)^\kI}{(2\kI)!}\left\lbrack\tF^\kI+O(\tF^{\kI+1})\right\rbrack.
\nn
\eear

\end{proof}

% ----------------------------------------------------------------------
\subsection{The dilatation operator}
\label{subsec:DilatationOperator}

The dilatation operator $\DilOp$ was defined in \eqref{eqn:dilatOp}. We will now compute its matrix elements in the basis of the orthogonal polynomials $\seqPolyOrth$. 

\begin{proposition}
For $k=0,1,2,\dots$, the polynomial $x\PolyOrth_k'(x)$ can be expanded as
\be\label{eqn:DilOpQk}
\DilOp(\PolyOrth_\kI)=
x\PolyOrth_\kI' = 
\kI\PolyOrth_\kI
-\sum_{\mI=1}^{\kI}\frac{(2\kI + 1)}{(2\mI-1)(2\mI+1)}\PolyOrth_{\kI-\mI}
\,.
\ee
\end{proposition}

\begin{proof}
Multiplying \eqref{eqn:DilOpQk} by $\tF^\kI$ and summing over $\kI$, using the definition of the generating function $\genFQk(x,\tF)$ in \eqref{eqn:genFQkDef}, we can rewrite \eqref{eqn:DilOpQk} as

\bear
x\partial_x\genFQk &=&
\tF\partial_\tF\genFQk\left\lbrack1
-\sum_{\mI=1}^\infty\frac{2\tF^\mI}{(2\mI-1)(2\mI+1)}
\right\rbrack
-\genFQk\sum_{\mI=1}^\infty\frac{\tF^\mI}{2\mI-1}
\nn\\
&=&
-\tfrac{1}{2}\sqrt{\tF}\log\left(\frac{1+\sqrt{\tF}}{1-\sqrt{\tF}}\right)
\left\lbrack
\genFQk+(\tF-1)\partial_\tF\genFQk
\right\rbrack\,,
\nn
\eear
and it is straightforward to check, using the explicit expression for $\genFQk$ in \eqref{eqn:genFQkDef}, that this holds.

\end{proof}

% ======================================================================
\section{Proof of the necessary condition in \thmref{thm:main}}\label{sec:necpf}

In this section we will prove the first part of \thmref{thm:main}.
The assertion that $\xb_\kI$, defined in \eqref{eqn:xbkdef} [and more explicitly in \eqref{eqn:xbkdefAlt} below], have differences $\xd_\kI=\xb_{\kI+1}-\xb_\kI$ that grow no faster than $O(\log\kI/\kI)$ will be proved in \propref{prop:xdkIAsymp}, and the recursion relation \eqref{eqn:OpWseq} will be proved in \propref{prop:recrelkAll}. We will actually prove a statement slightly more general than the first part of \thmref{thm:main}: for any $\sA\in\C$, zero of zeta or not, the following matrix equation \eqref{eqn:OpWgen} has a nonzero solution $\left(\xb_\kI\right)_{\kI=0}^\infty$ with $\xb_{\kI+1}-\xb_\kI=O(\log\kI/\kI)$, provided $\Re\sA<2$, and provided $\sA$ is not a negative even integer. The matrix equation is
\be\label{eqn:OpWgen}
\begin{pmatrix}
-\frac{\sA}{2\cdot 1}-\frac{1}{2} & \frac{1}{1\cdot 3} & \frac{1}{3\cdot 5} & \frac{1}{5\cdot 7} & \frac{1}{7\cdot 9} & \cdots \\
&  & & & & \\
0 & -\frac{\sA}{2\cdot 3}-\frac{1}{2} &  \frac{1}{1\cdot 3} & \frac{1}{3\cdot 5} & \frac{1}{5\cdot 7} & \cdots \\
&  & & & & \\
0& 0 & -\frac{\sA}{2\cdot 5}-\frac{1}{2}  &  \frac{1}{1\cdot 3} & \frac{1}{3\cdot 5} & \cdots \\
&  & & & & \\
0& 0 &0  & -\frac{\sA}{2\cdot 7}-\frac{1}{2}  &  \frac{1}{1\cdot 3} & \cdots \\
&  & & & & \\
0& 0 & 0 & 0 & -\frac{\sA}{2\cdot 9}-\frac{1}{2} & \cdots \\
&  & & & & \\
\vdots &\vdots & \vdots &\vdots & \ddots\\
\end{pmatrix}
\begin{pmatrix}
\xb_0 \\ \\ \xb_1 \\ \\ \xb_2 \\ \\ \xb_3 \\ \\ \xb_4 \\ \\ \vdots \\
\end{pmatrix}
=0,
\ee
with $\xb_0=\eta(\sA)$.
Equation \eqref{eqn:OpW} follows when $\xb_0=0$.

Before we delve into the proofs, we will elaborate on the motivation for \thmref{thm:main}.
Let $\sA\in\C$. Then $x^{-(\sA+1)/2}$ is an eigenfunction of the dilatation operator $\DilOp=x\frac{d}{dx}$.
If we extend the definition of $\innerP{\cdot}{\cdot}$ in \eqref{eqn:innerPDef} to include inner products between polynomials and $x^{-(\sA+1)/2}$:
\be\label{eqn:innerPDefz}
\innerP{x^{-(\sA+1)/2}}{\PolyP}
\eqdef\sum_{j=0}^n\baPoly_j\eta(\sA-2j)
\qquad
\text{for a polynomial $\PolyP(x)=\sum_{j=0}^n \aPoly_j x^j$,}
\ee
we are led to define the sequence of complex numbers $\seqxb$ by
\be\label{eqn:xbkdefAlt}
\xb_\kI\eqdef
\innerP{x^{-(\sA+1)/2}}{\PolyOrth_\kI},\qquad\text{for $\kI=0,1,2,\dots$}.
\ee
Equation \eqref{eqn:xbkdefAlt} is the same as \eqref{eqn:xbkdef}, except that \eqref{eqn:xbkdefAlt} also includes the $\kI=0$ term. [Note that the coefficients of $\PolyOrth_\kI$ are real, and hence in this case $\baPoly_j=\aPoly_j$ in \eqref{eqn:innerPDefz}.]
Since $\PolyOrth_0=1$, by definition $\xb_0=\eta(\sA)$, and so $\xb_0=0$ is equivalent to $\sA$ being a zero of the zeta function.

We now proceed with a heuristic derivation of \eqref{eqn:OpWgen} that, as we will later show, nonetheless leads to a correct conclusion. The conclusion will be proved independently of the heuristic arguments, but we present the latter for motivation.
We begin by pretending that $x^{-(\sA+1)/2}$ can be expanded in the basis $\seqPolyOrth$:
$$
x^{-(\sA+1)/2}\eqnotreally\sum_{\kI=0}^\infty
\frac{\innerP{x^{-(\sA+1)/2}}{\PolyOrth_\kI}}{\innerP{\PolyOrth_\kI}{\PolyOrth_\kI}}\PolyOrth_\kI
\eqnotreally\sum_{\kI=0}^\infty\frac{4\innerP{x^{-(\sA+1)/2}}{\PolyOrth_\kI}}{2\kI+1}\PolyOrth_\kI\,.
$$
It is not clear in what sense the series on the RHS converges to the function $x^{-(\sA+1)/2}$, since it is certainly not true pointwise in $x$. (For example, it can be checked from \eqref{eqn:genFQkDef} that $\PolyOrth_\kI(1)=2\kI+1$, but $\sum_{\kI=0}^\infty\xb_\kI$ does not appear to converge, at least empirically.)
Thus, here and below ``$\eqnotreally$'' indicates a heuristic statement brought for motivation only. To be sure, equations with $\eqnotreally$ will not be used in the actual proof.

To proceed, if the expansion of $x^{-(\sA+1)/2}$ were valid, we might hope to use the adjoint operator $\DilOp^\dagger$ of $\DilOp$ to write the eigenvalue equation
\bear
\lefteqn{
-\left(\frac{\sA+1}{2}\right)\xb_\kI =
\innerP{\DilOp(x^{-(\sA+1)/2})}{\PolyOrth_\kI}
\eqnotreally
\innerP{x^{-(\sA+1)/2}}{\DilOp^\dagger(\PolyOrth_\kI)}
}\nn\\
&\eqnotreally&
\sum_{\mI=0}^\infty
\frac{
\innerP{x^{-(\sA+1)/2}}{\PolyOrth_\mI}\innerP{\PolyOrth_\mI}{\DilOp^\dagger(\PolyOrth_\kI)}
}{\innerP{\PolyOrth_\mI}{\PolyOrth_\mI}}
\nn\\
&\eqnotreally&
\sum_{\mI=0}^\infty
\frac{\innerP{x^{-(\sA+1)/2}}{\PolyOrth_\mI}\innerP{\DilOp(\PolyOrth_\mI)}{\PolyOrth_\kI}}{\innerP{\PolyOrth_\mI}{\PolyOrth_\mI}}
=
\sum_{\mI=\kI}^\infty
\frac{4\innerP{\DilOp(\PolyOrth_\mI)}{\PolyOrth_\kI}}{2\mI+1}\xb_\mI\,.
\nn
\eear 
The end result turns out to be correct, and using \eqref{eqn:DilOpQk} it can be rewritten as
\bear
-\left(\frac{\sA+1}{2}\right)\xb_\kI &=&
\sum_{\mI=0}^\infty
\frac{4\innerP{\DilOp(\PolyOrth_\mI)}{\PolyOrth_\kI}}{2\mI+1}\xb_\mI
\nn\\
&=&
\kI\xb_\kI
-\sum_{\mI=\kI+1}^\infty
\frac{2\kI+1}{(2\mI-2\kI-1)(2\mI-2\kI+1)}\xb_\mI
\nn
\eear
which is equivalent to \eqref{eqn:OpWseq}.

We will now prove \eqref{eqn:OpWseq}, as well as the large $\kI$ behavior of being no worse than $\xb_\kI=O([\log\kI]^2)$, using an integral expression for $\xb_\kI$. We first need the following identity.

\begin{lemma}
Let $\sA\in\C$ with $\Re\sA<2$. Then,
\be\label{eqn:innerPint}
\innerP{x^{-(\sA+1)/2}}{\cosh(\uF\sqrt{x})} 
=
\eta(\sA)
+2\pi^{\sA-1}\sin\frac{\pi\sA}{2}
\int_0^\infty
\sin^2\left(\frac{\uF x}{2\pi}\right)
\frac{dx}{x^{\sA}\sinh x}
\ee
as an equality of formal power series in $\C[[\uF]]$.
\end{lemma}

\begin{proof}
Using the integral representation of the zeta function we have for $\Re\sA'<0$:
\bear
\eta(\sA')&=&
\frac{1}{\pi}(2\pi)^{\sA'}(1-2^{1-\sA'})\sin\frac{\pi\sA'}{2}\int_0^\infty\frac{x^{-\sA'}dx}{e^x-1}\,.
\nn
\eear
We use this to expand
\bear
\lefteqn{
\innerP{x^{-(\sA+1)/2}}{\cosh(\uF\sqrt{x})} 
=
\sum_{k=0}^\infty\frac{\uF^{2k}}{(2k)!}\eta(\sA-2k)
}\nn\\
&=&
\eta(\sA)
+\frac{1}{\pi}\sum_{k=1}^\infty\frac{\uF^{2k}}{(2k)!}
(2\pi)^{\sA-2k}(1-2^{2k+1-\sA})(-1)^k\sin\frac{\pi\sA}{2}\int_0^\infty\frac{x^{2k-\sA}dx}{e^x-1}
\nn\\
&=&
\eta(\sA)
-\frac{4\sin\frac{\pi\sA}{2}}{(2\pi)^{1-\sA}}
\int_0^\infty
\sin^2\left(\frac{\uF x}{4\pi}\right)
\frac{x^{-\sA}dx}{e^x-1}
+\frac{4\sin\frac{\pi\sA}{2}}{\pi^{1-\sA}}
\int_0^\infty
\sin^2\left(\frac{\uF x}{2\pi}\right)
\frac{x^{-\sA}dx}{e^x-1}
\nn\\
&=&
\eta(\sA)
+\frac{2\sin\frac{\pi\sA}{2}}{\pi^{1-\sA}}
\int_0^\infty
\sin^2\left(\frac{\uF x}{2\pi}\right)
\frac{x^{-\sA}dx}{\sinh x}\,.
\nn
\eear
%where before the last line we have changed the integration variables $x\rightarrow 2x$ and combined the two integrals with some simple algebra.

\end{proof}
Note that for $\Re\sA<2$ the integral on the RHS of \eqref{eqn:innerPint} converges for $|\Im\uF|<\pi$.

We now define the generating function $\genFb(\tF)\in\C[[\tF]]$ of the series $\seqxb$ [which is in turn defined by \eqref{eqn:xbkdefAlt}] by
\be\label{eqn:genFbDef}
\genFb(\tF)\eqdef\sum_{\kI=0}^\infty\xb_\kI\tF^\kI.
\ee
For future reference, we denote
\be\label{eqn:uMF}
\uMF\eqdef\log\left(\frac{1+\sqrt{\tF}}{1-\sqrt{\tF}}\right).
\ee
We now have the following integral expression for $\genFb$.
\begin{lemma}
For $\Re\sA<2$ the generating function $\genFb$ can be expressed as
\be\label{eqn:intgenFbtF}
\genFb(\tF)=
\frac{1}{1-\tF}
\left\lbrack
\eta(\sA)
+2\pi^{\sA-1}\sin\frac{\pi\sA}{2}
\int_0^\infty
\sin^2\left(\frac{\uMF x}{2\pi}\right)
\frac{x^{-\sA}dx}{\sinh x}
\right\rbrack\,.
\ee
\end{lemma}

\begin{proof}
The expression \eqref{eqn:intgenFbtF} is understood as a formal power series in $\C[[\tF]]$, and as such the integral is convergent for $\Re\sA<2$, since the coefficient of $\tF^k$ in $\sin^2\left(\frac{\uMF x}{2\pi}\right)$ is a polynomial that is divisible by $x^2$, thanks to the relation \eqref{eqn:uMF} between $\uMF$ and $\tF$. 
Using \eqref{eqn:xbkdefAlt} together with the generating function \eqref{eqn:genFQkDef}, the integral expression \eqref{eqn:innerPint}, and the notation \eqref{eqn:uMF}, we find
\bear
\genFb(\tF) &=&
\innerP{x^{-(\sA+1)/2}}{\genFQk(x,\tF)}
=\frac{1}{1-\tF}\innerP{x^{-(\sA+1)/2}}{\cosh\left(\uMF\sqrt{x}\right)}
\nn\\
&=&
\frac{1}{1-\tF}
\left\lbrack
\eta(\sA)
+2\pi^{\sA-1}\sin\frac{\pi\sA}{2}
\int_0^\infty
\sin^2\left(\frac{\uMF x}{2\pi}\right)
\frac{x^{-\sA}dx}{\sinh x}
\right\rbrack\,.
\nn
\eear
\end{proof}

Note that the expression \eqref{eqn:intgenFbtF} can also be regarded as the Taylor series of an analytic function with a radius of convergence $|\tF|<1$ around $\tF=0$, because it is easy to check that $|\tF|<1$ implies $|\Im\uMF|<\frac{\pi}{2}$, and therefore the integrand in \eqref{eqn:intgenFbtF} falls off exponentially fast as $x\rightarrow\infty$. In fact, $\uMF$ is an analytic function of $\tF$ in the range $\C\setminus[1,\infty)$, the real ray $[1,\infty)$ is a cut, and away from the cut $|\Im\uMF|<\pi$. Therefore, the integral \eqref{eqn:intgenFbtF} converges to an analytic function for $\tF\in\C\setminus[1,\infty)$.

We also note that the integral in \eqref{eqn:intgenFbtF} can be expressed in terms of the Hurwitz zeta function
$$
\zeta(\sA,a)=\frac{1}{\Gamma(s)}\int_0^\infty\frac{x^{\sA-1}e^{-a x}dx}{1-e^{-x}}
$$
to get
\bear
\genFb(\tF) &=&
\frac{1}{1-\tF}(2\pi)^{\sA-1}\Gamma(1-\sA)\sin\frac{\pi\sA}{2}
\left\lbrack
\zeta\left(1-\sA,\frac{1}{2}+\frac{i\uMF}{2\pi}\right)
+\zeta\left(1-\sA,\frac{1}{2}-\frac{i\uMF}{2\pi}\right)\right\rbrack\,,
\nn\\ &&
\label{eqn:genFbHurwitzZeta}
\eear
The appearance of a Hurwitz zeta function suggests a connection with the work of \cite{Bender:2016wob}, where the proposed wavefunction was a Hurwitz zeta function, but we will not pursue this connection, and we will not need to use \eqref{eqn:genFbHurwitzZeta} for the rest of this paper.

Next, we examine the growth of $\xd_\kI$ as $\kI\rightarrow\infty$.
In this section we do not need to assume that $\zeta(\sA)=0$ yet.
We will need the following.

\begin{lemma}\label{lemma:smallx}
There exist (positive) constants $\ConstC_1,\ConstC_2$ such that for $\Re\sA<1$ and $0<x<1$, the coefficient of $\tF^\kI$ in 
\be\label{eqn:int0toepsilon}
\sin^2\left(\frac{\uMF x}{2\pi}\right)
\ee
is bounded by
\be\label{eqn:BoundOnCoeffA}
\ConstC_1 x^2\frac{\log\kI}{\kI}
+\ConstC_2 x^3\frac{(\log\kI)^2}{\kI},
\ee
where $\uMF$ is given in \eqref{eqn:uMF}.
\end{lemma}
\begin{proof}

We write
\bear
\sin^2\left(\frac{\uMF x}{2\pi}\right)
&=&
\frac{\uMF}{2\pi}\int_0^x
\sin\left(\frac{\uMF\xvar}{\pi}\right)
d\xvar\,,
\label{eqn:intuMFsin}
\eear
and first evaluate the coefficient of $\tF^\kI$ in
$
\uMF\sin\left(\frac{\uMF\xvar}{\pi}\right)\,.
$
We have
\bear
\uMF &=&
\log\left(\frac{1+\sqrt{\tF}}{1-\sqrt{\tF}}\right)
=2\sum_{\nI=0}^\infty
\frac{\tF^{\nI+\frac{1}{2}}}{2\nI+1}
\nn
\eear
and
\bear
\sin\left(\frac{\uMF\xvar}{\pi}\right) &=&
\frac{1}{2i}\left\lbrack
\left(\frac{1+\sqrt{\tF}}{1-\sqrt{\tF}}\right)^{i\xvar/\pi}
-\left(\frac{1-\sqrt{\tF}}{1+\sqrt{\tF}}\right)^{i\xvar/\pi}
\right\rbrack
\nn\\
&=&
-i\sum_{\nI=0}^\infty
\left\lbrack
\sum_{\jI=0}^{2\nI+1}
(-1)^\jI\binom{-i\xvar/\pi}{\jI}\binom{i\xvar/\pi}{2\nI+1-\jI}
\right\rbrack
\tF^{\nI+\frac{1}{2}}\,,
\label{eqn:sinwxcoeff}
\eear
where
\be\label{eqn:binomxvar}
\binom{i\xvar/\pi}{\jI}=
\left\{\begin{array}{ll}
1 & \text{for $\jI=0$} \\
(-1)^{\jI+1}\frac{i\xvar}{\pi\jI}\prod_{\lI=1}^{\jI-1}\left(1-\frac{i\xvar}{\pi\lI}\right)
& \text{for $\jI>0$} \\
\end{array}\right.
\ee
is the (generalized) binomial coefficient.
Altogether, the coefficient of $\tF^\kI$ in $\uMF\sin\left(\frac{\uMF\xvar}{\pi}\right)$ is given by

$$
-2i\sum_{\nI=0}^{\kI-1}
\left\lbrack\frac{1}{2\kI-2\nI-1}
\sum_{\jI=0}^{2\nI+1}
(-1)^\jI\binom{-i\xvar/\pi}{\jI}\binom{i\xvar/\pi}{2\nI+1-\jI}
\right\rbrack.
$$
From now on we assume that $\xvar<x$.
{}From \eqref{eqn:binomxvar} we get a bound on the binomial coefficient, for $\jI>0$,
\be\label{eqn:boundbinom}
\left|\binom{\pm i\xvar/\pi}{\jI}\right|
=\frac{\xvar}{\jI\pi}\prod_{\lI=1}^{\jI-1}\left(1+\frac{\xvar^2}{\pi^2\lI^2}\right)^{1/2}
<\frac{\xvar}{\jI\pi}\exp\left(\frac{\xvar^2}{12}\right)
<\frac{\xvar}{\jI\pi}e^{x^2/12},
\ee
where we used the inequalities $1+\frac{\xvar^2}{\pi^2\lI^2}\le\exp\left(\frac{\xvar^2}{\pi^2\lI^2}\right)$ and $\sum_1^{\jI-1}\frac{1}{\lI^2}<\frac{\pi^2}{12}$.
With the bound \eqref{eqn:boundbinom} we now calculate,
\bear
\lefteqn{
\left|
\sum_{\jI=0}^{2\nI+1}
(-1)^\jI\binom{-i\xvar/\pi}{\jI}\binom{i\xvar/\pi}{2\nI+1-\jI}
\right|
}\nn\\ &&
<\frac{2\xvar}{(2\nI+1)\pi}e^{\frac{x^2}{12}}
+\frac{\xvar^2}{\pi^2}e^{\frac{x^2}{6}}
\sum_{\jI=1}^{2\nI}\frac{1}{\jI(2\nI+1-\jI)}
\nn\\
&=&
\frac{2\xvar}{(2\nI+1)\pi}e^{\frac{x^2}{12}}
+\frac{\xvar^2}{(2\nI+1)\pi^2}e^{\frac{x^2}{6}}
\sum_{\jI=1}^{2\nI}\left(\frac{1}{\jI}+\frac{1}{2\nI+1-\jI}\right)
\nn\\
&<&
\frac{2\xvar}{(2\nI+1)\pi}e^{\frac{x^2}{12}}
+\frac{2\xvar^2}{(2\nI+1)\pi^2}e^{\frac{x^2}{6}}\log(5\nI),
\label{eqn:coeffsumbinbin}
\eear
where we inserted ``$5$'' in $\log(5\nI)$ for convenience, so that we could write $\sum_{\jI=1}^{2\nI}<\log(5\nI)$.
Since also (for $\kI\ge 2$)
$$
\sum_{\nI=1}^{\kI-1}
\frac{1}{(2\kI-2\nI-1)(2\nI+1)}<\frac{\log\kI}{\kI}
\,,
$$
we have the overall bound
\bear
\lefteqn{
\left|
\sum_{\nI=1}^{\kI-1}
\frac{1}{2\kI-2\nI-1}
\left\lbrack
\sum_{\jI=0}^{2\nI+1}
(-1)^\jI\binom{-i\xvar/\pi}{\jI}\binom{i\xvar/\pi}{2\nI+1-\jI}
\right\rbrack\right|
}\nn\\
&<&
\frac{2\xvar}{\pi}e^{\frac{x^2}{12}}\frac{\log(5\kI)}{\kI}
+\frac{2\xvar^2}{\pi^2}e^{\frac{x^2}{6}}\frac{\log(5\kI)^2}{\kI},
\nn
\eear
and the coefficient of $\tF^\kI$ in \eqref{eqn:intuMFsin} is bounded by
$$
\frac{x^2}{2\pi^2}e^{\frac{x^2}{12}}\left\lbrack\frac{\log(5\kI)}{\kI}\right\rbrack
+\frac{x^3}{3\pi^3}e^{\frac{x^2}{6}}\left\lbrack\frac{\log(5\kI)^2}{\kI}\right\rbrack.
$$
Finally, for $x<1$ we can easily find constants $\ConstC_1$ and $\ConstC_2$ such that the expression above is smaller than \eqref{eqn:BoundOnCoeffA}.

\end{proof}

We can now prove the claim from \thmref{thm:main} about the large $\kI$ behavior of $\xd_\kI=\xb_{\kI+1}-\xb_\kI$ [defined in \eqref{eqn:xdDef}].

\begin{proposition}\label{prop:xdkIAsymp}
For $0<\Re\sA<1$, the difference behaves as $\xd_\kI=O(\frac{1}{\kI}\log\kI)$.
\end{proposition}

\begin{proof}
The generating function of $\xd_\kI$ is
\be\label{eqn:genFdDef}
\genFd(\tF)\eqdef\sum_{\kI=0}^\infty\xd_\kI\tF^\kI
=\frac{1}{\tF}(1-\tF)\genFb(\tF)-\frac{\xb_0}{\tF}\,.
\ee
The $\xb_0/\tF$ term is not important for the large $\kI$ behavior of $\xd_\kI$, and from \eqref{eqn:intgenFbtF} we find:
\be\label{eqn:intgenFb}
\genFd(\tF)+\frac{\xb_0}{\tF} =
\eta(\sA)
+2\pi^{\sA-1}\sin\left(\frac{\pi\sA}{2}\right)
\int_0^\infty
\sin^2\left(\frac{\uMF x}{2\pi}\right)
\frac{x^{-\sA}dx}{\sinh x}
\,.
\ee

\vskip 12pt
\begin{figure}[t]
\begin{picture}(400,230)
\put(150,100){\begin{picture}(0,0)

\put(102,50){$\ContourC_1$}

\put(150,110){\framebox{the $\sqt$-plane}}
\thicklines

\put(50,0){\begin{picture}(0,0)
\put(0,0){\circle*{3}}

\qbezier(10,3)(8,9)(2,10)
\qbezier(2,10)(-5,11)(-8,5)
\qbezier(-8,5)(-11,0)(-8,-5)
\qbezier(-8,-5)(-5,-11)(2,-10)
\qbezier(2,-10)(8,-9)(10,-3)

\put(10,3){\vector(1,0){25}}
\put(35,3){\line(1,0){25}}
\put(60,-3){\vector(-1,0){25}}
\put(35,-3){\line(-1,0){25}}

\thinlines
%\multiput(0,0)(10,0){12}{\line(1,0){6}}
\put(0,0){\line(1,0){120}}
\put(100,2){cut}
\thicklines

% radius
\thinlines
\put(0,0){\vector(0,1){10}}
\put(-2,11){$\epsilon'$}

%\put(10,-15){\vector(0,1){12}}

\end{picture}}

\thicklines

\put(-50,0){\begin{picture}(0,0)
\put(0,0){\circle*{3}}
\qbezier(-10,3)(-8,9)(-2,10)
\qbezier(-2,10)(5,11)(8,5)
\qbezier(8,5)(11,0)(8,-5)
\qbezier(8,-5)(5,-11)(-2,-10)
\qbezier(-2,-10)(-8,-9)(-10,-3)

\put(-10,-3){\vector(-1,0){25}}
\put(-35,-3){\line(-1,0){25}}
\put(-60,3){\vector(1,0){25}}
\put(-35,3){\line(1,0){25}}

\thinlines
\put(0,0){\line(-1,0){120}}
\thicklines
\put(-100,2){cut}
\end{picture}}

\qbezier(110,3)(108,65)(54,96)
\qbezier(54,96)(0,126)(-54,96)
\qbezier(-54,96)(-108,65)(-110,3)

\qbezier(110,-3)(108,-65)(54,-96)
\qbezier(54,-96)(0,-126)(-54,-96)
\qbezier(-54,-96)(-108,-65)(-110,-3)

% radiues
\put(0,0){\circle*{2}}
\thinlines
\put(0,0){\vector(4,3){89}}
\put(40,40){$R$}

\end{picture}}
\end{picture}
\caption{The contour $\ContourC_1$ used in \eqref{eqn:doubleintCircle}. The small circles around $\pm1$ are of radius $\epsilon'<1$ and the large circle is of radius $R>1$. In the limit $R\rightarrow\infty$ $\ContourC_1$ becomes two Hankel contours.}
\label{fig:ContourC1}
\end{figure}
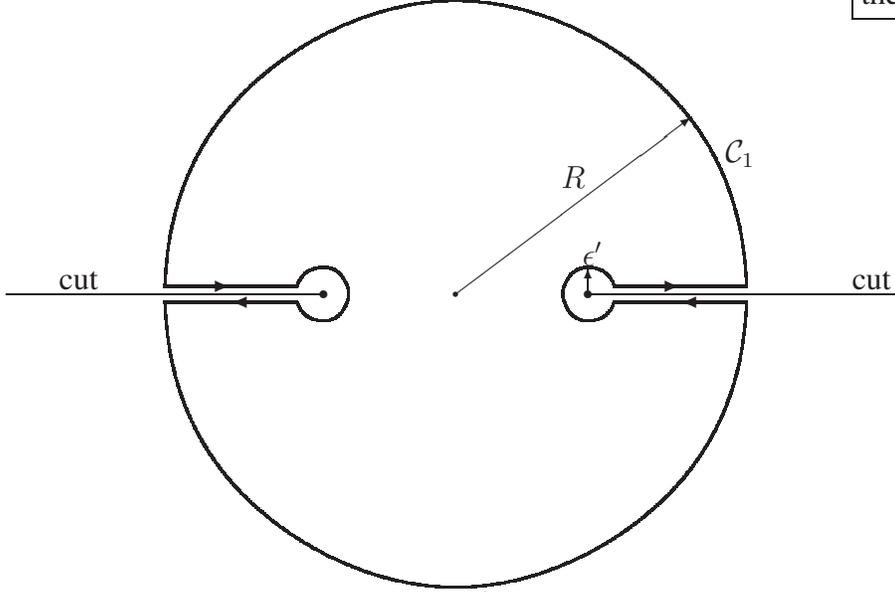

For $\kI\ge 1$, and $\sqt=\sqrt{\tF}$, we consider the integral
\be\label{eqn:doubleintCircle}
\xd_\kI=\frac{1}{2\pi i}\oint_{\ContourC}\genFd(\tF)\frac{d\sqt}{\sqt^{2\kI+1}}
=-i\pi^{\sA-2}\sin\left(\frac{\pi\sA}{2}\right)\oint_{\ContourC}\left\lbrack
\int_0^\infty
\sin^2\left(\frac{\uMF x}{2\pi}\right)
\frac{x^{-\sA}dx}{\sinh x}\right\rbrack\frac{d\sqt}{\sqt^{2\kI+1}}\,,
\ee
where $\ContourC$ is a circle around the origin of radius less than $1$.
The double integral in \eqref{eqn:doubleintCircle} is absolutely convergent, because
\bear
\lefteqn{
\left|\sin^2\left(\frac{\uMF x}{2\pi}\right)\right|
=\frac{1}{4}\left|
2
-\left(\frac{1+\sqt}{1-\sqt}\right)^{\frac{ix}{\pi}}
-\left(\frac{1-\sqt}{1+\sqt}\right)^{\frac{ix}{\pi}}
\right|
}\nn\\ &&
<
\frac{1}{2}
+\cosh\left\lbrack\frac{x}{\pi}\Im\log\left(\frac{1+\sqt}{1-\sqt}\right)\right\rbrack,
\nn
\eear
and along $\ContourC$ the factor $\left|\Im\log\left(\frac{1+\sqt}{1-\sqt}\right)\right|$ is bounded by a number less than $\pi$, since $\Im\log\left(\frac{1+\sqt}{1-\sqt}\right)=\pm\pi$ only on the cut $(-\infty,-1)\cup(1,\infty)$. We can therefore change the order of the integrals and calculate
\be\label{eqn:doubleintC1}
\frac{1}{2\pi i}\oint_{\ContourC}\genFd(\tF)\frac{d\sqt}{\sqt^{2\kI+1}}
=-i\pi^{\sA-2}\sin\left(\frac{\pi\sA}{2}\right)\int_0^\infty\left\lbrack\oint_{\ContourC}
\sin^2\left(\frac{\uMF x}{2\pi}\right)
\frac{d\sqt}{\sqt^{2\kI+1}}\right\rbrack \frac{x^{-\sA}dx}{\sinh x}\,.
\ee
We now need to put a bound on the inner integral $\oint_{\ContourC}
\sin^2\left(\frac{\uMF x}{2\pi}\right)
\frac{d\sqt}{\sqt^{2\kI+1}}$.
The bound, as it turns out, will be $x$-dependent, and in order for the bound to converge when plugged back into the outer integral on the RHS of \eqref{eqn:doubleintC1}, we need to use different techniques for small and large $x$.

For $0<x<\epsilon<1$, we use \lemref{lemma:smallx} to get a bound on the coefficient of $\sqt^{2\kI}$ in $\sin^2\left(\frac{\uMF x}{2\pi}\right)$ as in \eqref{eqn:BoundOnCoeffA}.
It follows that the coefficient of $\tF^\kI=\sqt^{2\kI}$ in
$$
\int_0^\epsilon
\sin^2\left(\frac{\uMF x}{2\pi}\right)
\frac{x^{-\sA}dx}{\sinh x}
$$
is bounded by
\bear
\lefteqn{
\frac{1}{\kI}\int_0^\epsilon
\left\lbrack
\ConstC_1 x^2\log\kI
+\ConstC_2 x^3(\log\kI)^2
\right\rbrack
\frac{x^{-\Re\sA}dx}{\sinh x}
}\nn\\
&<&
\frac{1}{\kI}\int_0^\epsilon
\left\lbrack
\ConstC_1 x^2\log\kI
+\ConstC_2 x^3(\log\kI)^2
\right\rbrack
x^{-\Re\sA-1}dx
\nn\\
&=&
\frac{\epsilon^{2-\Re\sA}\ConstC_1}{(2-\Re\sA)\kI}\log\kI
+\frac{\epsilon^{3-\Re\sA}\ConstC_2}{(3-\Re\sA)\kI}(\log\kI)^2
\,.
\nn
\eear
Since $\Re\sA<1$, we can find constants $\ConstC_3,\ConstC_4$ such that, for $\kI>2$,
\bear
\lefteqn{
\frac{\epsilon^{2-\Re\sA}}{2\pi^2(2-\Re\sA)\kI}e^{\frac{\epsilon^2}{12}}\log(5\kI)
+\frac{\epsilon^{3-\Re\sA}}{3\pi^3(3-\Re\sA)\kI}e^{\frac{\epsilon^2}{6}}\left\lbrack\log(5\kI)\right\rbrack^2
}\nn\\
&&\qquad\qquad\qquad\qquad\qquad\qquad
<\ConstC_3\frac{\epsilon(\log\kI)}{\kI}+\ConstC_4\frac{\epsilon^2(\log\kI)^2}{\kI}.
\nn
\eear
Thus,
\be\label{eqn:boundepsilonint}
\left|
\frac{1}{2\pi i}\int_0^\epsilon
\left\lbrack
\oint_{\ContourC}
\sin^2\left(\frac{\uMF x}{2\pi}\right)
\frac{d\sqt}{\sqt^{2\kI+1}}
\right\rbrack\frac{x^{-\sA}dx}{\sinh x}
\right|
<\ConstC_3\frac{\epsilon(\log\kI)}{\kI}+\ConstC_4\frac{\epsilon^2(\log\kI)^2}{\kI}.
\ee

In the region $\epsilon<x<\infty$ we do not need to worry about $x=0$, but if we extend the arguments of \lemref{lemma:smallx} we will get a bound that behaves as $e^{x^2/6}$, which is not good. We must therefore find a different bound on the coefficient of $\tF^\kI$ in $\sin^2\left(\frac{\uMF x}{2\pi}\right)$.

We define
\be\label{eqn:XikI}
\intXi_\kI(x)\eqdef
\oint_{\ContourC}\sin^2\left(\frac{\uMF x}{2\pi}\right)
\frac{d\sqt}{\sqt^{2\kI+1}}\,.
\ee
Next, we deform the countour $\ContourC$ into the countour $\ContourC_1$ that goes around the cuts and closes along semicircles of radius $R>1$, as depicted in \figref{fig:ContourC1}.
We note that for $\kI\ge 1$, the contribution of the points with $|\sqt|=R$ to the contour integral in \eqref{eqn:XikI} is negligible as $R\rightarrow\infty$, since $\sqt\rightarrow\pm i\pi$ (depending on whether $\Im\sqt$ is positive or negative) and $\sin^2\left(\frac{\uMF x}{2\pi}\right)$ has a finite ($x$-dependent) limit. Thus, we can deform $\ContourC_1$ into two Hankel contours that surround the cuts.
Next, we expand
$$
\sin^2\left(\frac{\uMF x}{2\pi}\right)=
\frac{1}{2}
-\frac{1}{4}\left(\frac{1+\sqt}{1-\sqt}\right)^{\frac{ix}{\pi}}
-\frac{1}{4}\left(\frac{1+\sqt}{1-\sqt}\right)^{-\frac{ix}{\pi}}.
$$
Let $\ContourC_L,\ContourC_R$ be the left and right Hankel contours, so that $\ContourC_L$ surrounds the cut $(-\infty,-1]$ and $\ContourC_R$ surrounds $[0,\infty)$.
With a change of variables $\sqt\rightarrow-\sqt$ we find
$$
\oint_{\ContourC_L}\left(\frac{1\pm\sqt}{1\mp\sqt}\right)^{\frac{ix}{\pi}}
\frac{d\sqt}{\sqt^{2\kI+1}}
=
\oint_{\ContourC_R}\left(\frac{1\mp\sqt}{1\pm\sqt}\right)^{\frac{ix}{\pi}}
\frac{d\sqt}{\sqt^{2\kI+1}},
$$
so for $\kI\ge 1$ we have from \eqref{eqn:XikI}
\bear
\intXi_\kI(x) &=&
-\frac{1}{2}\oint_{\ContourC_R}\left(\frac{1+\sqt}{1-\sqt}\right)^{\frac{ix}{\pi}}
\frac{d\sqt}{\sqt^{2\kI+1}}
-\frac{1}{2}\oint_{\ContourC_R}\left(\frac{1+\sqt}{1-\sqt}\right)^{-\frac{ix}{\pi}}
\frac{d\sqt}{\sqt^{2\kI+1}}\,.
\label{eqn:XikxTwoInts}
\nn\\
&&
\eear
Now, noting that 
$
\Arg\left(\frac{1+\sqt}{1-\sqt}\right)>0
$
when $\sqt$ is in the upper half plane and
$
\Arg\left(\frac{1+\sqt}{1-\sqt}\right)<0
$
when $\sqt$ is in the lower half plane, so that
$$
\left|\left(\frac{1+\sqt}{1-\sqt}\right)^{\frac{ix}{\pi}}\right|
\rightarrow e^{-x}
$$
when approaching the cut $[1,\infty)$ from above,
and
$$
\left|\left(\frac{1+\sqt}{1-\sqt}\right)^{\frac{ix}{\pi}}\right|
\rightarrow e^{x}
$$
when approaching from below, we calculate
\be\label{eqn:intonCR}
\int_{\ContourC_R}
\left(\frac{1+\sqt}{1-\sqt}\right)^{\frac{ix}{\pi}}\frac{d\sqt}{\sqt^{2\kI+1}}
=
\left(e^x-e^{-x}\right)
\int_0^\infty
e^{\frac{i x\rMF}{\pi}}
\left(
\frac{\tanh^{2\kI+1}(\frac{\rMF}{2})}{2\sinh^2(\frac{\rMF}{2})}
\right)d\rMF,
\ee
where along the cut we changed variables to
\be\label{eqn:rMFdef}
\rMF\eqdef\log\left(\frac{\sqt+1}{\sqt-1}\right),
\ee
so that $\sqt=\coth(\frac{\rMF}{2})$.
We also used the fact that as $\epsilon'\rightarrow 0$ the contribution of the circular part of $\ContourC_R$ (around $\sqt=1$) goes to zero.
Now, it is easy to get the bound
\bear
\left|
\int_0^\infty
e^{\frac{i x\rMF}{\pi}}
\left(
\frac{\tanh^{2\kI+1}(\frac{\rMF}{2})}{2\sinh^2(\frac{\rMF}{2})}
\right)d\rMF
\right|
&<&
\int_0^\infty
\frac{\tanh^{2\kI+1}(\frac{\rMF}{2})d\rMF}{2\sinh^2(\frac{\rMF}{2})}
=\frac{1}{2\kI}\,,
\nn
\eear
and we get a similar bound for the second term on the RHS of \eqref{eqn:XikxTwoInts}.
Thus, for $\kI>0$,
$$
\left|
\oint_{\ContourC_1}\sin^2\left(\frac{\uMF x}{2\pi}\right)
\frac{d\sqt}{\sqt^{2\kI+1}}
\right|<\frac{\sinh x}{\kI}\,.
$$
But this bound is not good enough because it is $O(e^x)$ and will lead to a divergent $\int^\infty(\cdots)dx$ on the RHS of \eqref{eqn:doubleintC1}, unless $\Re\sA>1$, which is outside the range of $\sA$ we are interested in.
So, to improve the bound we integrate by parts the integral on the RHS of \eqref{eqn:intonCR}:
\bear
\lefteqn{
\int_0^\infty
e^{\frac{i x\rMF}{\pi}}
\left(
\frac{\tanh^{2\kI+1}(\frac{\rMF}{2})}{2\sinh^2(\frac{\rMF}{2})}
\right)d\rMF
=
\frac{\pi}{i x}
\int_0^\infty
\frac{d}{d\rMF}\left(e^{\frac{i x\rMF}{\pi}}\right)
\left(
\frac{\tanh^{2\kI+1}(\frac{\rMF}{2})}{2\sinh^2(\frac{\rMF}{2})}
\right)d\rMF
}\nn\\
&=&
\frac{\pi}{i x}
\left\lbrack
e^{\frac{i x\rMF}{\pi}}
\left(
\frac{\tanh^{2\kI+1}(\frac{\rMF}{2})}{2\sinh^2(\frac{\rMF}{2})}
\right)\right\rbrack_0^\infty
-\frac{\pi}{i x}
\int_0^\infty
e^{\frac{i x\rMF}{\pi}}
\frac{d}{d\rMF}\left(
\frac{\tanh^{2\kI+1}(\frac{\rMF}{2})}{2\sinh^2(\frac{\rMF}{2})}
\right)d\rMF
\nn\\
&=&
-\frac{\pi}{i x}
\int_0^\infty
e^{\frac{i x\rMF}{\pi}}
\left(\frac{2\kI+1}{4\cosh^2(\frac{\rMF}{2})}-\frac{1}{2}\right)\frac{\tanh^{2\kI}(\frac{\rMF}{2})}{\sinh^2(\frac{\rMF}{2})}
d\rMF\,,
\nn
\eear
for $\kI\ge 1$.
Therefore,
\bear
\lefteqn{
\left|\int_0^\infty
e^{\frac{i x\rMF}{\pi}}
\left(
\frac{\tanh^{2\kI+1}(\frac{\rMF}{2})}{2\sinh^2(\frac{\rMF}{2})}
\right)d\rMF\right|
}\nn\\ &&
<
\frac{\pi}{x}
\int_0^\infty
\left(\frac{2\kI+1}{4\cosh^2(\frac{\rMF}{2})}+\frac{1}{2}\right)\frac{\tanh^{2\kI}(\frac{\rMF}{2})}{\sinh^2(\frac{\rMF}{2})}
d\rMF
=\frac{2\pi}{(2\kI-1)x}\,.
\nn\\
&&
\label{eqn:boundintexpixwpi}
\eear
Substituting that back into \eqref{eqn:intonCR}, we get the bound
\be\label{eqn:intonCHnewBound}
\left|\int_{\ContourC_R}
\left(\frac{1+\sqt}{1-\sqt}\right)^{\pm\frac{ix}{\pi}}\frac{d\sqt}{\sqt^{2\kI+1}}
\right|
<
\frac{2\pi\sinh x}{(2\kI-1)x}\,,
\ee
and then from \eqref{eqn:XikxTwoInts} we get
\be\label{eqn:boundXikIImproved}
|\intXi_\kI(x)| < \frac{2\pi\sinh x}{(2\kI-1)x}\,.
\ee
Then, plugging \eqref{eqn:boundXikIImproved} and the bound \eqref{eqn:boundepsilonint} into the integral on the RHS of \eqref{eqn:doubleintC1} we get, for $\Re\sA>0$,
\bear
\label{eqn:doubleintBound}
\lefteqn{
\left|\frac{1}{2\pi i}\int_0^\infty\left\lbrack\oint_{\ContourC}
\sin^2\left(\frac{\uMF x}{2\pi}\right)
\frac{d\sqt}{\sqt^{2\kI+1}}\right\rbrack \frac{x^{-\sA}dx}{\sinh x}\right|
}\nn\\
&& \le
\left|\frac{1}{2\pi i}\int_0^\epsilon\left\lbrack\oint_{\ContourC}
\sin^2\left(\frac{\uMF x}{2\pi}\right)
\frac{d\sqt}{\sqt^{2\kI+1}}\right\rbrack \frac{x^{-\sA}dx}{\sinh x}\right|
\nn\\ &&\qquad
+\left|
\frac{1}{2\pi i}\int_\epsilon^\infty\left\lbrack\oint_{\ContourC}
\sin^2\left(\frac{\uMF x}{2\pi}\right)
\frac{d\sqt}{\sqt^{2\kI+1}}\right\rbrack \frac{x^{-\sA}dx}{\sinh x}\right|
\nn\\
&<&
\ConstC_3\frac{\epsilon(\log\kI)}{\kI}+\ConstC_4\frac{\epsilon^2(\log\kI)^2}{\kI}
+
\frac{1}{(2\kI-1)}
\int_\epsilon^\infty
x^{-\Re\sA-1}dx
\nn\\
&<&
\ConstC_3\frac{\epsilon(\log\kI)}{\kI}+\ConstC_4\frac{\epsilon^2(\log\kI)^2}{\kI}
+\frac{1}{(2\kI-1)\epsilon\Re\sA}\,.
\nn
\eear
Taking $\epsilon=1/\log\kI$, we get
\be\label{eqn:doubleintBoundB}
\left|\frac{1}{2\pi i}\int_0^\infty\left\lbrack\oint_{\ContourC}
\sin^2\left(\frac{\uMF x}{2\pi}\right)
\frac{d\sqt}{\sqt^{2\kI+1}}\right\rbrack \frac{x^{-\sA}dx}{\sinh x}\right|
<\ConstC_5(\sA)\frac{\log\kI}{\kI}\,,
\ee
for some ($\sA$-dependent) constant $\ConstC_5(\sA)$.
Plugging \eqref{eqn:doubleintBoundB} into \eqref{eqn:doubleintC1} we find
$$
|\xd_\kI|=
\left|
\frac{1}{2\pi i}\oint_{\ContourC}\genFd(\tF)\frac{d\sqt}{\sqt^{2\kI+1}}
\right|=O(\frac{\log\kI}{\kI}).
$$
Note that we can get better bound than $O(\log\kI/\kI)$ by taking $\epsilon=(\log\kI)^{-2/3}$. The bound is then $|\xd_\kI|=O([\log\kI]^{2/3}/\kI)$, but for our purposes the bound in \thmref{thm:main} is sufficient.

\end{proof}
Now that we have established that $\xd_\kI=O(\log\kI/\kI)$, it follows that
\be\label{eqn:bkIlog}
\xb_\kI=\sum_{\nI=0}^{\kI-1}\xd_\kI = O([\log\kI]^2),
\ee
and therefore the sum $\sum_{\nI=0}^\infty\frac{\xb_{\kI+\nI}}{(2\nI+1)(2\nI+3)}$ on the LHS of \eqref{eqn:OpWseq} converges. We will now prove \eqref{eqn:OpWseq}. We will need the following simple result.

\begin{lemma}
\label{lem:xcothxint}
For $x\in\R$,
\be\label{eqn:intsinSquared}
\int_{-1}^1
\left\lbrack
2
-\left(\frac{1+\sqt}{1-\sqt}\right)^{\frac{ix}{\pi}}
-\left(\frac{1-\sqt}{1+\sqt}\right)^{\frac{ix}{\pi}}
\right\rbrack\frac{d\sqt}{\sqt^2}
=4(x\coth x -1).
\ee
\end{lemma}

\begin{proof}

\vskip 12pt
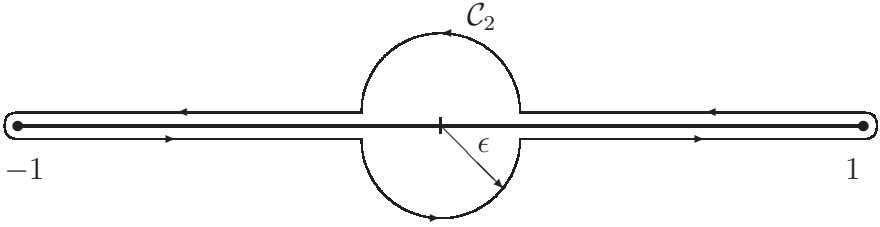
\begin{figure}[t]
\begin{picture}(400,100)
\put(150,50){\begin{picture}(0,0)

\put(10,38){$\ContourC_2$}
\thicklines
\put(160,0){\circle*{4}}
\put(-160,0){\circle*{4}}
\put(-165,-20){$-1$}
\put(153,-20){$1$}
\put(-160,0){\line(1,0){320}}
\put(0,-3){\line(0,1){6}}

\thinlines

% semicircle around 1 with qbeziers
%\qbezier(200,-10)(203,-10)(205,-9)
%\qbezier(205,-9)(207,-7)(209,-5)
%\qbezier(209,-5)(210,-3)(210,0)
%\qbezier(210,0)(210,3)(209,5)
%\qbezier(209,5)(207,7)(205,9)
%\qbezier(205,9)(203,10)(200,10)

% qbezier around 1
\qbezier(160,5)(165,5)(165,0)
\qbezier(160,-5)(165,-5)(165,0)

% line
\put(160,5){\vector(-1,0){60}}
\put(100,5){\line(-1,0){70}}

% semicircle around origin with qbezier
\qbezier(30,5)(30,13)(26,20)
\qbezier(26,20)(22,27)(15,31)
\qbezier(15,31)(8,35)(0,35)
\put(0,35){\vector(-1,0){0}}
\qbezier(0,35)(-8,35)(-15,31)
\qbezier(-15,31)(-22,27)(-26,20)
\qbezier(-26,20)(-30,13)(-30,5)

% line
\put(-30,5){\vector(-1,0){70}}
\put(-100,5){\line(-1,0){60}}

% semicircle around -1 with qbeziers
%\qbezier(-200,-10)(-203,-10)(-205,-9)
%\qbezier(-205,-9)(-207,-7)(-209,-5)
%\qbezier(-209,-5)(-210,-3)(-210,0)
%\qbezier(-210,0)(-210,3)(-209,5)
%\qbezier(-209,5)(-207,7)(-205,9)
%\qbezier(-205,9)(-203,10)(-200,10)

% qbezier around -1
\qbezier(-160,5)(-165,5)(-165,0)
\qbezier(-160,-5)(-165,-5)(-165,0)

% line
\put(-160,-5){\vector(1,0){60}}
\put(-100,-5){\line(1,0){70}}

% semicircle around origin with qbezier
\qbezier(30,-5)(30,-13)(26,-20)
\qbezier(26,-20)(22,-27)(15,-31)
\qbezier(15,-31)(8,-35)(0,-35)
\put(0,-35){\vector(1,0){0}}
\qbezier(0,-35)(-8,-35)(-15,-31)
\qbezier(-15,-31)(-22,-27)(-26,-20)
\qbezier(-26,-20)(-30,-13)(-30,-5)

\put(0,0){\vector(1,-1){24}}
\put(14,-10){$\epsilon$}

% line
\put(30,-5){\vector(1,0){70}}
\put(100,-5){\line(1,0){60}}

\end{picture}}
\end{picture}
\caption{The contour $\ContourC_2$ goes clockwise around the cut $[-1,1]$ of the function $(1+\zC)^{i x/\pi}(1-\zC)^{-i x/\pi}/\zC^2$. In the range $(-1,-\epsilon]\cup[\epsilon,1)$ the contour $\ContourC_2$ passes infinitesimally close to the real axis above and below the cut. In the range $(-\epsilon,\epsilon)$ of the cut, the contour deviates from the cut and forms two semicircles of radius $\epsilon$.}
\label{fig:ContourC2}
\end{figure}

For $\zC\in\C\setminus[-1,1]$, consider the analytic function
\be\label{eqn:func1def}
\func_1(\zC)\eqdef\frac{1}{\zC^2}\left(\frac{1+\zC}{1-\zC}\right)^{\frac{ix}{\pi}}\,,
\ee
defined so that for $-1<\Re\zC<1$ the limiting values of $\func_1(\zC)$ when $\zC$ approaches the cut $(-1,1)$ from above and below are given by
$$
\lim_{\Im\zC\rightarrow 0^{+}}\func_1(\zC) = 
\frac{1}{(\Re\zC)^2}\left(\frac{1+\Re\zC}{1-\Re\zC}\right)^{\frac{i x}{\pi}}
\,,\qquad
\lim_{\Im\zC\rightarrow 0^{-}}\func_1(\zC) = 
\frac{1}{(\Re\zC)^2}\left(\frac{1+\Re\zC}{1-\Re\zC}\right)^{\frac{i x}{\pi}}
e^{-2x}\,.
$$
Thus, crossing the cut $[-1,1]$ from the upper half $\zC$-plane to the lower half plane, the function $\func_1(\zC)$ gets multiplied by $e^{-2x}$.

Now let $0<\epsilon<1$ and consider the contour integral $\oint\func_1(\zC)d\zC$ around the contour $\ContourC_2$ depicted in \figref{fig:ContourC2}. 
The contribution to the integral from the four horizontal segments along $(-1,-\epsilon]\cup[\epsilon,1)$, above and below the cut, is:
$$
\left(e^{-2x}-1\right)
\left\lbrack
\int_{-1}^{-\epsilon}\left(\frac{1+\sqt}{1-\sqt}\right)^{\frac{ix}{\pi}}\frac{d\sqt}{\sqt^2}
+\int_{\epsilon}^{1}\left(\frac{1+\sqt}{1-\sqt}\right)^{\frac{ix}{\pi}}\frac{d\sqt}{\sqt^2}
\right\rbrack.
$$
The contribution from the two semicircles of radius $\epsilon$ (around the origin) is
$$
-2x+\frac{2}{\epsilon}
+e^{-2x}\left(-2x-\frac{2}{\epsilon}\right)
+O(\epsilon),
$$
and the integral $\int\func_1(\zC)d\zC$ around the singularities $\pm 1$ is negligible.
On the other hand, the contour $\ContourC_2$ can be deformed to infinity, and since the function $\func_1(\zC)$ is analytic near $\zC=\infty$ the contour integral vanishes.
Thus, we find
\bear
0 &=&
\oint_{\ContourC_2}\func_1(\zC)d\zC
=
\left(e^{-2x}-1\right)
\left\lbrack
\int_{-1}^{-\epsilon}\left(\frac{1+\sqt}{1-\sqt}\right)^{\frac{ix}{\pi}}\frac{d\sqt}{\sqt^2}
+\int_{\epsilon}^{1}\left(\frac{1+\sqt}{1-\sqt}\right)^{\frac{ix}{\pi}}\frac{d\sqt}{\sqt^2}
\right\rbrack
\nn\\
&&
-2x+\frac{2}{\epsilon}
+e^{-2x}\left(
-2x-\frac{2}{\epsilon}
\right)+O(\epsilon),
\nn
\eear
So,
\be\label{eqn:intCupper}
\int_{-1}^{-\epsilon}\left(\frac{1+\sqt}{1-\sqt}\right)^{\frac{ix}{\pi}}\frac{d\sqt}{\sqt^2}
+\int_{\epsilon}^{1}\left(\frac{1+\sqt}{1-\sqt}\right)^{\frac{ix}{\pi}}\frac{d\sqt}{\sqt^2}
=
-2x\coth x+\frac{2}{\epsilon}
+O(\epsilon).
\ee
Similarly,
\be\label{eqn:intClower}
\int_{-1}^{-\epsilon}\left(\frac{1+\sqt}{1-\sqt}\right)^{-\frac{ix}{\pi}}\frac{d\sqt}{\sqt^2}
+\int_{\epsilon}^{1}\left(\frac{1+\sqt}{1-\sqt}\right)^{-\frac{ix}{\pi}}\frac{d\sqt}{\sqt^2}
=
-2x\coth x+\frac{2}{\epsilon}
+O(\epsilon).
\ee
Combining \eqref{eqn:intCupper} and \eqref{eqn:intClower}, together with
$
-2\int_{-1}^{-\epsilon}\frac{d\sqt}{\sqt^2}
-2\int_{\epsilon}^{1}\frac{d\sqt}{\sqt^2}
=
4-\frac{4}{\epsilon}
$,
and taking the limit $\epsilon\rightarrow 0$, we find \eqref{eqn:intsinSquared}.
\end{proof}

Now we can prove that the $\kI=0$ case of \eqref{eqn:OpWseq} holds.
We will prove a more general result for $\sA$ not necessarily a zero of the zeta function.
\begin{proposition}\label{prop:recrelkz}
For $\Re\sA<2$, the numbers $\seqxbz$ defined in \eqref{eqn:xbkdefAlt} satisfy the relation 
\be\label{eqn:xbkdeZero}
-\frac{1}{2}\left(\sA+1\right)\xb_{0}
+\sum_{\nI=0}^\infty\frac{\xb_{\nI+1}}{(2\nI+1)(2\nI+3)}=0.
\ee
\end{proposition}
\begin{proof}
Since $|\xb_\kI|=O([\log\kI]^2)$, the generating function \eqref{eqn:genFbDef} can be viewed as a Taylor series of an analytic function that converges in the disk $|\tF|<1$. We denote this analytic function also by $\genFb(\tF)$ and use it to evaluate the sum as follows.
\bear
\lefteqn{
\sum_{\nI=0}^\infty\frac{\xb_{\nI+1}}{(2\nI+1)(2\nI+3)} 
}\nn\\
&=&
\frac{1}{2}\int_0^1
(1-\sqt^2)\sum_{\nI=0}^\infty\xb_{\nI+1}\sqt^{2\nI}
d\sqt
=
\frac{1}{2}\int_0^1
(1-\sqt^2)\left(\genFb(\sqt^2)-\xb_0\right)\frac{d\sqt}{\sqt^2}\,.
\label{eqn:intrepsum}
\nn\\
&&
\eear
[The order of integration and sum can be exchanged since $|\xb_\kI|=O([\log\kI]^2)$.]

Since $\xb_0=\eta(\sA)$, by definition, we can use \eqref{eqn:intgenFbtF} to write
$$
(1-\sqt^2)(\genFb(\sqt^2)-\xb_0)=
\sqt^2\xb_0+
2\pi^{\sA-1}\sin\frac{\pi\sA}{2}
\int_0^\infty
\sin^2\left(\frac{\rMF x}{2\pi}\right)
\frac{x^{-\sA}dx}{\sinh x}\,,
$$
where $\rMF$ was defined in \eqref{eqn:rMFdef}.
So we have
$$
\sum_{\nI=0}^\infty\frac{\xb_{\nI+1}}{(2\nI+1)(2\nI+3)}
=
\frac{1}{2}\xb_0+
\pi^{\sA-1}\sin\frac{\pi\sA}{2}
\int_0^\infty
\left\lbrack\int_0^1\sin^2\left(\frac{\rMF x}{2\pi}\right)\frac{d\sqt}{\sqt^2}\right\rbrack
\frac{x^{-\sA}dx}{\sinh x}\,,
$$
and the order of integrations can be exchanged, since the double integral converges absolutely.
Using \eqref{eqn:intsinSquared}, we calculate
\bear
\int_0^1\sin^2\left(\frac{\rMF x}{2\pi}\right)\frac{d\sqt}{\sqt^2}
&=&
\frac{1}{8}\int_{-1}^1
\left\lbrack
2
-\left(\frac{1+\sqt}{1-\sqt}\right)^{\frac{ix}{\pi}}
-\left(\frac{1-\sqt}{1+\sqt}\right)^{\frac{ix}{\pi}}
\right\rbrack\frac{d\sqt}{\sqt^2}
\nn\\
=
\frac{1}{2}(x\coth x-1),
\nn\\
&&\label{eqn:IntSinSquareSqt2}
\eear
and therefore
\bear
-\frac{1}{2}\xb_0
+\sum_{\nI=0}^\infty\frac{\xb_{\nI+1}}{(2\nI+1)(2\nI+3)}
&=&
\frac{1}{2}\pi^{\sA-1}\sin\frac{\pi\sA}{2}
\int_0^\infty
\left(
x\coth x-1
\right)
\frac{x^{-\sA}dx}{\sinh x}\,.
\nn\\ &&
\label{eqn:sumbkzandintxcothx}
\eear
The last integral converges for $\Re\sA<2$, and for $\Re\sA<0$ we can easily see using integration by parts that it evaluates to
\be\label{eqn:intwithcoth}
\int_0^\infty
\left(
x\coth x-1
\right)
\frac{x^{-\sA}dx}{\sinh x}
=
-\sA\int_0^\infty\frac{x^{-\sA}dx}{\sinh x}
=
\frac{\pi^{1-\sA}\sA\eta(\sA)}{\sin\left(\frac{\pi\sA}{2}\right)},
\ee
where we have used \eqref{eqn:intxsAsinh} for the integral of $x^{-\sA}/\sinh x$.
It follows that
\be\label{eqn:sinintxcothx}
\frac{1}{2}\pi^{\sA-1}\sin\frac{\pi\sA}{2}
\int_0^\infty
\left(
x\coth x-1
\right)
\frac{x^{-\sA}dx}{\sinh x}
=\frac{1}{2}\sA\eta(\sA).
\ee
Since both sides of \eqref{eqn:sinintxcothx} are analytic functions that are well defined in the entire domain $\Re\sA<2$, equation \eqref{eqn:sinintxcothx} holds for $\Re\sA<2$.
Substituting the integral \eqref{eqn:sinintxcothx} back into the sum \eqref{eqn:sumbkzandintxcothx}, and setting $\xb_0=\eta(s)$, we find
$$
\sum_{\nI=0}^\infty\frac{\xb_{\nI+1}}{(2\nI+1)(2\nI+3)}
=\frac{1}{2}\xb_0+\frac{1}{2}\sA\eta(\sA)=\frac{1}{2}(\sA+1)\eta(\sA),
$$
which proves \eqref{eqn:xbkdeZero}.

\end{proof}

In \propref{prop:recrelkz} we proved the $\kI=0$ case of \eqref{eqn:OpWseq}. To extend the proof to $\kI>0$ we need the following.
\begin{lemma}
\label{lem:intsinML}
Let $\left(\PolyML_\kI\right)_{\kI=0}^\infty$ be the Mittag-Leffler polynomials, defined in \eqref{eqn:genFMLk}. For $x\in\R$,
\bear
\lefteqn{
\int_0^1
\left\lbrack
2
-\left(\frac{1+\sqt}{1-\sqt}\right)^{\frac{ix}{\pi}}
-\left(\frac{1-\sqt}{1+\sqt}\right)^{\frac{ix}{\pi}}
+2\sum_{\nI=1}^{\kI-1}\PolyML_{2\nI}\left(\tfrac{i x}{\pi}\right)\sqt^{2\nI}
\right\rbrack\frac{d\sqt}{\sqt^{2\kI}}
}\nn\\
&=&
2\sum_{\nI=0}^{\kI-1}\frac{\PolyML_{2\nI}\left(\tfrac{i x}{\pi}\right)}{2\nI-2\kI+1}
-i\pi\PolyML_{2\kI-1}\left(\tfrac{i x}{\pi}\right)\coth x.
\nn\\ &&
\label{eqn:intsinSquaredML}
\eear
\end{lemma}

\begin{proof}
Note that the integrand on the LHS of \eqref{eqn:intsinSquared} is regular at $\sqt=0$ since
$$
\left(\frac{1+\sqt}{1-\sqt}\right)^{\frac{ix}{\pi}}
+\left(\frac{1-\sqt}{1+\sqt}\right)^{\frac{ix}{\pi}}
=2+2\sum_{\nI=1}^{\kI-1}\PolyML_{2\nI}\left(\tfrac{i x}{\pi}\right)\sqt^{2\nI}
+O(\sqt^{2\kI}).
$$

Similarly to \eqref{eqn:func1def}, we define
\be\label{eqn:funckIdef}
\func_\kI(\zC)\eqdef\frac{1}{\zC^{2\kI}}\left(\frac{1+\zC}{1-\zC}\right)^{\frac{ix}{\pi}}\,,
\ee
defined with a cut along $[-1,1]$, so that for $-1<\Re\zC<1$ the limiting values of $\func_\kI(\zC)$ when $\zC$ approaches the cut $(-1,1)$ from above and below are given by
$$
\lim_{\Im\zC\rightarrow 0^{+}}\func_\kI(\zC) = 
\frac{1}{(\Re\zC)^{2\kI}}\left(\frac{1+\Re\zC}{1-\Re\zC}\right)^{\frac{i x}{\pi}}
\,,\qquad
\lim_{\Im\zC\rightarrow 0^{-}}\func_\kI(\zC) = 
\frac{1}{(\Re\zC)^{2\kI}}\left(\frac{1+\Re\zC}{1-\Re\zC}\right)^{\frac{i x}{\pi}}
e^{-2x}\,.
$$
As in the proof of \lemref{lem:xcothxint}, we can then calculate the integral of \eqref{eqn:intsinSquaredML} using $\oint_{\ContourC_2}\func_\kI(\zC)d\zC$. Technically, we need the contribution to the contour integral from the two semicircles of radius $\epsilon$ around the origin (see \figref{fig:ContourC2}).

Using \eqref{eqn:genFMLk} we expand
$$
\left(\frac{1+\zC}{1-\zC}\right)^{\frac{ix}{\pi}} =
\sum_{\nI=0}^{2\kI-1}\PolyML_\nI\left(\tfrac{i x}{\pi}\right)\zC^\nI+O(\zC^{2\kI}),
$$
in terms of the Mittag-Leffler polynomials $\PolyML_\nI$. The contribution of the two semicircles is therefore
\bear
\lefteqn{
\int_{\text{semicircles}}\func_\kI(\zC)d\zC
=\int_{\text{semicircles}}\sum_{\nI=0}^{2\kI-1}\PolyML_\nI\left(\tfrac{i x}{\pi}\right)\zC^{\nI-2\kI}d\zC + O(\epsilon)
}\nn\\
&=&
i\pi\PolyML_{2\kI-1}\left(\tfrac{i x}{\pi}\right)\left(1+e^{-2x}\right)
+\sum_{\nI=0}^{2\kI-2}\PolyML_\nI\left(\tfrac{i x}{\pi}\right)\frac{\epsilon^{\nI-2\kI+1}}{\nI-2\kI+1}\left\lbrack(-1)^{\nI+1}-1\right\rbrack
\nn\\ &&
+\sum_{\nI=0}^{2\kI-2}\PolyML_\nI\left(\tfrac{i x}{\pi}\right)\frac{\epsilon^{\nI-2\kI+1}}{\nI-2\kI+1}
\left\lbrack 1-(-1)^{\nI+1}\right\rbrack e^{-2x}
 + O(\epsilon)
\nn\\
&=&
i\pi\left(1+e^{-2x}\right)\PolyML_{2\kI-1}\left(\tfrac{i x}{\pi}\right)
+2\left(e^{-2x}-1\right)\sum_{\nI=0}^{\kI-1}\PolyML_{2\nI}\left(\tfrac{i x}{\pi}\right)\frac{\epsilon^{2\nI-2\kI+1}}{2\nI-2\kI+1}
 + O(\epsilon).
\nn
\eear
Following an argument similar to the one leading to \eqref{eqn:intCupper}-\eqref{eqn:intClower}, and noting that $\PolyML_{2\nI}$ is an even function while $\PolyML_{2\kI-1}$ is odd, we now expand $0=\oint_{\ContourC_2}\func_\kI(\zC)d\zC$ to get
\bear
\lefteqn{
\int_{-1}^{-\epsilon}\left(\frac{1+\sqt}{1-\sqt}\right)^{\frac{ix}{\pi}}\frac{d\sqt}{\sqt^{2\kI}}
+\int_{\epsilon}^{1}\left(\frac{1+\sqt}{1-\sqt}\right)^{\frac{ix}{\pi}}\frac{d\sqt}{\sqt^{2\kI}}
}
\nn\\
&=&
i\pi\PolyML_{2\kI-1}\left(\tfrac{i x}{\pi}\right)\coth x
-2\sum_{\nI=0}^{\kI-1}\PolyML_{2\nI}\left(\tfrac{i x}{\pi}\right)\frac{\epsilon^{2\nI-2\kI+1}}{2\nI-2\kI+1}
+O(\epsilon).
\label{eqn:intCupperkI}
\nn
\eear

It follows that

\bear
\lefteqn{
\int_0^1
\left\lbrack
2\sum_{\nI=0}^{\kI-1}\PolyML_{2\nI}\left(\tfrac{i x}{\pi}\right)\sqt^{2\nI}
-\left(\frac{1+\sqt}{1-\sqt}\right)^{\frac{ix}{\pi}}
-\left(\frac{1-\sqt}{1+\sqt}\right)^{\frac{ix}{\pi}}
\right\rbrack\frac{d\sqt}{\sqt^{2\kI}}
}
\nn\\ &&
=\lim_{\epsilon\rightarrow 0}\int_\epsilon^1
\left\lbrack
2\sum_{\nI=0}^{\kI-1}\PolyML_{2\nI}\left(\tfrac{i x}{\pi}\right)\sqt^{2\nI}
-\left(\frac{1+\sqt}{1-\sqt}\right)^{\frac{ix}{\pi}}
-\left(\frac{1-\sqt}{1+\sqt}\right)^{\frac{ix}{\pi}}
\right\rbrack\frac{d\sqt}{\sqt^{2\kI}}
\nn\\
&=&
\lim_{\epsilon\rightarrow 0}\Bigl\lbrack
2\sum_{\nI=0}^{\kI-1}\PolyML_{2\nI}\left(\tfrac{i x}{\pi}\right)
\left(\frac{\epsilon^{2\nI-2\kI+1}}{2\kI-2\nI-1}-\frac{1}{2\kI-2\nI-1}\right)
\nn\\ &&
-i\pi\PolyML_{2\kI-1}\left(\tfrac{i x}{\pi}\right)\coth x
+2\sum_{\nI=0}^{\kI-1}\PolyML_{2\nI}\left(\tfrac{i x}{\pi}\right)\frac{\epsilon^{2\nI-2\kI+1}}{2\nI-2\kI+1}
+O(\epsilon)\Bigr\rbrack
\nn\\
&=&
2\sum_{\nI=0}^{\kI-1}\frac{\PolyML_{2\nI}\left(\tfrac{i x}{\pi}\right)}{2\nI-2\kI+1}
-i\pi\PolyML_{2\kI-1}\left(\tfrac{i x}{\pi}\right)\coth x.
\nn
\eear
Equation \eqref{eqn:intsinSquaredML} follows, since $\PolyML_{0}=1$.

\end{proof}

We also need an integral expression for $\xd_\kI$.

\begin{lemma}
\label{lem:xdML}
$\xd_\kI$ is given by
\be\label{eqn:xdML}
\xd_\kI=-\pi^{\sA-1}\sin\left(\frac{\pi\sA}{2}\right)\int_0^\infty
\PolyML_{2(\kI+1)}\left(\tfrac{i x}{\pi}\right)\frac{x^{-\sA}dx}{\sinh x}
\,.
\ee
\end{lemma}

\begin{proof}
Note that \eqref{eqn:xdML} converges at $x=0$ because $\PolyML_{2\nI}(x)=O(x^2)$ for $\nI\ge 1$.
Using \eqref{eqn:xdDef}, \eqref{eqn:xbkdefAlt}, and \eqref{eqn:PolyOrthML},
\be\label{eqn:xdIntML}
\xd_\kI = \xb_{\kI+1}-\xb_\kI
=\innerP{x^{-(\sA+1)/2}}{\PolyOrth_{\kI+1}(x)-\PolyOrth_\kI(x)}
=
\innerP{x^{-(\sA+1)/2}}{\PolyML_{2(\kI+1)}(\sqrt{x})}.
\ee
$\PolyML_{2(\kI+1)}$ is an even function that vanishes at zero, and if we set
$$
\PolyML_{2(\kI+1)}(\sqrt{x})=
\sum_{\jI=1}^{\kI+1} \aPoly_\jI x^\jI,
$$
we can use \eqref{eqn:innerPDefz} to write
\be\label{eqn:innerPML}
\innerP{x^{-(\sA+1)/2}}{\PolyML_{2(\kI+1)}(\sqrt{x})}
=\sum_{\jI=1}^{\kI+1}\aPoly_\jI\eta(\sA-2\jI),
\ee
and using \eqref{eqn:intxsAsinh} we can express this as
\bear
\sum_{\jI=1}^{\kI+1}\aPoly_\jI\eta(\sA-2\jI)&=&
-\sum_{\jI=1}^{\kI+1}\aPoly_\jI
\pi^{\sA-2\jI-1}\sin\left(\frac{\pi\sA}{2}-\jI\pi\right)\int_0^\infty\frac{x^{2\jI-\sA}dx}{\sinh x}
\nn\\
&=&
-\pi^{\sA-1}\sin\left(\frac{\pi\sA}{2}\right)\int_0^\infty
\PolyML_{2(\kI+1)}\left(\tfrac{i x}{\pi}\right)\frac{x^{-\sA}dx}{\sinh x}
\nn
\eear

\end{proof}

Next, we need a couple of identities for Mittag-Leffler polynomials.

\begin{lemma}
\label{lem:DML}
The Mittag-Leffler polynomials satisfy
\be\label{eqn:DML}
\PolyML_{\nI}'(x)=2\sum_{\kI=0}^{\left\lfloor\frac{\nI-1}{2}\right\rfloor}\frac{\PolyML_{\nI-2\kI-1}(x)}{2\kI+1}
\ee
and
\be\label{eqn:MLoverx}
\frac{1}{x}\PolyML_{2\nI+1}(x)=\frac{2}{2\nI+1}\sum_{\kI=0}^\nI\PolyML_{2\kI}(x).
\ee

\end{lemma}

\begin{proof}
These identities follow directly from the generating function \eqref{eqn:genFMLk}.
See \S10 of \cite{Lomont:2001} for details.
\end{proof}

We can now prove the extension of \propref{prop:recrelkz} to $\kI>0$.

\begin{proposition}\label{prop:recrelkAll}
For $\Re\sA<2$, the numbers $\seqxbz$ defined in \eqref{eqn:xbkdefAlt} satisfy the relation 
\be\label{eqn:recrelAllk}
-\frac{1}{2}\left(\frac{\sA}{2\kI-1}+1\right)\xb_{\kI-1}
+\sum_{\nI=0}^\infty\frac{\xb_{\kI+\nI}}{(2\nI+1)(2\nI+3)}=0,
\qquad
\kI=1,2,\dots
\ee
\end{proposition}

\begin{proof}
We use an integral representation of the sum similar to \eqref{eqn:intrepsum}:
\bear
\lefteqn{
\sum_{\nI=0}^\infty\frac{\xb_{\nI+\kI}}{(2\nI+1)(2\nI+3)} 
=
\frac{1}{2}\int_0^1
(1-\sqt^2)\sum_{\nI=0}^\infty\xb_{\nI+\kI}\sqt^{2\nI}
d\sqt
}\nn\\
&=&
\frac{1}{2}\int_0^1
(1-\sqt^2)\left(\genFb(\sqt^2)-\sum_{\mI=0}^{\kI-1}\xb_\mI\sqt^{2\mI}\right)\frac{d\sqt}{\sqt^{2\kI}}\,.
\label{eqn:intrepsumkI}
\eear
The evaluation of the integral is similar to the evaluation of \eqref{eqn:intrepsum} in \propref{prop:recrelkz}, where we used \eqref{eqn:intsinSquared}, except that $\frac{d\sqt}{\sqt^2}$ is replaced with $\frac{d\sqt}{\sqt^{2\kI}}$.
We substitute
\be\label{eqn:intgenFbtFsubs}
(1-\sqt^2)\genFb(\sqt^2)=
\xb_0
+2\pi^{\sA-1}\sin\frac{\pi\sA}{2}
\int_0^\infty
\sin^2\left(\frac{\uMF x}{2\pi}\right)
\frac{x^{-\sA}dx}{\sinh x}
\,,
\ee
where we used \eqref{eqn:intgenFbtF} and $\xb_0=\eta(\sA)$.
We then calculate
\bear
\lefteqn{
\int_0^\infty
\sin^2\left(\frac{\uMF x}{2\pi}\right)
\frac{x^{-\sA}dx}{\sinh x}
=
\frac{1}{4}
\int_0^\infty
\left\lbrack
2
-\left(\frac{1+\sqt}{1-\sqt}\right)^{\frac{ix}{\pi}}
-\left(\frac{1-\sqt}{1+\sqt}\right)^{\frac{ix}{\pi}}
\right\rbrack
\frac{x^{-\sA}dx}{\sinh x}
}\nn\\
&=&
\frac{1}{4}
\int_0^\infty
\left\lbrack
2
-\left(\frac{1+\sqt}{1-\sqt}\right)^{\frac{ix}{\pi}}
-\left(\frac{1-\sqt}{1+\sqt}\right)^{\frac{ix}{\pi}}
+2\sum_{\nI=1}^{\kI-1}\PolyML_{2\nI}\left(\tfrac{i x}{\pi}\right)\sqt^{2\nI}
\right\rbrack
\frac{x^{-\sA}dx}{\sinh x}
\nn\\ &&
-\frac{1}{2}
\sum_{\nI=1}^{\kI-1}\sqt^{2\nI}
\int_0^\infty\PolyML_{2\nI}\left(\tfrac{i x}{\pi}\right)\frac{x^{-\sA}dx}{\sinh x}\,.
\nn\\ &&
\label{eqn:intsinsqsqtML}
\eear
Note that for small $x$ and $\nI>0$ the even Mittag-Leffler polynomials behave as $\PolyML_{2\nI}(x)=O(x^2)$, so the integrals above converge absolutely.

Using \eqref{eqn:xdML} we can write
\be\label{eqn:sumsqtintML}
\sum_{\nI=1}^{\kI-1}\sqt^{2\nI}
\int_0^\infty\PolyML_{2\nI}\left(\tfrac{i x}{\pi}\right)\frac{x^{-\sA}dx}{\sinh x}
=
-\frac{\sum_{\nI=1}^{\kI-1}\xd_{\nI-1}\sqt^{2\nI}}{\pi^{\sA-1}\sin\left(\frac{\pi\sA}{2}\right)}
\ee
and from the definition of $\xd_\nI$ we have
\be\label{eqn:sumxdsqt}
\sum_{\nI=1}^{\kI-1}\xd_{\nI-1}\sqt^{2\nI}
=\sum_{\nI=1}^{\kI-1}(\xb_{\nI}-\xb_{\nI-1})\sqt^{2\nI}
=\xb_{\kI-1}\sqt^{2\kI}-\xb_0+(1-\sqt^2)\sum_{\nI=0}^{\kI-1}\xb_{\nI}\sqt^{2\nI}\,.
\ee
Combining \eqref{eqn:intgenFbtFsubs}, \eqref{eqn:sumsqtintML}, and \eqref{eqn:sumxdsqt}, we have
\bear
\lefteqn{
(1-\sqt^2)\left(\genFb(\sqt^2)-\sum_{\mI=0}^{\kI-1}\xb_\mI\sqt^{2\mI}\right)
}\nn\\
&=&
2\pi^{\sA-1}\sin\left(\frac{\pi\sA}{2}\right)
\int_0^\infty
\left\lbrack
\sin^2\left(\frac{\uMF x}{2\pi}\right)
+\frac{1}{2}\sum_{\nI=1}^{\kI-1}\PolyML_{2\nI}\left(\tfrac{i x}{\pi}\right)\sqt^{2\nI}
\right\rbrack
\frac{x^{-\sA}dx}{\sinh x}
+\xb_{\kI-1}\sqt^{2\kI}\,,
\nn
\eear
and using \eqref{eqn:intsinsqsqtML}, the last equation can be written as
\bear
\lefteqn{
(1-\sqt^2)\left(\genFb(\sqt^2)-\sum_{\mI=0}^{\kI-1}\xb_\mI\sqt^{2\mI}\right)
=\xb_{\kI-1}\sqt^{2\kI}
}\nn\\ &&
+
\frac{1}{2}\pi^{\sA-1}\sin\left(\frac{\pi\sA}{2}\right)
\int_0^\infty
\Bigl\lbrack
2
-\Bigl(\frac{1+\sqt}{1-\sqt}\Bigr)^{\frac{ix}{\pi}}
-\Bigl(\frac{1-\sqt}{1+\sqt}\Bigr)^{\frac{ix}{\pi}}
\nn\\ &&\qquad\qquad\qquad\qquad\qquad\qquad\qquad
+2\sum_{\nI=1}^{\kI-1}\PolyML_{2\nI}\left(\tfrac{i x}{\pi}\right)\sqt^{2\nI}
\Bigr\rbrack
\frac{x^{-\sA}dx}{\sinh x}\,.
\nn
\eear
Now we use \eqref{eqn:intrepsumkI} and \eqref{eqn:intsinSquaredML} to write
\bear
\lefteqn{
\sum_{\nI=0}^\infty\frac{\xb_{\nI+\kI}}{(2\nI+1)(2\nI+3)} 
=
\frac{1}{2}\int_0^1
(1-\sqt^2)\left(\genFb(\sqt^2)-\sum_{\mI=0}^{\kI-1}\xb_\mI\sqt^{2\mI}\right)\frac{d\sqt}{\sqt^{2\kI}}
}\nn\\
&=&
\frac{1}{2}\xb_{\kI-1}
+\frac{1}{4}\pi^{\sA-1}\sin\left(\frac{\pi\sA}{2}\right)
\int_0^\infty
\Bigl\lbrack
2\sum_{\nI=0}^{\kI-1}\frac{\PolyML_{2\nI}\left(\tfrac{i x}{\pi}\right)}{2\nI-2\kI+1}
\nn\\ && \qquad\qquad\qquad\qquad\qquad\qquad\qquad
-i\pi\PolyML_{2\kI-1}\left(\tfrac{i x}{\pi}\right)\coth x
\Bigr\rbrack
\frac{x^{-\sA}dx}{\sinh x}\,.
\nn\\ &&
\label{eqn:intrepsumkIX}
\eear
For $\epsilon>0$ we have, integrating by parts,
\bear
\lefteqn{
\int_\epsilon^\infty
\left\lbrack
2\sum_{\nI=0}^{\kI-1}\frac{\PolyML_{2\nI}\left(\tfrac{i x}{\pi}\right)}{2\nI-2\kI+1}
-i\pi\PolyML_{2\kI-1}\left(\tfrac{i x}{\pi}\right)\coth x
\right\rbrack
\frac{x^{-\sA}dx}{\sinh x}
}\nn\\
&=&
\int_\epsilon^\infty
\left\lbrack
2\sum_{\nI=1}^{\kI-1}\frac{\PolyML_{2\nI}\left(\tfrac{i x}{\pi}\right)}{2\nI-2\kI+1}
+\frac{2}{1-2\kI}
-i\pi\PolyML_{2\kI-1}\left(\tfrac{i x}{\pi}\right)\coth x
\right\rbrack
\frac{x^{-\sA}dx}{\sinh x}
\nn\\
&=&
2\int_\epsilon^\infty
\left(\sum_{\nI=1}^{\kI-1}\frac{\PolyML_{2\nI}\left(\tfrac{i x}{\pi}\right)}{2\nI-2\kI+1}\right)\frac{x^{-\sA}dx}{\sinh x}
+\frac{2}{1-2\kI}
\int_\epsilon^\infty
\frac{x^{-\sA}dx}{\sinh x}
\nn\\ &&
-i\pi\PolyML_{2\kI-1}\left(\tfrac{i\epsilon}{\pi}\right)
\frac{\epsilon^{-\sA}}{\sinh\epsilon}
\nn\\ &&
+\int_\epsilon^\infty
\PolyML_{2\kI-1}'\left(\tfrac{i x}{\pi}\right)
\frac{x^{-\sA}dx}{\sinh x}
+i\pi\sA\int_\epsilon^\infty
\PolyML_{2\kI-1}\left(\tfrac{i x}{\pi}\right)
\frac{x^{-\sA-1}dx}{\sinh x},
\nn\\
&&\label{eqn:intlimepsilon}
\eear
and using \eqref{eqn:DML}, we can write
\be\label{eqn:DPolyML}
\PolyML_{2\kI-1}'\left(\tfrac{i x}{\pi}\right)=
2\sum_{\jI=0}^{\kI-1}\frac{1}{2\jI+1}\PolyML_{2\kI-2\jI-2}\left(\tfrac{i x}{\pi}\right)
=
\frac{2}{2\kI-1}
-2\sum_{\nI=1}^{\kI-1}\frac{1}{2\nI-2\kI+1}\PolyML_{2\nI}\left(\tfrac{i x}{\pi}\right).
\ee
Therefore, \eqref{eqn:intlimepsilon} becomes
\bear
\lefteqn{
\int_\epsilon^\infty
\left\lbrack
2\sum_{\nI=0}^{\kI-1}\frac{\PolyML_{2\nI}\left(\tfrac{i x}{\pi}\right)}{2\nI-2\kI+1}
-i\pi\PolyML_{2\kI-1}\left(\tfrac{i x}{\pi}\right)\coth x
\right\rbrack
\frac{x^{-\sA}dx}{\sinh x}
}\nn\\
&=&
-i\pi\PolyML_{2\kI-1}\left(\tfrac{i\epsilon}{\pi}\right)
\frac{\epsilon^{-\sA}}{\sinh\epsilon}
+i\pi\sA\int_\epsilon^\infty
\PolyML_{2\kI-1}\left(\tfrac{i x}{\pi}\right)
\frac{x^{-\sA-1}dx}{\sinh x}\,.
\label{eqn:intlimepsilonB}
\eear
Using \eqref{eqn:MLoverx} we write
\be\label{eqn:M2kminus1overx}
\frac{1}{x}\PolyML_{2\kI-1}(\tfrac{ix}{\pi})=\frac{2i}{(2\kI-1)\pi}\sum_{\nI=0}^{\kI-1}\PolyML_{2\nI}(\tfrac{ix}{\pi}).
\ee
and therefore
\bear
\lefteqn{
i\pi\sA\int_\epsilon^\infty
\PolyML_{2\kI-1}\left(\tfrac{i x}{\pi}\right)
\frac{x^{-\sA-1}dx}{\sinh x} =
-\frac{2\sA}{2\kI-1}
\int_\epsilon^\infty
\left(\sum_{\nI=0}^{\kI-1}\PolyML_{2\nI}(\tfrac{i x}{\pi})\right)
\frac{x^{-\sA}dx}{\sinh x}
}\nn\\
&=&
-\frac{2\sA}{2\kI-1}\int_\epsilon^\infty
\frac{x^{-\sA}dx}{\sinh x}
-\frac{2\sA}{2\kI-1}
\int_\epsilon^\infty
\left(\sum_{\nI=1}^{\kI-1}\PolyML_{2\nI}(\tfrac{i x}{\pi})\right)
\frac{x^{-\sA}dx}{\sinh x}
\nn
\eear
Note that near $x=0$ we have $\PolyML_{2\nI}\left(\tfrac{i x}{\pi}\right)=O(x^2)$ for $\nI\ge 1$ and so the limit $\epsilon\rightarrow 0$ can be taken in the second term (since $\Re\sA<1$), and using \eqref{lem:xdML} we get
$$
\int_0^\infty
\left(\sum_{\nI=1}^{\kI-1}\PolyML_{2\nI}(\tfrac{i x}{\pi})\right)
\frac{x^{-\sA}dx}{\sinh x}
=-\frac{\pi^{1-\sA}}{\sin\left(\frac{\pi\sA}{2}\right)}\sum_{\nI=1}^{\kI-1}\xd_{\nI-1}
=-\frac{\pi^{1-\sA}}{\sin\left(\frac{\pi\sA}{2}\right)}(\xb_{\kI-1}-\xb_0).
$$
Altogether, \eqref{eqn:intlimepsilonB} becomes
\bear
\lefteqn{
\int_\epsilon^\infty
\left\lbrack
2\sum_{\nI=0}^{\kI-1}\frac{\PolyML_{2\nI}\left(\tfrac{i x}{\pi}\right)}{2\nI-2\kI+1}
-i\pi\PolyML_{2\kI-1}\left(\tfrac{i x}{\pi}\right)\coth x
\right\rbrack
\frac{x^{-\sA}dx}{\sinh x}
}\nn\\
&=&
-i\pi\PolyML_{2\kI-1}\left(\tfrac{i\epsilon}{\pi}\right)
\frac{\epsilon^{-\sA}}{\sinh\epsilon}
-\frac{2\sA}{2\kI-1}\int_\epsilon^\infty
\frac{x^{-\sA}dx}{\sinh x}
\nn\\ &&
+\frac{2\sA\pi^{1-\sA}}{(2\kI-1)\sin\left(\frac{\pi\sA}{2}\right)}(\xb_{\kI-1}-\xb_0)+O(\epsilon)
\,.
\label{eqn:intlimepsilonC}
\eear
From \eqref{eqn:M2kminus1overx} we get
$$
\PolyML_{2\kI-1}(\tfrac{i\epsilon}{\pi})=\frac{2i\epsilon}{(2\kI-1)\pi}+O(\epsilon^3).
$$
and therefore
\be\label{eqn:M2kminus1}
-i\pi\PolyML_{2\kI-1}\left(\tfrac{i\epsilon}{\pi}\right)
\frac{\epsilon^{-\sA}}{\sinh\epsilon}
=\frac{2\epsilon^{-\sA}}{2\kI-1}+O(\epsilon^{2-\sA}).
\ee
We also note that for $\Re\sA<2$ we have
\be\label{eqn:limintxminussA}
\lim_{\epsilon\rightarrow 0}\left\lbrack
\epsilon^{-\sA}-\sA\int_\epsilon^\infty
\frac{x^{-\sA}dx}{\sinh x}
\right\rbrack
=\frac{\sA\pi^{1-\sA}\eta(\sA)}{\sin\left(\frac{\pi\sA}{2}\right)},
\ee
as can be verified by integration by parts:
\bear
\sA\int_\epsilon^\infty
\frac{x^{-\sA}dx}{\sinh x}
&=&
-\int_\epsilon^\infty
(x^{-\sA})'\frac{x}{\sinh x}dx
=\frac{\epsilon^{1-\sA}}{\sinh\epsilon}
+\int_\epsilon^\infty
x^{-\sA}\left(\frac{x}{\sinh x}\right)'dx,
\nn
\eear
and therefore for $\Re\sA<2$,
$$
\lim_{\epsilon\rightarrow 0}\left\lbrack
\epsilon^{-\sA}-\sA\int_\epsilon^\infty
\frac{x^{-\sA}dx}{\sinh x}
\right\rbrack
=\int_0^\infty
x^{-\sA}\left(\frac{x}{\sinh x}\right)'dx.
$$
For $\Re\sA<0$ the same formula, using \eqref{eqn:intxsAsinh}, shows that
$$
\int_0^\infty
x^{-\sA}\left(\frac{x}{\sinh x}\right)'dx=\sA\int_\epsilon^\infty
\frac{x^{-\sA}dx}{\sinh x}
=\frac{\sA\pi^{1-\sA}\eta(\sA)}{\sin\left(\frac{\pi\sA}{2}\right)},
$$
and analytic continuation in $\sA$ shows that 
$$
\int_0^\infty
x^{-\sA}\left(\frac{x}{\sinh x}\right)'dx
=\frac{\sA\pi^{1-\sA}\eta(\sA)}{\sin\left(\frac{\pi\sA}{2}\right)}
=\frac{\sA\pi^{1-\sA}}{\sin\left(\frac{\pi\sA}{2}\right)}\xb_0
$$
holds for $\Re\sA<2$.

Combining \eqref{eqn:intrepsumkIX}, \eqref{eqn:intlimepsilonC}, \eqref{eqn:M2kminus1}, and \eqref{eqn:limintxminussA}, we get

\bear
\sum_{\nI=0}^\infty\frac{\xb_{\nI+\kI}}{(2\nI+1)(2\nI+3)} 
&=&
\frac{1}{2}\xb_{\kI-1}
+\frac{\sA}{2(2\kI-1)}\xb_{\kI-1}
\nn
\eear

\end{proof}

%\begin{corollary}
If $\sZ$ is such that $\zeta(\sZ)=0$, then $\xb_0=\eta(\sZ)=0$ and \eqref{eqn:recrelAllk} reduces to \eqref{eqn:OpWseq}.
%\end{corollary}

% ======================================================================
\section{Proof of the sufficient condition in \thmref{thm:main}}\label{sec:sufpf}

We will now prove the second part of \thmref{thm:main} which states that, under certain conditions listed there, existence of a nontrivial solution to \eqref{eqn:OpW} implies $\zeta(\sA)=0$.

Let $\sA\in\C$, and let $\seqxc$ be a solution to
\bear
\lefteqn{
\begin{pmatrix}
1 & 0 & 0 & 0 & 0 & \cdots \\
\end{pmatrix} =
}\nn\\ &&
\begin{pmatrix}
\xc_1 & \xc_2 & \xc_3 & \xc_4 & \xc_5 & \cdots \\
\end{pmatrix}
\begin{pmatrix}
\frac{1}{1\cdot 3} & \frac{1}{3\cdot 5} & \frac{1}{5\cdot 7} & \frac{1}{7\cdot 9} & \cdots \\
 & & & & \\
-\frac{\sA}{2\cdot 3}-\frac{1}{2} &  \frac{1}{1\cdot 3} & \frac{1}{3\cdot 5} & \frac{1}{5\cdot 7} & \cdots \\
 & & & & \\
0 & -\frac{\sA}{2\cdot 5}-\frac{1}{2}  &  \frac{1}{1\cdot 3} & \frac{1}{3\cdot 5} & \cdots \\
 & & & & \\
0 &0  & -\frac{\sA}{2\cdot 7}-\frac{1}{2}  &  \frac{1}{1\cdot 3} & \cdots \\
 & & & & \\
0 & 0 & 0 & -\frac{\sA}{2\cdot 9}-\frac{1}{2} & \cdots \\
 & & & & \\
\vdots &\vdots & \vdots &\vdots & \ddots\\
\end{pmatrix}.
%\begin{pmatrix}
%1 \\ \\ 0 \\ \\ 0 \\ \\ 0 \\ \\ 0 \\ \\ \vdots \\
%\end{pmatrix}.
\nn\\ &&
\label{eqn:xcW1}
\eear
This is to be read as a series of identities for finite sums
\be\label{eqn:xcrl}
\delta_{\kI 1}=
-\left(\frac{\sA}{2(2\kI+1)}+\frac{1}{2}\right)\xc_{\kI+1}
+\sum_{n=0}^{\kI-1}\frac{\xc_{\kI-n}}{(2n+1)(2n+3)}\,,
\qquad
\kI=1,2,\dots
\ee
%which does not require any condition on the large $\kI$ behavior of $\xc_\kI$ for convergence.
Now, suppose $\seqxb$ is as in \thmref{thm:main} and $\xb_1\neq 0.$
Suppose also that we can find a solution $(\xc_\kI)_{\kI=1}^\infty$ to \eqref{eqn:xcrl} with $\xc_\kI$'s that tend to zero fast enough as $\kI\rightarrow\infty$ so that the infinite sum implied by the matrix notation
\be\label{eqn:cWb0}
\begin{pmatrix}
\xc_1 & \xc_2 & \xc_3 & \xc_4 & \xc_5 & \cdots \\
\end{pmatrix}
\begin{pmatrix}
\frac{1}{1\cdot 3} & \frac{1}{3\cdot 5} & \frac{1}{5\cdot 7} & \frac{1}{7\cdot 9} & \cdots \\
 & & & & \\
-\frac{\sA}{2\cdot 3}-\frac{1}{2} &  \frac{1}{1\cdot 3} & \frac{1}{3\cdot 5} & \frac{1}{5\cdot 7} & \cdots \\
 & & & & \\
0 & -\frac{\sA}{2\cdot 5}-\frac{1}{2}  &  \frac{1}{1\cdot 3} & \frac{1}{3\cdot 5} & \cdots \\
 & & & & \\
0 &0  & -\frac{\sA}{2\cdot 7}-\frac{1}{2}  &  \frac{1}{1\cdot 3} & \cdots \\
 & & & & \\
0 & 0 & 0 & -\frac{\sA}{2\cdot 9}-\frac{1}{2} & \cdots \\
 & & & & \\
\vdots &\vdots & \vdots &\vdots & \ddots\\
\end{pmatrix}
\begin{pmatrix}
\xb_1 \\ \\ \xb_2 \\ \\ \xb_3 \\ \\ \xb_4 \\ \\ \vdots \\ \\ \vdots \\
\end{pmatrix}
%\stackrel{?}{=}0.
\ee
is absolutely convergent.
Then, by \eqref{eqn:OpW}, the expression \eqref{eqn:cWb0} vanishes, and by \eqref{eqn:xcW1} it equals $\xb_1$, and so $\xb_1=0$, contradicting the assumptions of \thmref{thm:main}.

Thus, in order to prove the second part of \thmref{thm:main} it would suffice to show that if $\zeta(\sA)\neq 0$ there exists a solution of the recursion relation \eqref{eqn:xcrl} with the requisite absolute convergence property. Unfortunately, that does not appear to be the case, which complicates the rest of the argument a bit. What we will show below, however, is that if $\zeta(\sA)\neq 0$, then there exists a nonzero solution to \eqref{eqn:xcrl} with $\xc_\kI=O(\frac{1}{\kI})$, and using the expressions \eqref{eqn:OpWxd} for the recursion relation that $\xb_\kI$ satisfies, the $\xc_\kI=O(\frac{1}{\kI})$ behavior is sufficient to reach the same conclusion that $\xb_1=0$, and hence a contradiction.

Thus, the actual strategy will be to find a sequence $\seqxc$ that satisfies \eqref{eqn:xcrl} and falls off fast enough so that, using \eqref{eqn:OpWxd},
\be\label{eqn:actualxcxd}
\sum_{\kI=1}^\infty\xc_\kI\left\lbrack
-\frac{\sA}{2(2\kI-1)}\xb_{\kI-1}
+\frac{1}{2}\sum_{n=0}^\infty\frac{\xd_{\kI+n-1}}{2n+1}
\right\rbrack = 0
\ee
is absolutely convergent, and hence the terms can be rearranged to conclude that $\xb_1=0$, using \eqref{eqn:OpWxd} and \eqref{eqn:xcrl}. The conditions that we will use to show absolute convergence of \eqref{eqn:actualxcxd} are
\be\label{eqn:sumxcfinA}
\sum_{\kI=1}^\infty\frac{|\xc_\kI||\xb_{\kI-1}|}{2\kI-1}<\infty\qquad
\text{and}\qquad
\sum_{\nI=0}^\infty\sum_{\kI=1}^\infty\frac{|\xc_\kI||\xd_{\kI+\nI-1}|}{2\nI+1}<\infty.
\ee

To see where the condition $\zeta(\sA)\neq 0$ enters, we note that the recursion relation \eqref{eqn:xcrl} determines $\xc_2,\xc_3,\dots$ in terms of $\xc_1$, but $\xc_1$ is arbitrary. For a generic $\xc_1$, as we will see below, the conditions \eqref{eqn:sumxcfinA} cannot be satisfied. However, for generic $\sA$, $\xc_1$ can be adjusted so as to obtain a solution $\seqxc$ with good asymptotic behavior as $\kI\rightarrow\infty$. But the formula for $\xc_1$ [equation \eqref{eqn:xc1choice} below] has a denominator that vanishes when $\zeta(\sA)=0$, and so we can only find a suitable $\xc_1$ in this way when $\zeta(\sA)\neq 0$.
We proceed to show the details.

Define the generating functions
\be\label{eqn:DefgenFc}
\genFc(\tF)\eqdef\sum_{\kI=1}^\infty\frac{\xc_\kI\tF^{2\kI-1}}{2\kI-1}\,,
\qquad
\genFcD(\tF)\eqdef
\genFc'(\tF)=\sum_{\kI=1}^\infty\xc_\kI\tF^{2\kI-2}\,.
\ee
The recursion relation \eqref{eqn:xcrl} is equivalent to
\be\label{eqn:genFcODE}
0 =\tF^2
+\frac{\sA}{2}\left\lbrack\frac{1}{\tF}\genFc(\tF)-\xc_1\right\rbrack
-\frac{1}{2}\xc_1
-\frac{1}{4\tF}\genFc'(\tF)(\tF^2-1)\log\left(\frac{1+\tF}{1-\tF}\right).
\ee
Changing variables to $\uF$, defined by $\tF=\tanh\uF$,
we find the solution to the first order linear differential equation \eqref{eqn:genFcODE} in the form
\be\label{eqn:genFcInt}
\genFc=\frac{1}{\uF^\sA}\int_0^\uF\left\lbrack
(\sA+1)\xc_1\tanh\vF
-2\tanh^3\vF
\right\rbrack\vF^{\sA-1}d\vF
+\frac{\ConstF}{\uF^\sA}\,,\qquad
\uF\eqdef\frac{1}{2}\log\left(\frac{1+\tF}{1-\tF}\right),
\ee
where $\ConstF$ is an arbitrary integration constant.
The integral converges, since we assume $1>\Re\sA>-1$.

In order to have a good analytic behavior at $\uF=0$ (hence at $\tF=0$), we need to set
$\ConstF=0$, and in order for $\xc_\kI$ to tend to zero as $\kI\rightarrow\infty$, we at least need $\genFc$ to be analytic in the whole disk $|\tF|<1$.
We will go further and require $\genFc$ to converge as $\tF\rightarrow 1$, i.e., as $\uF\rightarrow\infty$.
For $-1<\Re\sA<0$, we can write the $\uF\rightarrow\infty$ limit of the coefficient of $\xc_1$ in \eqref{eqn:genFcInt} as
\be\label{eqn:limCoeffxc1}
\int_0^\infty\vF^{\sA-1}\tanh\vF\,d\vF
%=(2^{1-\sA}-1)\pi^{\sA}\frac{\zeta(1-\sA)}{\cos(\frac{\pi\sA}{2})}
=-\frac{4\Gamma(\sA)}{2^{\sA+1}}\eta(\sA).
\ee
%and using the functional identity of zeta we see that it is proportional to $\zeta(\sA)$.
Unfortunately, \eqref{eqn:limCoeffxc1} does not converge for $\Re\sA>0$, and so we need to rewrite \eqref{eqn:genFcInt} in a way that will have a manifestly convergent limit as $\uF\rightarrow\infty$. We achieve this by separating the integral into $\int_0^1$ and $\int_1^\infty$. Thus, for $1>\Re\sA>-1$ we rewrite \eqref{eqn:genFcInt} with $\ConstF=0$ as
\bear
\genFc &=&
\frac{1}{\uF^\sA}(\sA+1)\xc_1\int_0^1\vF^{\sA-1}\tanh\vF d\vF
+\frac{1}{\uF^\sA}(\sA+1)\xc_1\int_1^\uF\vF^{\sA-1}\left(\tanh\vF-1\right)d\vF
\nn\\ &&
-\frac{2}{\uF^\sA}\int_0^1\vF^{\sA-1}\tanh^3\vF d\vF
-\frac{2}{\uF^\sA}\int_1^\uF\vF^{\sA-1}\left(\tanh^3\vF-1\right)d\vF
\nn\\ &&
+\frac{1}{\uF^\sA}(\sA+1)\xc_1\int_1^\uF\vF^{\sA-1}d\vF
-\frac{2}{\uF^\sA}\int_1^\uF\vF^{\sA-1}d\vF
\nonumber\\
&=&
\frac{1}{\uF^\sA}(\sA+1)\xc_1\left\lbrack
\int_0^1\vF^{\sA-1}\tanh\vF d\vF
+\int_1^\uF\vF^{\sA-1}\left(\tanh\vF-1\right)d\vF
-\frac{1}{\sA}
\right\rbrack
\nn\\ &&
-\frac{2}{\uF^\sA}\left\lbrack
\int_0^1\vF^{\sA-1}\tanh^3\vF d\vF
+\int_1^\uF\vF^{\sA-1}\left(\tanh^3\vF-1\right)d\vF
-\frac{1}{\sA}
\right\rbrack
\nn\\ &&
+\left(\frac{1}{\sA}+1\right)\xc_1
-\frac{2}{\sA}\,.
\nonumber\\ &&
\label{eqn:genFcExtended}
\eear
Taking the derivative of \eqref{eqn:genFcExtended} we get
\bear
\genFcD &=&
-\frac{\sA(\sA+1)}{(1-\tF^2)\uF^{\sA+1}}\xc_1\left\lbrack
\int_0^1\vF^{\sA-1}\tanh\vF d\vF
+\int_1^\uF\vF^{\sA-1}\left(\tanh\vF-1\right)d\vF
-\frac{1}{\sA}
\right\rbrack
\nn\\ &&
+\frac{2\sA}{(1-\tF^2)\uF^{\sA+1}}\left\lbrack
\int_0^1\vF^{\sA-1}\tanh^3\vF d\vF
+\int_1^\uF\vF^{\sA-1}\left(\tanh^3\vF-1\right)d\vF
-\frac{1}{\sA}
\right\rbrack
\nonumber\\ &&
+\frac{2(1+\tF+\tF^2)-(\sA+1)\xc_1}{(1+\tF)\uF}
\,.
\nonumber\\ &&
\label{eqn:genFcDerivative}
\eear
where we substituted $\tF=\tanh\uF$.

The behavior of $\genFcD$ in the limit $\uF\rightarrow\infty$ ($\tF\rightarrow 1$) will be governed by the first two terms on the RHS of \eqref{eqn:genFcDerivative}. $\genFcD$ is at its best behavior as $\uF\rightarrow\infty$ if we set
\bear
\xc_1 &=&
\left(\frac{2}{\sA+1}\right)
\frac{
\int_0^1\vF^{\sA-1}\tanh^3\vF d\vF
+\int_1^\infty\vF^{\sA-1}\left(\tanh^3\vF-1\right)d\vF
-\frac{1}{\sA}
}{
\int_0^1\vF^{\sA-1}\tanh\vF d\vF
+\int_1^\infty\vF^{\sA-1}\left(\tanh\vF-1\right)d\vF
-\frac{1}{\sA}
}\,,
\nn\\
&&
\label{eqn:xc1choiceRatio}
\eear
so that $(1-\tF^2)\uF^{\sA+1}\genFcD\rightarrow 0$ as $\tF\rightarrow 1$.
We can simplify \eqref{eqn:xc1choiceRatio} as follows. From \eqref{eqn:limCoeffxc1} we get
\be\label{eqn:limCoeffxc1Ex}
\int_0^1\vF^{\sA-1}\tanh\vF d\vF
+\int_1^\infty\vF^{\sA-1}\left(\tanh\vF-1\right)d\vF
-\frac{1}{\sA}
=-\frac{4\Gamma(\sA)}{2^{\sA+1}}\eta(\sA).
\ee
While \eqref{eqn:limCoeffxc1} only converges for $-1<\Re\sA<0$, \eqref{eqn:limCoeffxc1Ex} converges for all $\Re\sA>-1$.
Similarly,
\be\label{eqn:inttanhEtaAC}
\int_0^1\vF^{\sA-1}\tanh^3\vF d\vF
+\int_1^\infty\vF^{\sA-1}\left(\tanh^3\vF-1\right)d\vF
-\frac{1}{\sA}
=-\frac{4\Gamma(\sA)}{2^{\sA+1}}\left\lbrack 2\eta(\sA-2)+\eta(\sA)\right\rbrack,
\ee
which is valid for $\Re\sA>-3$, and follows by analytically continuing the identity
\be\label{eqn:inttanhEta}
\int_0^\infty\vF^{\sA-1}\tanh^3\vF\,d\vF 
=-\frac{4\Gamma(\sA)}{2^{\sA+1}}\left\lbrack 2\eta(\sA-2)+\eta(\sA)\right\rbrack,
\qquad
(0>\Re\sA>-3).
\ee
The latter can easily be derived by integration by parts and Taylor expansion in $e^{-2\vF}$.
Using \eqref{eqn:inttanhEtaAC} and \eqref{eqn:inttanhEtaAC} we can rewrite \eqref{eqn:xc1choiceRatio} as
\bear
\xc_1 &=&
\left(\frac{2}{\sA+1}\right)
\frac{
2\eta(\sA-2)+\eta(\sA)
}{
\eta(\sA)
}\,.
\label{eqn:xc1choice}
\eear
Thus, if $\zeta(\sA)\neq 0$, and hence $\eta(\sA)\neq 0$, we can set $\xc_1$ as in \eqref{eqn:xc1choice}, and for that value the generating function $\genFcD$ in \eqref{eqn:genFcDerivative} will behave better than $O(\frac{1}{(1-\tF^2)\uF^{\sA+1}})$ as $\uF\rightarrow\infty$.
For such a solution, we will now show that \eqref{eqn:sumxcfinA} is satisfied.
We first need to show the following.

\begin{proposition}
\label{prop:xcOk}
Let $\sA$ be such that $\zeta(\sA)\neq 0$, with $0<\Re\sA<1$, and let $\seqxc$ be a solution to \eqref{eqn:xcrl} with $\xc_1$ given by \eqref{eqn:xc1choice}, which is well defined since $\eta(\sA)\neq 0$ by assumption. Then $\xc_\kI=O(\frac{1}{\kI})$.

\end{proposition}

\begin{proof}
We will estimate $\xc_\kI$ for large $\kI$ using the generating function $\genFcD$ given in \eqref{eqn:genFcDerivative}, but first we need to rewrite it.
Recall that $\uF$ is defined in \eqref{eqn:genFcInt} as $\uF\eqdef\frac{1}{2}\log\left(\frac{1+\tF}{1-\tF}\right)$, and it is analytic in $\tF$ anywhere away from the two real rays $(-\infty,-1]\cup[1,\infty)$. The inverse relation between $\tF$ and $\uF$ is $\tF=\tanh\uF$.
Next, we note that the expression \eqref{eqn:genFcInt} (with $\ConstF=0$ as we discussed above),
\be\label{eqn:genFcIntZ}
\genFc(\uF)=\frac{1}{\uF^\sA}\int_0^\uF\left\lbrack
(\sA+1)\xc_1\tanh\vF
-2\tanh^3\vF
\right\rbrack\vF^{\sA-1}d\vF
\,,
\ee
defines an analytic function of $\uF$ at $\uF=0$.
Since $\tanh\uF$ only has poles at $\uF=i(2n+1)\frac{\pi}{2}$ ($n=0,1,2,\dots$), $\genFc$ can be analytically continued to the entire $\uF$-plane except for cuts which we can choose to be along the segments $[(4n+1)\frac{i\pi}{2},(4n+3)\frac{i\pi}{2}]$ ($n\in\Z$). We will only need the analytic continuation to the strip $\{|\Im\uF|\le\frac{\pi}{2}\}\setminus\{-\frac{i\pi}{2},\frac{i\pi}{2}\}$ depicted in \figref{fig:PathsP12}. The expression \eqref{eqn:genFcDerivative} for $\genFcD(\uF)=\genFc'(\uF)$ is also analytic in this strip.

For $\xc_1$ given by \eqref{eqn:xc1choice}, we can rewrite \eqref{eqn:genFcDerivative} in a way that reflects a better behavior at $\uF\rightarrow\infty$.
First, we use the equivalent form \eqref{eqn:xc1choiceRatio} to write
\bear
0 &=&
\frac{\sA}{(1-\tF^2)\uF^{\sA+1}}(\sA+1)\xc_1\left\lbrack
\int_0^1\vF^{\sA-1}\tanh\vF d\vF
+\int_1^\infty\vF^{\sA-1}\left(\tanh\vF-1\right)d\vF
-\frac{1}{\sA}
\right\rbrack
\nn\\ &&
-\frac{2\sA}{(1-\tF^2)\uF^{\sA+1}}\left\lbrack
\int_0^1\vF^{\sA-1}\tanh^3\vF d\vF
+\int_1^\infty\vF^{\sA-1}\left(\tanh^3\vF-1\right)d\vF
-\frac{1}{\sA}
\right\rbrack,
\nn\\ &&
\label{eqn:genFcExtendeduFCpeff}
\eear
and then we add \eqref{eqn:genFcExtendeduFCpeff} to \eqref{eqn:genFcDerivative} to get
\bear
\genFcD &=&
\frac{\sA}{(1-\tF^2)\uF^{\sA+1}}(\sA+1)\xc_1
\int_\uF^\infty\vF^{\sA-1}\left(\tanh\vF-1\right)d\vF
\nn\\ &&
-\frac{2\sA}{(1-\tF^2)\uF^{\sA+1}}
\int_\uF^\infty\vF^{\sA-1}\left(\tanh^3\vF-1\right)d\vF
+\frac{2(1+\tF+\tF^2)-(\sA+1)\xc_1}{(1+\tF)\uF}
\,.
\nonumber\\ &&
\label{eqn:genFcDerivativExtendedC}
\eear
The integrals are to be performed along a path within the strip $\{|\Im\uF|\le\frac{\pi}{2}\}\setminus\{-\frac{i\pi}{2},\frac{i\pi}{2}\}$.

Now we are ready to estimate $\xc_\kI$ for large $\kI$,
using the definition \eqref{eqn:DefgenFc}.
For $\kI\ge 1$, we calculate $\xc_\kI$ with a contour integral around the origin,
\bear
\xc_\kI &=& 
\frac{1}{2\pi i}\oint_{\ContourC_3}
\genFcD(\uF(\tF))\frac{d\tF}{\tF^{2\kI-1}}
=
\frac{1}{2\pi i}\oint_{\ContourC_3'}
\frac{\genFcD(\uF)d\uF}{\cosh^2\uF\tanh^{2\kI-1}\uF}\,,
\label{eqn:xckContourI}
\eear
where $\ContourC_3$ is a circle around $\tF=0$ of radius less than $1$, and $\ContourC_3'$ is the image of $\ContourC_3$ under the map $\tF\mapsto\uF(\tF)$.
The integrand is well-defined as long as $\sinh\uF\neq 0$.
We now deform the contour $\ContourC_3'$ into two paths $\PathP_1',\PathP_2'$ that run along $\Im\uF=\pm \frac{\pi}{2}$, avoiding the singularities at $\uF=\pm i\frac{\pi}{2}$ (where $\sinh\uF=0$) by going around them along small semicircles (see \figref{fig:PathsP12}).

\vskip 12pt
\begin{figure}[t]
\begin{picture}(400,140)
\put(200,70){\begin{picture}(0,0)

\put(-80,50){\vector(-1,0){70}}
\put(-10,50){\line(-1,0){70}}
\put(80,50){\line(-1,0){70}}
\put(150,50){\vector(-1,0){70}}
\put(155,45){$\PathP_1'$}

\put(-165,-55){$\PathP_2'$}
\put(-150,-50){\vector(1,0){70}}
\put(-80,-50){\line(1,0){70}}
\put(10,-50){\line(1,0){70}}
\put(80,-50){\vector(1,0){70}}

\qbezier(10,50)(10,44)(5,41)
\qbezier(5,41)(0,38)(-5,41)
\qbezier(-5,41)(-10,44)(-10,50)
\put(0,50){\circle*{4}}
\put(-2,60){$\frac{i\pi}{2}$}

\put(-150,-50){\vector(1,0){70}}
\put(-80,-50){\line(1,0){70}}
\put(10,-50){\line(1,0){70}}
\put(80,-50){\vector(1,0){70}}

\qbezier(10,-50)(10,-44)(5,-41)
\qbezier(5,-41)(0,-38)(-5,-41)
\qbezier(-5,-41)(-10,-44)(-10,-50)
\put(0,-50){\circle*{4}}
\put(-2,-62){$-\frac{i\pi}{2}$}

\put(0,0){\circle*{4}}
\put(-3,5){$0$}
\put(0,0){\circle{40}}
\put(20,0){\vector(0,1){0}}
\put(22,0){$\ContourC_3'$}

\thinlines
\multiput(3,3)(12,12){4}{\line(1,1){10}}
\put(50,50){\circle*{4}}
\put(47,53){$\uF$}

\put(110,50){\circle*{4}}
\put(105,57){$\frac{i\pi}{2}+\betaC_\kI$}

\put(-110,50){\circle*{4}}
\put(-115,57){$\frac{i\pi}{2}-\betaC_\kI$}

\end{picture}}
\end{picture}
\caption{
The range $|\Im\uF|\le\frac{\pi}{2}$ of the complex $\uF$-plane, with $\pm\frac{i\pi}{2}$ excluded, is part of the domain of analyticity of $\genFcD$. Also depicted are contours in the $\uF$-plane used in the evaluation of \eqref{eqn:xckContourI}.
The contour $\ContourC_3'$ is a small circle around the origin.
The integral \eqref{eqn:xckContourI} can be deformed to two paths, $\PathP_1'$ and $\PathP_2'$, that run along $\Im\uF=\pm\frac{\pi}{2}$, respectively, avoiding the singularities at $\uF=\pm\frac{i\pi}{2}$ with small semicircles.
The dashed line from $0$ to $\uF$ is the path for the $\int(\cdots)d\vF$ integral \eqref{eqn:genFcDerivativExtendedC} which defines $\genFcD(\uF)$.
The marked points $\frac{i\pi}{2}\pm\betaC_\kI$ are used to separate the ``large $\uF$'' from the ``not-large $\uF$'' part of the path.
$\PathP_3'$ (not shown) is defined as the part of the path $\PathP_1'$ from $\frac{i\pi}{2}+\betaC_\kI$ to $\frac{i\pi}{2}-\betaC_\kI$.
}
\label{fig:PathsP12}
\end{figure}

We have
\bear
\oint_{\ContourC_3'}
\frac{\genFcD(\uF)d\uF}{\cosh^2\uF\tanh^{2\kI-1}\uF}
&=&
\int_{\PathP_1'}
\frac{\genFcD(\uF)d\uF}{\cosh^2\uF\tanh^{2\kI-1}\uF}
+
\int_{\PathP_2'}
\frac{\genFcD(\uF)d\uF}{\cosh^2\uF\tanh^{2\kI-1}\uF}
\nn\\
&=&
2\int_{\PathP_1'}
\frac{\genFcD(\uF)d\uF}{\cosh^2\uF\tanh^{2\kI-1}\uF}
=
2\int_{\PathP_1'}
\frac{(1-\tF^2)\genFcD(\uF)d\uF}{\tanh^{2\kI-1}\uF}
\label{eqn:IntgenFcDOnP1}
\eear
since $\genFcD(\uF)$ is an even function of $\uF$.

Along $\PathP_1'$, for large $\uF$ the integrand $\frac{\genFcD(\uF)}{\cosh^2\uF\tanh^{2\kI-1}}$ is small because $|\tanh\uF|>1$ (along that portion of the path), $\cosh\uF$ is large, and $\genFcD(\uF)$ is small. On the other hand, on the part of $\PathP_1'$ where $\uF$ is not large, $\genFcD(\uF)$ and $1/\cosh^2\uF$ are bounded (by a $\kI$-independent bound), and $1/\tanh^{2\kI-1}\uF$ is small. To put a bound on $\genFcD$ for not-large $\uF$, we will use the derivative of \eqref{eqn:genFcIntZ} to write 
\bear
\frac{\genFcD(\uF)}{\cosh^2\uF}&=&-\frac{\sA}{\uF^{\sA+1}}\int_0^\uF\left\lbrack
(\sA+1)\xc_1\tanh\vF
-2\tanh^3\vF
\right\rbrack\vF^{\sA-1}d\vF
\nonumber\\ &&
+\frac{1}{\uF}\left\lbrack
(\sA+1)\xc_1\tanh\uF
-2\tanh^3\uF
\right\rbrack
\label{eqn:genFcIntZDerivative}
\,,
\eear
with the integral taken along a straight line from $0$ to $\uF$.

On the other hand, to show that $\genFcD(\uF)$ is small for large $\uF$, we will use the expression \eqref{eqn:genFcDerivativExtendedC}.
We now turn to the technical details.

% - - - - - - - - - - - - - - - - - - - - - - - - - - - - - - - - - - -
%\noindent\underline{\bf Breaking $\PathP_1$ up into large and not-large $\uF$}
%\newline

\subsection{Breaking $\PathP_1$ up into large and not-large $\uF$}
%\newline
We need to separate the part of $\PathP_1$ with large $|\Re\uF|$ from the rest of the path.
We pick a $\betaC_\kI>0$, whose precise value will be determined later on, and we mark the points $i\frac{\pi}{2}\pm\betaC_\kI$ on $\PathP_1$ (see \figref{fig:PathsP12}).
For now, we only assume that $\betaC_\kI\rightarrow\infty$ as $\kI\rightarrow\infty$.
Let the semicircle $\SemiC_1$ (in the middle of $\PathP_1$) be of radius $\SemiRad>0$, and
separate the integral
\be\label{eqn:PathP1Only}
\int_{\PathP_1'}
\frac{\genFcD(\uF)d\uF}{\cosh^2\uF\tanh^{2\kI-1}\uF},
\ee
into
$$
\int_{\PathP_1'} = \int_{\frac{i\pi}{2}-\infty}^{\frac{i\pi}{2}-\betaC_\kI}
+\int_{\frac{i\pi}{2}-\betaC_\kI}^{\frac{i\pi}{2}-\SemiRad}
+\int_{\SemiC_1}
+\int_{\frac{i\pi}{2}+\SemiRad}^{\frac{i\pi}{2}+\betaC_\kI}
+\int_{\frac{i\pi}{2}+\betaC_\kI}^{\frac{i\pi}{2}+\infty}
\,.
$$
We then denote the portion of the path along ``not-large $\uF$'' by
$$
\int_{\PathP_3'}\eqdef
\int_{\frac{i\pi}{2}-\betaC_\kI}^{\frac{i\pi}{2}+\betaC_\kI}
\eqdef
\int_{\frac{i\pi}{2}-\betaC_\kI}^{\frac{i\pi}{2}-\SemiRad}
+\int_{\SemiC_1}
+\int_{\frac{i\pi}{2}+\SemiRad}^{\frac{i\pi}{2}+\betaC_\kI}
$$
so that
\be\label{eqn:PathP1BreakUp}
\int_{\PathP_1'} = \int_{\frac{i\pi}{2}-\infty}^{\frac{i\pi}{2}-\betaC_\kI}
+\int_{\PathP_3'}
+\int_{\frac{i\pi}{2}+\betaC_\kI}^{\frac{i\pi}{2}+\infty}
\,.
\ee

%  - - - - - - - - - - - - - - - - - - - - - - - - - - - - - - - - - - -

\subsection{A bound for not-large $\uF$}

Along the entire path $\PathP_1'$ the expression $\genFcD(\uF)/\cosh^2\uF=(1-\tF^2)\genFcD(\uF)$ is bounded by a $\kI$-independent bound,
\be\label{eqn:BoundOnFcDOverCoshOnP1}
\left|\frac{\genFcD(\uF)}{\cosh^2\uF}\right|<\ConstC_{6}\qquad
\text{along $\PathP_1'$},
\ee
where $\ConstC_{6}$ is a ($\sA$ and $\SemiRad$ dependent) constant [to be defined in \eqref{eqn:DefConstC6} below].
To see this we need the expression \eqref{eqn:genFcIntZDerivative}.
We set $\vF=\factorL\uF$ with $0\le\factorL\le 1$, and then
\bear
\frac{\genFcD}{\cosh^2\uF}&=&-\frac{\sA}{\uF}\int_0^1\left\lbrack
(\sA+1)\xc_1\tanh(\factorL\uF)
-2\tanh^3(\factorL\uF)
\right\rbrack\factorL^{\sA-1}d\factorL
\nn\\ &&
+\frac{(\sA+1)\xc_1\tanh\uF
-2\tanh^3\uF}{\uF}
\,.
\nn\\
&&\label{eqn:genFcIntZfL}
\eear
Now, for $\uF\in\PathP_3'$ the argument $\vF=\factorL\uF$, which appears in the integrand of \eqref{eqn:genFcIntZfL}, is in the domain $\DomainD$ defined by excising from the strip $|\Im\vF|\le\frac{\pi}{2}$ two semicircles of radius, say, $\SemiRad/2$ around $\frac{i\pi}{2}$:
$$
\DomainD=\{\vF:\quad\text{$|\Im\vF|\le\frac{\pi}{2}$ and
$\left|\vF\pm\frac{i\pi}{2}\right|\ge\frac{\SemiRad}{2}$}\}.
$$
For $\vF\in\DomainD$, we easily find that $\tanh\vF$ is bounded by a $\kI$-independent bound, since it is analytic in $\DomainD$ and $\tanh\vF\rightarrow\pm 1$ as $\vF\rightarrow\infty$ within $\DomainD$. Let that bound be $\ConstC_{7}$ so that
$$
|\tanh\vF|\le\ConstC_{7}\qquad\text{for $\vF\in\DomainD$.}
$$
(It can be checked that $\tanh\vF$ attains its maximum at $\vF=\pm\frac{i\pi}{2}\pm\frac{\SemiRad}{2}$ so we can set $\ConstC_{7}=\coth(\frac{\SemiRad}{2})$, but the precise value of $\ConstC_{7}$ is not important.) 
Also, 
$$
|\uF|\ge \frac{\pi-\SemiRad}{2}\,.
$$
Then, from \eqref{eqn:genFcIntZfL} we find an upper bound on $\genFcD/\cosh^2\uF$:
\be\label{eqn:DefConstC6}
\ConstC_{6}\eqdef
\frac{2}{\pi-\SemiRad}
\left(
\frac{|\sA|}{\Re\sA}+1
\right)
\left(|\sA+1||\xc_1|\ConstC_{7}+2\ConstC_{7}^3\right)
>
\left|\frac{\genFcD(\uF)}{\cosh^2\uF}\right|
\,\quad
\text{for $\uF$ on $\PathP_1'$.}
\ee
This is what we claimed in \eqref{eqn:BoundOnFcDOverCoshOnP1}.

Next, we need an upper bound on $1/\tanh^{2\kI-1}\uF$ in \eqref{eqn:PathP1Only}. At this point we restrict to $\uF$ along $\PathP_3'$.
Along that path $|\tanh\uF|$ has local minima at $\uF=i(\frac{\pi}{2}-\SemiRad)$ and at $\uF=\frac{i\pi}{2}\pm\betaC_\kI$. The value at $i(\frac{\pi}{2}-\SemiRad)$ is
$$
\left|\tanh\left\lbrack i\left(\frac{\pi}{2}-\SemiRad\right)\right\rbrack\right|=
\cot\SemiRad>1,
$$
and the value at $\frac{i\pi}{2}\pm\betaC_\kI$ is
$$
\left|\tanh\left(\frac{i\pi}{2}\pm\betaC_\kI\right)\right|
=1+\frac{2}{e^{2\betaC_\kI}-1}>1.
$$
For sufficiently large $\betaC_\kI$ (and hence for sufficiently large $\kI$) the value at $\frac{i\pi}{2}\pm\betaC_\kI$ is the global minimum. Therefore, since $e^{2\betaC_\kI}\ge 1$, 
\be\label{eqn:BoundTanhku}
\frac{1}{|\tanh^{2\kI-1}\uF|}<\left(1+\frac{2}{e^{2\betaC_\kI}}\right)^{1-2\kI}
<\exp\left\lbrack
-(2\kI-1) e^{-2\betaC_\kI}
\right\rbrack,
\ee
since $0<e^{-2\betaC_\kI}<1$ and $e^x<1+\frac{2}{x}$ for $0<x<1$.
Thus, for sufficiently large $\kI$,
\bear
\lefteqn{
\left|
\int_{\PathP_3'}
\frac{\genFcD(\uF)d\uF}{\cosh^2\uF\tanh^{2\kI-1}\uF}
\right|
\le
\ConstC_{6}
\exp\left\lbrack
-(2\kI-1) e^{-2\betaC_\kI}
\right\rbrack\int_{\PathP_3'}|d\uF|
}\nn\\
&<&
\ConstC_{6}
\exp\left\lbrack
-(2\kI-1) e^{-2\betaC_\kI}
\right\rbrack
(\pi\SemiRad+2\betaC_\kI)
<
4\ConstC_{6}
\exp\left\lbrack
-(2\kI-1) e^{-2\betaC_\kI}
\right\rbrack\betaC_\kI,
\nn\\ &&
\label{eqn:PathP1OnlyBound}
\nn
\eear
where we have used \eqref{eqn:BoundOnFcDOverCoshOnP1} and \eqref{eqn:BoundTanhku}.

We now take (for $\kI>1$),
\be\label{eqn:betaCvalue}
\betaC_\kI=\frac{1}{2}\log\left\lbrack\frac{\kI-1}{\log(\kI-1)}\right\rbrack
\ee
so that for $\kI>3$,
$$
\exp\left\lbrack
-(2\kI-1) e^{-2\betaC_\kI}
\right\rbrack\betaC_\kI
<
\frac{\log(\kI-1)}{2(\kI-1)^2}<\frac{\log\kI}{\kI^2}\,.
$$
Then, for sufficiently large $\kI$, we find from \eqref{eqn:PathP1OnlyBound}
\be\label{eqn:BoundPathP3}
\left|
\int_{\PathP_3'}
\frac{\genFcD(\uF)d\uF}{\cosh^2\uF\tanh^{2\kI-1}\uF}
\right|
<4\ConstC_{6}
\frac{\log\kI}{\kI^2}
\,.
\ee
We will substitute this bound into \eqref{eqn:PathP1BreakUp}, as part of the total bound for \eqref{eqn:PathP1Only}, but first we need to find a similar bound on the integrals $\int_{\frac{i\pi}{2}-\infty}^{\frac{i\pi}{2}-\betaC_\kI}$ and $\int_{\frac{i\pi}{2}+\betaC_\kI}^{\frac{i\pi}{2}+\infty}$ in \eqref{eqn:PathP1BreakUp}.

% - - - - - - - - - - - - - - - - - - - - - - - - - - - - - - - - - - -
\subsection{A bound for large $\uF$}

From \eqref{eqn:genFcDerivativExtendedC}, with $\tF=\tanh\uF$, we have
\bear
\lefteqn{
\frac{\genFcD}{\cosh^2\uF} =
\frac{\sA}{\uF^{\sA+1}}(\sA+1)\xc_1
\int_\uF^\infty\vF^{\sA-1}\left(\tanh\vF-1\right)d\vF
}\nn\\ &&
-\frac{2\sA}{\uF^{\sA+1}}
\int_\uF^\infty\vF^{\sA-1}\left(\tanh^3\vF-1\right)d\vF
+\frac{2(1-\tanh^3\uF)-(\sA+1)(1-\tanh\uF)\xc_1}{\uF}
\,.
\nonumber\\ &&
\label{eqn:genFcDerivativExtendedD}
\eear
 Along $\int_{\frac{i\pi}{2}-\infty}^{\frac{i\pi}{2}-\betaC_\kI}$ and $\int_{\frac{i\pi}{2}+\betaC_\kI}^{\frac{i\pi}{2}+\infty}$ we set $\uF=\ReuF+\frac{i\pi}{2}$ with $\ReuF\in\R$. To put a bound on the integrand of \eqref{eqn:PathP1Only}, we note that as $\ReuF\rightarrow\infty$ we have $\tanh\uF\rightarrow 1$, and at the same time \eqref{eqn:genFcDerivativExtendedD}
becomes small, since the factors $(\tanh\vF-1)$ and $(\tanh^3\vF-1)$ are small.

For $\vF = \frac{i\pi}{2}+\RevF$ we have $\tanh\vF = \coth\RevF$, and in the range $\RevF>\betaC_\kI$, if we also assume $\log 2<\betaC_\kI$, we have
\be\label{eqn:tanhv1Fineq}
\left|\tanh\vF-1\right| = \frac{2 e^{-2\RevF}}{1-e^{-2\RevF}}<\tfrac{8}{3} e^{-2\RevF},
\ee

\be\label{eqn:tanhv3Fineq}
\left|\tanh^3\vF-1\right| = \frac{2 e^{-2\RevF}}{1-e^{-2\RevF}}(1+\coth\RevF+\coth^2\RevF)<\frac{392}{27} e^{-2\RevF}\,,
\ee

and
$$
|\vF|=\left(\frac{\pi^2}{4}+\RevF^2\right)^{\frac{1}{2}}<
\RevF\sqrt{1+\frac{\pi^2}{4(\log 2)^2}}<3\RevF,
\qquad
|\vF^{\sA-1}|<3^{\Re\sA-1}e^{\frac{1}{2}\pi\left|\Im\sA\right|}\RevF^{\Re\sA-1}.
$$
Therefore, for $\uF = \frac{i\pi}{2}+\ReuF$, and $\log 2<\betaC_\kI<\ReuF\le\RevF$ as above, we have
\bear
\left|\int_\uF^{\infty}\vF^{\sA-1}(\tanh\vF-1)d\vF\right|
&<&\ConstC_{8}\int_\ReuF^\infty\RevF^{\Re\sA-1}e^{-2\RevF}d\RevF
\nn\\ &&
=2^{-\Re\sA}\ConstC_{8}\Gamma(\Re\sA,2\ReuF)
<\ConstC_{9}\ReuF^{\Re\sA-1}e^{-2\ReuF},
\nn\\
&&\label{eqn:inttanh1vFBound}
\eear
where $\Gamma(\alpha,x)$ is the incomplete Gamma function whose asymptotic form is $x^{\alpha-1}e^{-x}$ for large $x$,
and $\ConstC_{8},\ConstC_{9}$ are computable (possibly $\sA$-dependent) constants whose precise value is not important.

We have a similar bound for the integral with $\tanh^3\vF-1$:
\bear
\left|\int_\uF^{\infty}\vF^{\sA-1}(\tanh^3\vF-1)d\vF\right|
&<&\ConstC_{8}\int_\ReuF^\infty\RevF^{\Re\sA-1}e^{-2\RevF}d\RevF
\nn\\ &&
=2^{-\Re\sA}\ConstC_{8}\Gamma(\Re\sA,\ReuF)
<\ConstC_{9}\ReuF^{\Re\sA-1}e^{-2\ReuF},
\nn\\
&&
\label{eqn:inttanh3vFBound}
\eear
where $\ConstC_{8},\ConstC_{9}$ are assumed big enough to be reused.
The inequalities \eqref{eqn:tanhv1Fineq} and \eqref{eqn:tanhv3Fineq} also allow us to write
\be\label{eqn:ThirdTermBound}
\left|\frac{2(1-\tanh^3\uF)-(\sA+1)(1-\tanh\uF)\xc_1}{\uF}\right|
<\ConstC_{10} \frac{e^{-2\ReuF}}{\ReuF}\,,
\ee
for some $\sA$-dependent constant $\ConstC_{10}$.

 Thus, from \eqref{eqn:genFcDerivativExtendedD}, combined with \eqref{eqn:inttanh1vFBound}, \eqref{eqn:inttanh3vFBound}, and \eqref{eqn:ThirdTermBound}, as well as
$$
\frac{1}{|\uF^{\sA+1}|}<\frac{1}{\ReuF^{\Re\sA+1}}e^{\frac{\pi}{2}|\Im\sA|},
$$
we get
\bear
\lefteqn{
\left|\frac{\genFcD\left(\frac{i\pi}{2}+\ReuF\right)}{\cosh^2\left(\frac{i\pi}{2}+\ReuF\right)}\right| <
\ConstC_{10}\frac{e^{-2\ReuF}}{\ReuF}+
}\nonumber\\
&&
\left|\frac{\sA}{\uF^{\sA+1}}\right|
\left|
(\sA+1)\xc_1\int_{\frac{i\pi}{2}+\ReuF}^{\infty}\vF^{\sA-1}\left(\tanh^3\vF-1\right)d\vF
-2\int_{\frac{i\pi}{2}+\ReuF}^{\infty}\vF^{\sA-1}\left(\tanh\vF-1\right)d\vF
\right|
\nonumber\\
&<& 2\ConstC_{9}|\sA|\frac{e^{-2\ReuF}}{\ReuF^2}e^{\frac{\pi}{2}|\Im\sA|}
+\ConstC_{10} \frac{e^{-2\ReuF}}{\ReuF}<
\ConstC_{11}\frac{e^{-2\ReuF}}{\ReuF}
\nn\\
&&
\label{eqn:BoundgenFcLargeuF}
\eear
where $\ConstC_{11}$ is an $\sA$-dependent constant.
We also have $|\coth\uF|=|\tanh\ReuF|<1$, and using \eqref{eqn:BoundgenFcLargeuF}, we can now put the bound
\bear
\left|
\int_{\frac{i\pi}{2}+\betaC_\kI}^{\frac{i\pi}{2}+\infty}
\frac{\genFcD(\uF)d\uF}{\cosh^2\uF\tanh^{2\kI-1}\uF}
\right|
&<&
\int_{\betaC_\kI}^{\infty}
\left|
\frac{\genFcD(\uF)}{\cosh^2\uF}
\right|d\ReuF
<\ConstC_{11}
\int_{\beta_\kI}^\infty\frac{e^{-2\ReuF} d\ReuF}{\ReuF}
\nonumber\\
&=&
\ConstC_{11}\Gamma(0,2\betaC_\kI)
<\frac{\ConstC_{11}}{2\betaC_\kI}e^{-2\betaC_\kI},
\nonumber
\eear
where we used the bound $\Gamma(0,x)<e^{-x}/x$ for $x>0$.
Using \eqref{eqn:betaCvalue} this becomes
\bear
\left|
\int_{\frac{i\pi}{2}+\betaC_\kI}^{\frac{i\pi}{2}+\infty}
\frac{\genFcD(\uF)d\uF}{\cosh^2\uF\tanh^{2\kI-1}\uF}
\right|
&<&\frac{\ConstC_{11}\log(\kI-1)}{(\kI-1)\lbrack\log(\kI-1)-\log\log(\kI-1)\rbrack}
<\frac{\ConstC_{12}}{\kI}\,,
\nonumber\\ &&
\label{eqn:BoundIntFcToInfinityWithBeta}
\eear
where (like $\ConstC_{11}$) $\ConstC_{12}$ is an $\sA$-dependent constant whose precise value is not important for us, and we assume $\kI\ge 3$.

For the integral $\int_{\frac{i\pi}{2}-\infty}^{\frac{i\pi}{2}-\betaC_\kI}$ the bound is similar, except that we need to replace \eqref{eqn:genFcDerivativExtendedD} by
\bear
\lefteqn{
\frac{\genFcD(\uF)}{\cosh^2\uF} =\frac{\genFcD(-\uF)}{\cosh^2\uF} =
\frac{\sA}{(-\uF)^{\sA+1}}(\sA+1)\xc_1
\int_{-\uF}^\infty\vF^{\sA-1}\left(\tanh\vF-1\right)d\vF
}\nn\\ &&
-\frac{2\sA}{(-\uF)^{\sA+1}}
\int_{-\uF}^\infty\vF^{\sA-1}\left(\tanh^3\vF-1\right)d\vF
+\frac{(\sA+1)(1+\tanh\uF)\xc_1-2(1+\tanh^3\uF)}{\uF}
\,.
\nonumber\\ &&
\label{eqn:genFcDerivativExtendedDB}
\eear

% - - - - - - - - - - - - - - - - - - - - - - - - - - - - - - - - - - -
\subsection{Combining the bounds}

Combining \eqref{eqn:BoundPathP3} and \eqref{eqn:BoundIntFcToInfinityWithBeta} into \eqref{eqn:PathP1BreakUp} we get the requisite bound on \eqref{eqn:PathP1Only},

$$
\left|\int_{\PathP_1'}
\frac{\genFcD(\uF)d\uF}{\cosh^2\uF\tanh^{2\kI-1}\uF}
\right|<\frac{\ConstC_{13}}{\kI}\,.
$$
for any $\ConstC_{13}>\ConstC_{12}$ and sufficiently large $\kI$.
Therefore \eqref{eqn:xckContourI} gives $\xc_\kI=O(\frac{1}{\kI})$.

\end{proof}

It is convenient to have the following identity before proving the second part of \thmref{thm:main}.

\begin{lemma}
For $\nI\ge 0$,
\be\label{eqn:sumckuptomInf}
\sum_{\mI=0}^{\nI}\frac{\xc_{\nI+1-\kI}}{2\kI+1}
-\sA\sum_{\kI=\nI+2}^\infty\frac{\xc_\kI}{2\kI-1}
=2\delta_{0\nI}-2.
\ee
\end{lemma}

\begin{proof}
For $\mI>0$, by taking the sum of the first $(\mI-1)$ relations in \eqref{eqn:xcrl} and collecting the terms we get
\be\label{eqn:sumckuptom}
(\sA+1)\xc_1-2+2\delta_{\mI 1}
-\sA\sum_{\kI=1}^{\mI}\frac{\xc_\kI}{2\kI-1}
=\sum_{\kI=1}^{\mI}\frac{\xc_\kI}{2\mI-2\kI+1}\,.
\ee
Taking the $\mI\rightarrow\infty$ limit of \eqref{eqn:sumckuptom} we get
\be\label{eqn:sumckllim}
\sA\sum_{\kI=1}^{\infty}\frac{\xc_\kI}{2\kI-1}
=(\sA+1)\xc_1
-\lim_{\mI\rightarrow\infty}\sum_{\kI=1}^{\mI}\frac{\xc_\kI}{2\mI-2\kI+1}
=(\sA+1)\xc_1\,.
\ee
For the last part we used the fact that $|\xc_\kI|<\ConstC_{14}/\kI$ for some constant $\ConstC_{14}$, and so
\bear
\lefteqn{
\left|\sum_{\kI=1}^{\mI}\frac{\xc_\kI}{2\mI-2\kI+1}\right|
\le\sum_{\kI=1}^{\mI}\frac{|\xc_\kI|}{2\mI-2\kI+1}
<\ConstC_{14}\sum_{\kI=1}^{\mI}\frac{1}{(2\mI-2\kI+1)\kI}
}\nn\\
&&
=\frac{2\ConstC_{14}}{2\mI+1}\sum_{\kI=1}^{\mI}\left(\frac{1}{2\mI-2\kI+1}+\frac{1}{2\kI}\right)
<\frac{4\ConstC_{14}\log(2\mI+1)}{2\mI+1}\xrightarrow{\mI\rightarrow\infty}0.
\nn
\eear
Adding \eqref{eqn:sumckllim} to \eqref{eqn:sumckuptom} we get \eqref{eqn:sumckuptomInf}, after setting $\mI=\nI+1$.

\end{proof}

We are now ready to prove the sufficient condition of \thmref{thm:main} for the case $\xb_1\neq 0$.
\begin{proposition}
Let $\sA$ be in the range $1>\Re\sA>0$ and $\zeta(\sA)\neq 0$. Let $1< p<\infty$.
Suppose also that $\seqxb$ is a sequence of complex numbers such that the sequence of differences $\xd_\kI\eqdef\xb_{\kI+1}-\xb_\kI$ is in the Banach space $\lBanach^p$. Then, if the recursion relation \eqref{eqn:OpWseq} (which is convergent by our assumptions) holds, then $\xb_1=0$.
\end{proposition}

\begin{proof}
Let $\lBanach^q$ be the dual Banach space with $\frac{1}{p}+\frac{1}{q}=1$. Let $\seqxc$ be as in \propref{prop:xcOk}. Then \eqref{eqn:xcrl} holds, and by \propref{prop:xcOk} we have $\xc_\kI=O(\frac{\log\kI}{\kI})$.

By \eqref{eqn:OpWxdnob} we have
$$
0=
-\frac{\sA}{2\kI-1}\sum_{\nI=0}^{\kI-2}\xd_\nI
+\sum_{\nI=0}^\infty\frac{\xd_{\kI+\nI-1}}{2\nI+1}\,.
$$
These relations were formally derived by linear operations on the columns of the matrix in \eqref{eqn:OpW}, and the relation for $\kI$ corresponds to the $\kI^{th}$ row of that matrix. We will now take linear combinations of these expressions with coefficients $\xc_\kI$.
Consider the sum
\be\label{eqn:sumxcxd}
0=\sum_{\kI=1}^\infty
\xc_\kI
\left\lbrack
-\frac{\sA}{2\kI-1}\sum_{\nI=0}^{\kI-2}\xd_\nI
+\sum_{\nI=0}^\infty\frac{\xd_{\kI+\nI-1}}{2\nI+1}
\right\rbrack.
\ee
We first need to show that the RHS of \eqref{eqn:sumxcxd} is absolutely convergent since we will need to rearrange terms. Note that \eqref{eqn:OpWxdnob} is equivalent to \eqref{eqn:start}, and the $\sA$-dependent matrix there can be decomposed into upper and lower triangular matrices and written as $\OpWUT-\sA\OpWLT$, where $\OpWLT,\OpWUT$ are constant matrices given by:
\be\label{eqn:OpWULTdef}
\OpWLT\eqdef
\begin{pmatrix}
0 & 0 & 0 & 0 & 0 & \cdots \\
 & & & & & \\
\frac{1}{3} &  0 & 0 & 0 & 0 & \cdots \\
 & & & & & \\
\frac{1}{5} & \frac{1}{5}  &  0 & 0 & 0 & \cdots \\
 & & & & & \\
\frac{1}{7} &\frac{1}{7}  & \frac{1}{7}  &  0 & 0 & \cdots \\
 & & & & & \\
\frac{1}{9} & \frac{1}{9} & \frac{1}{9} & \frac{1}{9} &  0 & \cdots \\
% & & & & & \\
\vdots &\vdots & \vdots &\vdots & \vdots & \ddots\\
\end{pmatrix}
\,,\qquad
\OpWUT\eqdef
\begin{pmatrix}
1 & \frac{1}{3} & \frac{1}{5} & \frac{1}{7} &  \frac{1}{9} & \cdots \\
 & & & & & \\
0 &  1 & \frac{1}{3} & \frac{1}{5} & \frac{1}{7} & \cdots \\
 & & & & & \\
0 & 0  &  1 & \frac{1}{3} & \frac{1}{5} & \cdots \\
 & & & & & \\
0 &0  & 0  &  1 &  \frac{1}{3} & \cdots \\
 & & & & & \\
0 & 0 & 0 & 0 &  1 & \cdots \\
% & & & & & \\
\vdots &\vdots & \vdots &\vdots & \vdots & \ddots\\
\end{pmatrix}\,.
\ee
It is a consequence of Hardy's inequality, which puts a bound on the norm of the operator defined by $\left(\xd_\nI\right)_{\nI=1}^\infty\mapsto\left(\frac{1}{\nI}\sum_{\kI=1}^\nI\xd_\kI\right)_{\nI=1}^\infty$ on $\lBanach^p$ (see for instance \cite{Hardy} and references therein), that $\OpWLT$ is a bounded operator on $\lBanach^p$. 
Let $\|\OpWLT\|_p$ be its norm. Thus, by definition,
$$
\left\lbrack\sum_{\kI=2}^\infty\left(\frac{1}{2\kI-1}\sum_{\nI=0}^{\kI-2}|\xd_\nI|\right)^p
\right\rbrack^{\frac{1}{p}}
\le\|\OpWLT\|_p\left(\sum_{\kI=0}^\infty|\xd_\kI|^p\right)^{\frac{1}{p}}<\infty.
$$
Then, since $\seqxc\in\lBanach^q$, it follows from H\"older's inequality that
\be\label{eqn:HardyAndHolder}
\sum_{\kI=2}^\infty
\frac{|\xc_\kI|}{2\kI-1}\sum_{\nI=0}^{\kI-2}|\xd_\nI|
\le
\left(\sum_{\kI=1}^\infty|\xc_\kI|^q\right)^{\frac{1}{q}}
\left\lbrack\sum_{\kI=2}^\infty\left(\frac{1}{2\kI-1}\sum_{\nI=0}^{\kI-2}|\xd_\nI|\right)^p
\right\rbrack^{\frac{1}{p}}<\infty.
\ee
Since we need to show that 
$$
\sum_{\kI=1}^\infty
|\xc_\kI|
\left\lbrack
\frac{|\sA|}{2\kI-1}\sum_{\nI=0}^{\kI-2}|\xd_\nI|
+\sum_{\nI=0}^\infty\frac{|\xd_{\kI+\nI-1}|}{2\nI+1}
\right\rbrack<\infty,
$$
it only remains to check that $\sum_{\kI=1}^\infty\sum_{\nI=0}^\infty\frac{|\xc_\kI||\xd_{\kI+\nI-1}|}{2\nI+1}<\infty$.
This can be seen as follows. 
Since $\xc_\kI=O\left(\frac{1}{\kI}\right)$, we have $|\xc_\kI|<\ConstC_{15}/(2\kI+1)$ for some $\ConstC_{15}$,  and therefore
\bear
\lefteqn{
\sum_{\kI=1}^\infty\sum_{\nI=0}^\infty\frac{|\xc_\kI||\xd_{\kI+\nI-1}|}{2\nI+1}
<
\ConstC_{15}\sum_{\kI=1}^\infty\sum_{\nI=0}^\infty\frac{|\xd_{\kI+\nI-1}|}{(2\nI+1)(2\kI+1)}
}\nn\\
&=&\ConstC_{15}\sum_{\mI=0}^\infty\sum_{\nI=0}^{\mI}\frac{|\xd_{\mI}|}{(2\nI+1)(2\mI-2\nI+3)}
\nn\\
&=&
\frac{1}{2}\ConstC_{15}\sum_{\mI=0}^\infty\left\lbrack\frac{|\xd_\mI|}{\mI+2}\sum_{\nI=0}^\mI
\left(\frac{1}{2\nI+1}+\frac{1}{2\mI-2\nI+3}\right)\right\rbrack
\nn\\
&<&
\ConstC_{15}\sum_{\mI=0}^\infty\frac{|\xd_\mI|\log(2\mI+3)}{\mI+2}
<\ConstC_{15}\left(\sum_{\mI=0}^\infty\frac{[\log(2\mI+3)]^{q}}{(\mI+2)^q}
\right)^{\frac{1}{q}}\left(\sum_{\mI=0}^\infty|\xd_\mI|^p\right)^{\frac{1}{p}}<\infty\,.
\nn
\eear
We can therefore rearrange the terms in \eqref{eqn:sumxcxd} and collect the coefficients of $\xd_\nI$ to get
\be\label{eqn:sumxcxdR}
0=\sum_{\nI=0}^\infty
\xd_\nI
\left\lbrack
-\sA\sum_{\kI=\nI+2}^\infty\frac{\xc_\kI}{2\kI-1}
+\sum_{\mI=0}^{\nI}\frac{\xc_{\nI+1-\mI}}{2\mI+1}
\right\rbrack
\sum_{\nI=1}^\infty
\xd_\nI = -\xb_1,
\ee
where we used \eqref{eqn:sumckuptomInf} and then \eqref{eqn:xdDef}. We get $\xb_1=0$, which is a contradiction.

\end{proof}

So far we proved the second part of \thmref{thm:main} for the case $\xb_1\neq 0.$
Now, suppose $\seqxb$ is as in \thmref{thm:main} and $\xb_\jI\neq 0$ for some other positive integer $\jI$. Then, we can modify \eqref{eqn:xcrl} to
\be\label{eqn:xcrljA}
\delta_{\kI\jI}=
-\left(\frac{\sA}{2(2\kI+1)}+\frac{1}{2}\right)\xc_{\kI+1}
+\sum_{n=0}^{\kI-1}\frac{\xc_{\kI-\nI}}{(2\nI+1)(2\nI+3)}\,,
\ee
and \eqref{eqn:genFcODE} is modified to
\be\label{eqn:genFcODEj}
0 =\tF^{2\jI}
+\frac{\sA}{2}\left\lbrack\frac{1}{\tF}\genFc(\tF)-\xc_1\right\rbrack
-\frac{1}{2}\xc_1
-\frac{1}{4\tF}\genFc'(\tF)(\tF^2-1)\log\left(\frac{1+\tF}{1-\tF}\right).
\ee
Therefore, \eqref{eqn:genFcInt} is modified to
\be\label{eqn:genFcIntj}
\genFc=\frac{1}{\uF^\sA}\int_0^\uF\left\lbrack
(\sA+1)\xc_1\tanh\vF
-2\tanh^{2\jI+1}\vF
\right\rbrack\vF^{\sA-1}d\vF
+\frac{\ConstF}{\uF^\sA}\,,
\ee
and \eqref{eqn:xcrl} is modified to
\be\label{eqn:xcrljB}
\delta_{\kI\jI}=
-\left(\frac{\sA}{2(2\kI+1)}+\frac{1}{2}\right)\xc_{\kI+1}
+\sum_{n=0}^{\kI-1}\frac{\xc_{\kI-n}}{(2n+1)(2n+3)}\,,
\qquad
\kI=1,2,\cdots
\ee
and the rest of the argument carries on as for $\jI=1$.

As a side note, when acting on $\lBanach^p$, $\OpWUT$ [defined in \eqref{eqn:OpWULTdef}] does not always produce a sequence in $\lBanach^p$ for $p>1$. For example, consider the sequence
$$
\seqV\eqdef\left(\frac{1}{(2\kI+1)^{\frac{1}{p}}\log(2\kI+1)}\right)_{\kI=0}^\infty.
$$
Then $\|\seqV\|_p<\infty$ while $\|\OpWUT\seqV\|_p=\infty$, since we can put a lower bound on $\|\OpWUT\seqV\|_p^p$ as follows,
\bear
\lefteqn{
\sum_{\kI=1}^\infty\left\lbrack
\sum_{\nI=0}^\infty\frac{1}{2\nI+1}\left(\frac{1}{(2\kI+2\nI+1)^{\frac{1}{p}}\log(2\kI+2\nI+1)}\right)
\right\rbrack^p
}
\nn\\
&>&
\sum_{\kI=1}^\infty\left\lbrack
\sum_{\nI=0}^\kI\frac{1}{(2\nI+1)(2\kI+2\nI+1)^{\frac{1}{p}}\log(2\kI+2\nI+1)}
\right\rbrack^p
\nn\\
&>&
\sum_{\kI=1}^\infty\left\lbrack
\sum_{\nI=0}^\kI\frac{1}{2\nI+1}\left(\frac{1}{(4\kI+2)^{\frac{1}{p}}\log(4\kI+2)}\right)
\right\rbrack^p
\nn\\
&=&
\sum_{\kI=1}^\infty\left\lbrack
\sum_{\nI=0}^\kI\frac{1}{2\nI+1}\right\rbrack^p\frac{1}{(4\kI+2)\left\lbrack\log(4\kI+2)\right\rbrack^p}=\infty.
\nn
\eear

Note also that the space of solutions $\seqxb$ to \eqref{eqn:OpW} is at most $1$-dimensional.
This is because we can modify \eqref{eqn:xcrl} to
\be\label{eqn:xcrlmodifiedj}
\mu_j(\sA)\delta_{\kI 1}-\delta_{\kI\jI}=
-\left(\frac{\sA}{2(2\kI+1)}+\frac{1}{2}\right)\xc_{\kI+1}
+\sum_{\nI=0}^{\kI-1}\frac{\xc_{\kI-n}}{(2\nI+1)(2\nI+3)}\,,
\qquad
\kI=1,2,\dots
\ee
where $\mu_\jI(\sA)$ are computable from $\sA$ (and are the ratios $\xb_\jI/\xb_1$ in the solution we have established already), and then \eqref{eqn:genFcInt} can be modified to
\be\label{eqn:genFcIntModifiedjA}
\genFc=\frac{1}{\uF^\sA}\int_0^\uF\left\lbrack
(\sA+1)\xc_1\tanh\vF
-2\mu_j(\sA)\tanh^3\vF
+2\tanh^{2j+1}\vF
\right\rbrack\vF^{\sA-1}d\vF
+\frac{\ConstF}{\uF^\sA}\,.
\ee
Setting $\xc_1=0$ we can adjust $\mu_\jI(\sA)$ so that
\be\label{eqn:genFcIntModifiedjB}
\genFc=\frac{1}{\uF^\sA}\int_0^\infty\left\lbrack
2\tanh^{2\jI+1}\vF
-2\mu_\jI(\sA)\tanh^3\vF
\right\rbrack\vF^{\sA-1}d\vF
=0.
\ee
Using the resulting $\seqxc$ as before shows that $\xb_j=\mu_j(\sA)\xb_1$ for any solution.

% -----------------------------------------------------------------------------------
\section{Non simple zeros of $\zeta$}
\label{sec:NonSimpleZeros}

There is a fairly simple augmentation of \eqref{eqn:OpW} for the case of a non-simple zero $\sA$ (if it exists).
Let
\be\label{eqn:defOpA}
\matK\eqdef
\begin{pmatrix}
-1 & \frac{2\cdot 1}{1\cdot 3} & \frac{2\cdot 1}{3\cdot 5} & \frac{2\cdot 1}{5\cdot 7} & \frac{2\cdot 1}{7\cdot 9} & \cdots \\
&  & & & & \\
0 & -3 &  \frac{2\cdot 3}{1\cdot 3} & \frac{2\cdot 3}{3\cdot 5} & \frac{2\cdot 3}{5\cdot 7} & \cdots \\
&  & & & & \\
0& 0 & -5  &  \frac{2\cdot 5}{1\cdot 3} & \frac{2\cdot 5}{3\cdot 5} & \cdots \\
&  & & & & \\
0& 0 &0  & -7  &  \frac{2\cdot 7}{1\cdot 3} & \cdots \\
&  & & & & \\
0& 0 & 0 & 0 & -9 & \cdots \\
&  & & & & \\
\vdots &\vdots & \vdots &\vdots & \ddots\\
\end{pmatrix}\,.
\ee
Also, let
\be\label{eqn:defvecxbA}
\vecxb\eqdef
\begin{pmatrix}
\xb_1 \\ \\ \xb_2 \\ \\ \xb_3 \\ \\ \xb_4 \\ \\ \vdots \\
\end{pmatrix},\qquad
\vecxb'\eqdef\frac{d}{d\sA}
\begin{pmatrix}
\xb_1 \\ \\ \xb_2 \\ \\ \xb_3 \\ \\ \xb_4 \\ \\ \vdots \\
\end{pmatrix}\,.
\ee
Assuming $\zeta(\sA)=0$, we can rewrite \eqref{eqn:OpW} as
\be\label{eqn:Kb}
\matK\vecxb=\sA\vecxb,
\ee
and assuming $\zeta'(\sA)=0$, we can take the derivative of \eqref{eqn:OpWgen}, using $\xb_0'=0$, to obtain
\be\label{eqn:Kbp}
\matK\vecxb'=\sA\vecxb'+\vecxb.
\ee
We also define
\be\label{eqn:defvecxbB}
\vecxd\eqdef
\begin{pmatrix}
\xd_0 \\ \\ \xd_1 \\ \\ \xd_2 \\ \\ \xd_3 \\ \\ \vdots \\
\end{pmatrix},\qquad
\vecxd'\eqdef\frac{d}{d\sA}
\begin{pmatrix}
\xd_0 \\ \\ \xd_1 \\ \\ \xd_2 \\ \\ \xd_3 \\ \\ \vdots \\
\end{pmatrix}\,.
\ee
Then, setting 
$$
\xb_\kI=\sum_{\nI=0}^{\kI-1}\xd_\kI,\qquad\kI=1,2,\dots,
$$
and with the definitions of $\matU$ and $\matL$ in \eqref{eqn:OpWULTdef}, we rewrite \eqref{eqn:Kb} and \eqref{eqn:Kbp} as
\be\label{eqn:ABd}
\matU\vecxd=\sA\matL\vecxd,\qquad
\matU\vecxd'=\sA\matL\vecxd'+\matL\vecxd.
\ee
Moreover, analysis similar to the proof of \propref{prop:xdkIAsymp}
%(with $|\log x|$ inserted after taking the $\sA$ derivatives of the integrals of $x^{-\sA}$)
shows that $\xd_\kI'(\sA)=O(\frac{1}{\kI}\log\kI)$. Thus, if $\sA$ is a non-simple zero there exist $\vecxd,\vecxdt\in\lBanach^p$ solving
\be\label{eqn:ABlp}
\matU\vecxd=\sA\matL\vecxd,\qquad
\matU\vecxdt=\sA\matL\vecxdt+\matL\vecxd.
\ee
We will not attempt to prove a reverse of this statement, but we expect that an argument along the lines that led to \secref{sec:sufpf} would work.

% ======================================================================
\section{Relation to the Hilbert-P\'olya program}
\label{sec:HilbertPolya}

As a condition on $\seqxd$, \eqref{eqn:start} is equivalent to
\be\label{eqn:OpWxdMC}
\begin{pmatrix}
-\frac{\sZ+1}{4} & \frac{2}{1\cdot 3} & \frac{3}{3\cdot 5} & \frac{4}{5\cdot 7} &  \frac{5}{7\cdot 9} & \cdots \\
 & & & & & \\
0 &  -\frac{\sZ+3}{4} & \frac{3}{1\cdot 3} & \frac{4}{3\cdot 5} & \frac{5}{5\cdot 7} & \cdots \\
 & & & & & \\
0 & 0  &  -\frac{\sZ+5}{4} & \frac{4}{1\cdot 3} & \frac{5}{3\cdot 5} & \cdots \\
 & & & & & \\
0 &0  & 0  &  -\frac{\sZ+7}{4} &  \frac{5}{1\cdot 3} & \cdots \\
 & & & & & \\
0 & 0 & 0 & 0 &  -\frac{\sZ+9}{4} & \cdots \\
% & & & & & \\
\vdots &\vdots & \vdots &\vdots & \vdots & \ddots\\
\end{pmatrix}
\begin{pmatrix}
\xd_0 \\ \\ \xd_1 \\ \\ \xd_2 \\ \\ \xd_3 \\ \\  \vdots \\  \vdots \\
\end{pmatrix}
=0,
\ee
together with
\be\label{eqn:sumxd2kp1}
0=\sum_{\kI=0}^\infty\frac{\xd_\kI}{2\kI+1}.
\ee
To see this, note that \eqref{eqn:sumxd2kp1} is exactly the first element of the vector \eqref{eqn:start}, while \eqref{eqn:OpWxdMC} is obtained by row operations on the matrix of \eqref{eqn:start} --- subtracting $(2\kI-1)$ times the $\kI^{th}$ row from $(2\kI-3)$ times the $(\kI-1)^{th}$ row, and dividing by $4$, for $\kI=2,3,\dots$.

Defining
$$
\vecxd\eqdef
\begin{pmatrix}
\xd_0 \\ \\ \xd_1 \\ \\ \xd_2 \\ \\ \xd_3 \\ \\ \vdots \\  \vdots \\
\end{pmatrix},\qquad
\vecf\eqdef
\begin{pmatrix}
1 \\ \\ \frac{1}{3} \\ \\ \frac{1}{5} \\ \\ \frac{1}{7} \\ \\  \frac{1}{9} \\ \vdots \\
\end{pmatrix},
$$
and
\be\label{eqn:matAXdef}
\matAX\eqdef
\begin{pmatrix}
-\frac{3}{8} & \frac{2}{1\cdot 3} & \frac{3}{3\cdot 5} & \frac{4}{5\cdot 7} &  \frac{5}{7\cdot 9} & \cdots \\
 & & & & & \\
0 &  -\frac{7}{8} & \frac{3}{1\cdot 3} & \frac{4}{3\cdot 5} & \frac{5}{5\cdot 7} & \cdots \\
 & & & & & \\
0 & 0  &  -\frac{11}{8} & \frac{4}{1\cdot 3} & \frac{5}{3\cdot 5} & \cdots \\
 & & & & & \\
0 &0  & 0  &  -\frac{15}{8} &  \frac{5}{1\cdot 3} & \cdots \\
 & & & & & \\
0 & 0 & 0 & 0 &  -\frac{19}{8} & \cdots \\
% & & & & & \\
\vdots &\vdots & \vdots &\vdots & \vdots & \ddots\\
\end{pmatrix},
\ee
\eqref{eqn:OpWxdMC} and \eqref{eqn:sumxd2kp1} can be rewritten as
\be\label{eqn:Avecdx}
\frac{1}{4}\left(\sA-\frac{1}{2}\right)\vecxd
=\matAX\vecxd,\qquad
\vecf^\dagger\vecxd=0.
\ee
The matrix $\matAX$ is clearly not anti-hermitian, but if we can find a hermitian $\infty\times\infty$ matrix $\matPX=\matPX^\dagger$ and a vector $\veck$ such that $\matPX\matAX-\veck\vecf^\dagger$ is anti-hermitian, i.e.,
\be\label{eqn:PXAXhermitian}
(\matPX\matAX)+(\matPX\matAX)^\dagger = 
\veck\vecf^\dagger+\vecf\veck^\dagger,
\ee
Then, \eqref{eqn:Avecdx} leads to
\be\label{eqn:HPwish}
\frac{1}{4}\left(\sA-\frac{1}{2}\right)=\frac{\vecxd^\dagger\matPX\matAX\vecxd}{\vecxd^\dagger\matPX\vecxd}
=\frac{\vecxd^\dagger(\matPX\matAX-\veck\vecf^\dagger)\vecxd}{\vecxd^\dagger\matPX\vecxd}
\in i\R,
\ee
provided that all sums involved are absolutely convergent, and that $\vecxd^\dagger\matPX\vecxd\neq 0$. Thus, if we can find a solution to \eqref{eqn:PXAXhermitian} that guarantees $\vecxd^\dagger\matPX\vecxd\neq 0$, the Hilbert-P\'olya program will be realized!
Moreover, translating \eqref{eqn:ABlp} to a similar relation for $\matAX$:
\be\label{eqn:ABlpT}
\matPX\matAX\vecxd=\tfrac{1}{4}\left(\sA-\tfrac{1}{2}\right)\matPX\vecxd,\qquad
\matPX\matAX\vecxdt=\tfrac{1}{4}\left(\sA-\tfrac{1}{2}\right)\matPX\vecxdt+\tfrac{1}{4}\matPX\vecxd,\qquad
\vecf^\dagger\vecxd=\vecf^\dagger\vecxdt=0,
\ee
from which it follows, assuming $\sA-\frac{1}{2}$ is imaginary, that
\bear
\vecxd^\dagger\matPX\matAX\vecxdt
&=&
\tfrac{1}{4}\left(\sA-\tfrac{1}{2}\right)\vecxd^\dagger\matPX\vecxdt+\tfrac{1}{4}\vecxd^\dagger\matPX\vecxd
=\tfrac{1}{4}\left(\sA-\tfrac{1}{2}\right)\left(\vecxdt^\dagger\matPX\vecxd\right)^\star+\tfrac{1}{4}\vecxd^\dagger\matPX\vecxd
\nn\\
&=&
-\left(\vecxdt^\dagger\matPX\matAX\vecxd\right)^\star+\tfrac{1}{4}\vecxd^\dagger\matPX\vecxd
=-\vecxd^\dagger\left(\matPX\matAX\right)^\dagger\vecxd+\tfrac{1}{4}\vecxd^\dagger\matPX\vecxd.
\nn
\eear
But then \eqref{eqn:PXAXhermitian} and $\vecf^\dagger\vecxd=\vecf^\dagger\vecxdt=0$ imply $\vecxd^\dagger\matPX\vecxd=0$. Thus, if we can show that $\vecxd^\dagger\matPX\vecxd\neq 0$, the simplicity of the zeros would follow as well!
Not surprisingly, implementation of this plan encounters some problems, as will be detailed below. However, in the process we will discover new functionals that annihilate $\seqxd$.

To solve \eqref{eqn:PXAXhermitian} we define generating functions
\be\label{eqn:gePgenf}
\genP(\tG,\pG)\eqdef\sum_{i,j=1}^\infty\matPX_{ij}\tG^{i-1}\pG^{j-1},\qquad
\genf(\tG)\eqdef\sum_{i=1}^\infty\veck_i\tG^{i-1}.
\ee

\begin{proposition}
The general solution of \eqref{eqn:PXAXhermitian} with $\matPX=\matPX^\dagger$ is given by
\bear
\genP(\tG,\pG) &=&
-\frac{
4\cosh^3\left(\frac{\tH}{2}\right)\cosh^3\left(\frac{\pH}{2}\right)
}{
\sinh\left(\frac{\tH}{2}\right)\sinh\left(\frac{\pH}{2}\right)
}\left\lbrack
\int_0^1
\frac{\tH\xygeng(\zG\pH)\zG d\zG}{\cosh^2\left(\frac{\zG\tH}{2}\right)}
+\int_0^1\frac{\pH\xygeng(\zG\tH)\zG d\zG}{\cosh^2\left(\frac{\zG\pH}{2}\right)}
\right\rbrack
\nn\\
&&\label{eqn:genPsolution}
\eear
where
\be\label{eqn:tHpH}
\tH\eqdef
\log\left(\frac{1+\sqrt{\tG}}{1-\sqrt{\tG}}\right)
\,,\qquad
\pH\eqdef
\log\left(\frac{1 + \sqrt{\pG}}{1 - \sqrt{\pG}}\right),
\ee
and
\be\label{eqn:xygengDef}
\xygeng(\xG)\eqdef
\frac{\sinh(\tfrac{1}{2}\xG)}{2\cosh^3(\tfrac{1}{2}\xG)}\genf(\tanh^2(\tfrac{1}{2}\xG)).
\ee
\end{proposition}

\begin{proof}
After some algebra, we turn \eqref{eqn:PXAXhermitian} into a partial differential equation for $\genP$:
\bear
\lefteqn{
\frac{1}{8}\left\lbrack\left(\frac{3\pG-1}{\sqrt{\pG}}\right)\log\left(\frac{1 + \sqrt{\pG}}{1 - \sqrt{\pG}}\right)-1\right\rbrack
\genP(\tG,\pG)
+\frac{1}{4}\sqrt{\pG}(\pG-1)\log\left(\frac{1 + \sqrt{\pG}}{1 - \sqrt{\pG}}\right)\partial_\pG\genP(\tG,\pG)
}\nn\\ 
&+&
\frac{1}{8}\left\lbrack\left(\frac{3\tG-1}{\sqrt{\tG}}\right)\log\left(\frac{1 + \sqrt{\tG}}{1 - \sqrt{\tG}}\right)-1\right\rbrack
\genP(\tG,\pG)
+\frac{1}{4}\sqrt{\tG}(\tG-1)\log\left(\frac{1 + \sqrt{\tG}}{1 - \sqrt{\tG}}\right)\partial_\tG\genP(\tG,\pG)
\nn\\ 
&=&
\frac{1}{2\sqrt{\tG}}\log\left(\frac{1+\sqrt{\tG}}{1-\sqrt{\tG}}\right)\genf(\pG)
+\frac{1}{2\sqrt{\pG}}\log\left(\frac{1+\sqrt{\pG}}{1-\sqrt{\pG}}\right)\genf(\tG).
\nn\\ &&
\label{eqn:APmodifiedA}
\eear
The solution of this PDE is easy to obtain after changing variables to
$$
\xi_\pm\eqdef
\log\log\left(\frac{1 + \sqrt{\pG}}{1 - \sqrt{\pG}}\right)
\pm\log\log\left(\frac{1 + \sqrt{\tG}}{1 - \sqrt{\tG}}\right),
$$
which converts \eqref{eqn:APmodifiedA} to a linear ODE in $\xi_{+}$.
Let $\ConstFunctionX$ be an arbitrary analytic function.
Then, the general solution to \eqref{eqn:APmodifiedA} is
\bear
\genP(\tG,\pG) &=&
\frac{
4\cosh^3\left(\frac{\tH}{2}\right)\cosh^3\left(\frac{\pH}{2}\right)
}{
\sinh\left(\frac{\tH}{2}\right)\sinh\left(\frac{\pH}{2}\right)
}\Bigl\lbrack
\Bigl(\frac{1}{\tH}+\frac{1}{\pH}\Bigr)\ConstFunctionX(\frac{\tH}{\pH})
\nn\\ &&\qquad\qquad\qquad\qquad
-\int_0^1
\frac{\tH\xygeng(\zG\pH)\zG d\zG}{\cosh^2\left(\frac{\zG\tH}{2}\right)}
-\int_0^1\frac{\pH\xygeng(\zG\tH)\zG d\zG}{\cosh^2\left(\frac{\zG\pH}{2}\right)}
\Bigr\rbrack\,,
\nn\\
&&\label{eqn:genPsolutionWithConstF}
\eear
where $\tH$ and $\pH$ are given by \eqref{eqn:tHpH}, $\xygeng$ is given by \eqref{eqn:xygengDef}, and  $\ConstFunctionX$ is an arbitrary function of a single variable.
In order for \eqref{eqn:genPsolution} to be analytic in $\tH$ and $\pH$ at $\tH=\pH=0$, we must set $\ConstFunctionX=0$, thus recovering \eqref{eqn:genPsolution}.
\end{proof}

For example, if we choose
$$
\xygeng(\xG) = \frac{\xG}{16\cosh^2\left(\frac{\xG}{2}\right)}\,,
$$
corresponding to
$$
\genf(\tG) = 
\frac{1}{8\sqrt{\tG}}\log\left(\frac{1 + \sqrt{\tG}}{1 - \sqrt{\tG}}\right)
=\frac{1}{4}\sum_{\kI=0}^\infty\frac{\tG^\kI}{2\kI+1}\,,
$$
we get, after expanding $\genP$ in a power series in $\tG,\pG$, an expression for the matrix elements of $\matPX$ in the form
\be\label{eqn:matPXspecial}
\matPX_{\nI\mI}=
\frac{1}{2}\int_0^1\SqP_{\nI-1}(\zG^2)\SqP_{\mI-1}(\zG^2)\zG^2 d\zG\,,
\ee
where $(\SqP_\nI)_{\nI=0}^\infty$ is a series of polynomials defined by
\bear
\sum_{n=0}^\infty\SqP_\nI(\zG)\tG^{\nI} &=&
\frac{
\tH\cosh^3\left(\frac{\tH}{2}\right)
}{
\sinh\left(\frac{\tH}{2}\right)
\cosh^2\left(\frac{\sqrt{\zG}\tH}{2}\right)
}
\nn\\
&=&
\frac{1}{\sqrt{\tG}}\left(\frac{1}{1-\tG}\right)
\frac{4\log\left(\frac{1+\sqrt{\tG}}{1-\sqrt{\tG}}\right)}{
\left(
\frac{1+\sqrt{\tG}}{1-\sqrt{\tG}}
\right)^{\sqrt{\zG}}
+\left(
\frac{1-\sqrt{\tG}}{1+\sqrt{\tG}}
\right)^{\sqrt{\zG}}
+2
}\,.
\nn\\
&&\label{eqn:SqPGenFun}
\eear
The first four polynomials in the series are
$$
\begin{array}{ll}
\SqP_0=2,
& 
\SqP_1=\tfrac{8}{3}-2\zG, \\ & \\
\SqP_2=\tfrac{46}{15}-4\zG +\tfrac{4}{3}\zG^2, & 
\SqP_3=\tfrac{352}{105} -\tfrac{88}{15}\zG +\tfrac{32}{9}\zG^2-\tfrac{34}{45}\zG^3,
\end{array}
$$
and $\matPX$ then takes the form
\be\label{eqn:matPXresult}
\matPX=
\begin{pmatrix}
\frac{2}{3} & \frac{22}{45} & \frac{26}{63} & \frac{5218}{14175} & \cdots \\
& & & & \\
\frac{22}{45} & \frac{382}{945} & \frac{1702}{4725} & \frac{4438}{
  13365} & \cdots \\
& & & & \\
\frac{26}{63} & \frac{1702}{4725} & \frac{17114}{51975} & \frac{65634094}{212837625}  & \cdots \\
& & & & \\
 \frac{5218}{14175} & \frac{4438}{13365} & \frac{65634094}{212837625} & \frac{1266926}{4343625}  & \cdots \\
& & & & \\
\vdots & \vdots & \vdots & \vdots & \ddots \\
\end{pmatrix}.
\ee
Since $\matAX$ is triangular, \eqref{eqn:PXAXhermitian} also holds for the finite truncated matrices, and indeed, with \eqref{eqn:matAXdef} and \eqref{eqn:matPXresult} we calculate
$$
(\matPX\matAX)+(\matPX\matAX)^\dagger = 
-\frac{1}{2}\begin{pmatrix}
\frac{1}{1\cdot1} & \frac{1}{1\cdot3} & \frac{1}{1\cdot5} & \frac{1}{1\cdot7}  & \cdots \\
& & & & \\
\frac{1}{3\cdot1} & \frac{1}{3\cdot3} & \frac{1}{3\cdot5} & \frac{1}{3\cdot7} & \cdots \\
& & & & \\
\frac{1}{5\cdot1} & \frac{1}{5\cdot3} & \frac{1}{5\cdot5} & \frac{1}{5\cdot7} & \cdots \\
& & & & \\
\frac{1}{7\cdot1} & \frac{1}{7\cdot3} & \frac{1}{7\cdot5} & \frac{1}{7\cdot7} & \cdots \\
& & & & \\
\vdots & \vdots & \vdots & \vdots & \ddots \\
\end{pmatrix},
$$
as expected.

While it seems at first sight that \eqref{eqn:matPXspecial} defines a positive definite hermitian matrix, as we had hoped for, we will see that for $\vecxd=\vecxd(\sA)$,
$$
\vecxd^\dagger\matPX\vecxd=0
$$
for $\sA$ that is a zero of the zeta function. In particular, we will show that
\be\label{eqn:sumSqPxd}
\sum_{\nI=0}^\infty\SqP_\nI(\zG)\xd_\nI=0,
\ee
possibly reminiscent of Salem's criterion \cite{Salem:1953}.

To see that \eqref{eqn:sumSqPxd} is plausible, we note that if we modify \eqref{eqn:genPsolution} to
\bear
\genPM(\tG,\pG) =
-\frac{
4\cosh^3\left(\frac{\tH}{2}\right)\cosh^3\left(\frac{\pH}{2}\right)
}{
\sinh\left(\frac{\tH}{2}\right)\sinh\left(\frac{\pH}{2}\right)
}\Bigl\lbrack
\int_0^1
\frac{\tH\xygeng(\zG\pH)\zG^{1+2\gM}d\zG}{\cosh^2\left(\frac{\zG\tH}{2}\right)}
+\int_0^1\frac{\pH\xygeng(\zG\tH)\zG^{1+2\gM}d\zG}{\cosh^2\left(\frac{\zG\pH}{2}\right)}
\Bigr\rbrack,
\nn\\
&&\label{eqn:genPsolutionModified}
\eear
for $\gM\in\R$,
then \eqref{eqn:PXAXhermitian} will be modified to
\be\label{eqn:PXAXhermitianModified}
(\matPXM\matAX)+(\matPXM\matAX)^\dagger-\frac{1}{2}\gM\matPXM = 
\veck\vecf^\dagger+\vecf\veck^\dagger,
\ee
and \eqref{eqn:HPwish} changes to
\be\label{eqn:HPwishModified}
\frac{1}{4}\left(\sA-\frac{1}{2}\right)=\frac{\vecxd^\dagger\matPXM\matAX\vecxd}{\vecxd^\dagger\matPXM\vecxd}
=\frac{1}{4}\gM+\frac{\vecxd^\dagger(\matPXM\matAX-\frac{1}{4}\gM\matPX-\veck\vecf^\dagger)\vecxd}{\vecxd^\dagger\matPXM\vecxd}
\in \frac{1}{4}\gM+i\R,
\ee
since $\matPXM\matAX-\frac{1}{4}\gM\matPX-\veck\vecf^\dagger$ is antihermitian by \eqref{eqn:PXAXhermitianModified}.
This would correspond to zeros off the critical line, at $\sA\in\frac{1}{2}+\gM+i\R$, while \eqref{eqn:matPXspecial} will be replaced by
\be\label{eqn:matPXspecialModified}
\matPXM_{\nI\mI}=
\frac{1}{2}\int_0^1\SqP_{\nI-1}(\zG^2)\SqP_{\mI-1}(\zG^2)\zG^{2+2\gM} d\zG\,,
\ee
still leading to nonzero $\vecxd^\dagger\matPXM\vecxd$, unless \eqref{eqn:sumSqPxd} holds.

It is not hard to check that
\be\label{eqn:SqP01}
\SqP_\nI(0)=2\sum_{\kI=0}^\nI\frac{1}{2\kI+1}
\,,\qquad
\SqP_\nI(1)=\frac{2}{2\nI+1}\,.
\ee
Thus \eqref{eqn:sumSqPxd} is true at $\zG=1$ and does not converge for $\zG=0$.

We can calculate \eqref{eqn:sumSqPxd} using \eqref{eqn:intgenFb} [with the definition of $\genFd$ in \eqref{eqn:genFdDef}] and \eqref{eqn:SqPGenFun}:
\bear
\sum_{\nI=0}^\infty\SqP_\nI(\zG)\xd_\nI &=&
\frac{1}{2\pi}\int_0^{2\pi}
\left\lbrack
\sum_{\nI=0}^\infty\SqP_\nI(\zG)e^{i\nI\theta}\right\rbrack
\left(
\sum_{\mI=0}^\infty \xd_\nI e^{-i\mI\theta}
\right)d\theta
\nn\\
&=&
-\frac{1}{2\pi}\int_{\frac{i\pi}{2}-\infty}^{\frac{i\pi}{2}+\infty}
\frac{\qG}{
\cosh^2\left(\frac{\sqrt{\zG}\qG}{2}\right)
}
\genFd\left\lbrack
\tanh^2\left(\frac{\qG}{2}\right)
\right\rbrack
d\qG
\nn\\
&&
\!\!\!\!\!\!\!\!\!\!\!\!\!\!\!\!\!\!
\!\!\!\!\!\!\!\!\!\!\!\!\!\!\!\!\!\!
=-\pi^{\sA-2}\sin\left(\frac{\pi\sA}{2}\right)
\int_{\frac{i\pi}{2}-\infty}^{\frac{i\pi}{2}+\infty}
\frac{\qG}{
\cosh^2\left(\frac{\sqrt{\zG}\qG}{2}\right)
}
\left\{
\int_0^\infty
\sin^2\left\lbrack\frac{(\qG -i\pi)x}{2\pi}\right\rbrack
\frac{x^{-\sA}dx}{\sinh x}
\right\}
d\qG
\nn\\
&&
\label{eqn:CalcSumSqPxd}
\eear
where we substituted, for $0<\theta<2\pi$,
$$
\qG\eqdef\log\left(
\frac{1+e^{\frac{1}{2}i\theta}}{1-e^{\frac{1}{2}i\theta}}
\right)
=\frac{i\pi}{2}+\log\cot\left(\frac{\theta}{4}\right).
$$
We deform the $\qG$ contour of integration to the real axis, and using the fact that $\qG/\cosh^2\left(\frac{\sqrt{\zG}\qG}{2}\right)$ is an odd function and 
$$
\frac{1}{2}\left\{\sin^2\left\lbrack\frac{(\qG -i\pi)x}{2\pi}\right\rbrack
-\sin^2\left\lbrack\frac{(-\qG -i\pi)x}{2\pi}\right\rbrack\right\}
=-\frac{1}{2}i\sin\left(\frac{\qG x}{\pi}\right)\sinh x,
$$
we get
\bear
\sum_{\nI=0}^\infty\SqP_\nI(\zG)\xd_\nI &=&
\frac{1}{2}i\pi^{\sA-2}\sin\left(\frac{\pi\sA}{2}\right)
\int_{-\infty}^\infty
\frac{\qG}{
\cosh^2\left(\frac{\sqrt{\zG}\qG}{2}\right)
}
\left\lbrack
\int_0^\infty
x^{-\sA}\sin\left(\frac{\qG x}{\pi}\right)dx
\right\rbrack
d\qG
\nn\\
&&
\label{eqn:CalcSumSqPxdB}
\eear
We evaluate \eqref{eqn:CalcSumSqPxdB} by changing the order or integration, expanding $\sin\left(\frac{\qG x}{\pi}\right)$ in a Taylor series, and first calculating
\bear
\int_{-\infty}^{\infty}
\frac{\qG\sin\left(\frac{\qG x}{\pi}\right)}{\cosh^2\left(\frac{\sqrt{\zG}\qG}{2}\right)}
d\qG
 &=&
\sum_{n=0}^\infty(-1)^n\frac{2^{2n+3}\zG^{-n-\frac{3}{2}}x^{2n+1}}{\pi^{2n+1}(2n+1)!}\int_{-\infty}^{\infty}
\frac{\qG^{2n+2}}{\cosh^2\qG}
d\qG
\nn\\
&=&
\pi\sum_{n=0}^\infty\frac{2^{2n+4}\zG^{-n-\frac{3}{2}}x^{2n+1}}{(2n+1)!}
\left(1-2^{-2n-1}\right)B_{2n+2}
\nn\\
&=&
\frac{4\pi}{\zG\sinh\left(\frac{x}{\sqrt{\zG}}\right)}\left\lbrack\frac{x}{\sqrt{\zG}}\coth\left(\frac{x}{\sqrt{\zG}}\right)-1\right\rbrack
\,.
\eear
Putting this result back into \eqref{eqn:CalcSumSqPxdB}, it remains to compute
\bear
\lefteqn{
\frac{4\pi}{\zG}
\int_0^\infty
\left\lbrack\frac{x}{\sqrt{\zG}}\coth\left(\frac{x}{\sqrt{\zG}}\right)-1\right\rbrack
\frac{x^{-\sA}dx}{\sinh\left(\frac{x}{\sqrt{\zG}}\right)}
}\nn\\
&=&
4\pi\zG^{-\frac{1+\sA}{2}}
\int_0^\infty
\left(x\coth x-1\right)
\frac{x^{-\sA}dx}{\sinh x}\,.
\nn\\
&&
\label{eqn:ThisVanishes}
\eear
From the comment below \eqref{eqn:SqP01}, we know that  \eqref{eqn:ThisVanishes} vanishes for $\zG=1$, and hence for all $\zG\neq 0$. We can also compute it directly using \eqref{eqn:intxsAsinh} with integration by parts and analytic continuation from $\Re\sA<0$ to $\Re\sA<1$:
\be\label{eqn:IntegralIsZero}
\int_0^\infty
\left(x\coth x-1\right)
\frac{x^{-\sA}dx}{\sinh x}
=
\frac{\sA\pi^{1-\sA}\eta(\sA)}{\sin\left(\frac{\pi\sA}{2}\right)}
\ee
which vanishes when $\eta(\sA)=0$.
Thus, \eqref{eqn:sumSqPxd} holds [with $\SqP_\nI$ defined by expanding \eqref{eqn:SqPGenFun}], and we have shown that if $\zeta(\sA)=0$ (and $\sA$ is not a negative even integer) then $\vecxd(\sA)^\dagger\matPXM\vecxd(\sA)=0$ for any solution $\matPXM=\matPXM^\dagger$ of \eqref{eqn:PXAXhermitianModified}.

\section{Discussion}
\label{sec:Discussion}

We have constructed a sequence of functionals on $\lBanach^p$ ($\infty>p>1$) that depend linearly on a parameter $\sA$ and whose kernels have a common nontrivial vector precisely when $\sA$ is a nontrivial zero of the Riemann zeta function. This statement can be viewed as a variant of the Berry-Keating observation about the role of the dilatation operator in the problem of the Riemann Hypothesis. It might be interesting to explore whether a simple matrix as \eqref{eqn:start} can also be constructed for Dirichlet or Artin $L$-functions, or automorphic $L$-functions, and whether the elements of the matrix correspond to any meaningful data of the automorphic representation. For example, the order of a zero at $\sA=1$ is easy to translate to a linear algebraic condition, similarly to \secref{sec:NonSimpleZeros}, and therefore if a condition in the spirit of \eqref{eqn:start} can be found for an $L$-function associated with an elliptic curve, it might be interesting to explore whether additional arithmetic data that appears in the Birch and Swinnerton-Dyer conjecture is also encoded in the matrix.

We have also explored whether it is possible to convert the matrix equation \eqref{eqn:start} to a hermitian spectral problem. This amounts to solving a partial differential equation, and we presented the general solution in \eqref{eqn:genPsolution}. We showed that the general solution does not actually yield a hermitian spectral problem because the change-of-basis matrix has a nontrivial kernel, but in the process we found additional constant (i.e., $\sA$-independent) functionals that annihilate an $\lBanach^p$ solution of \eqref{eqn:start}. These additional functionals are given by substituting any $0<\zG<1$ into \eqref{eqn:sumSqPxd}, where the polynomials $\SqP_\nI(\zG)$ are defined by their generating function \eqref{eqn:SqPGenFun}. Alternatively, we can take the $i^{th}$ additional functional to be the sequence of $i^{th}$ derivatives $\SqP_\nI^{(i)}(1)$ at $\zG=1$. It might be interesting to explore how to properly truncate the infinite matrix in \eqref{eqn:start} (or appropriate linear combinations of the rows), perhaps adding linear combinations of the additional functionals, to get a finite matrix whose determinant has solutions $\sZ$ that approximate the actual zeros of $\zeta(\sZ)$.

\bibliographystyle{amsplain}

%\bibliography{refs}

\begin{thebibliography}{99}


\bibitem{BerryKeating1}
M.~V.~Berry and J.~P.~Keating, 
``$H = xp$ and the Riemann zeros,'' in {\it Supersymmetry and Trace Formulae: Chaos and
Disorder} ed. J.~P.~Keating, D.~E.~Khmelnitskii and I.~V.~Lerner (New York: Plenum), 1999.

\bibitem{BerryKeating2}
M.~V.~Berry, J.~P.~Keating, Siam Review 41, 236, 1999.

\bibitem{Bender:2016wob}
C.~M.~Bender, D.~C.~Brody and M.~P.~M\"uller,
``Hamiltonian for the zeros of the Riemann zeta function,''
Phys. Rev. Lett. \textbf{118} (2017) no.13, 130201
doi:10.1103/PhysRevLett.118.130201
[arXiv:1608.03679 [quant-ph]].

\bibitem{Berry:1986}
{M. Berry, Riemann's zeta function: a model of quantum chaos,''}
Lecture Notes in Physics, 263, Springer (1986).

\bibitem{Connes:1998}
A.~Connes,
{\it Trace formula in noncommutative geometry and the zeros of the Riemann zeta function,}
arXiv:math/9811068 [math.NT]

\bibitem{Meyer:2005}
R.~Meyer,
A spectral interpretation for the zeros of the Riemann zeta function, 
Mathematisches Institut, Georg-August-Universit\"at G\"ottingen: Seminars Winter Term 2004/2005, Universit"atsdrucke G\"ottingen, 2005, pp. 117-137,
arXiv:math/0412277 [math.NT]


\bibitem{Sierra:2011tb}
G.~Sierra and J.~Rodriguez-Laguna,
``The $H=xp$ model revisited and the Riemann zeros,''
Phys. Rev. Lett. \textbf{106} (2011), 200201
doi:10.1103/PhysRevLett.106.200201
[arXiv:1102.5356 [math-ph]].

\bibitem{Sierra:2016rgn}
G.~Sierra,
``The Riemann zeros as spectrum and the Riemann hypothesis,''
Symmetry \textbf{11}, no.4, 494 (2019)
doi:10.3390/sym11040494
[arXiv:1601.01797 [math-ph]].



\bibitem{Bishop:2018mgy}
M.~Bishop, E.~Aiken and D.~Singleton,
``Modified commutation relationships from the Berry-Keating program,''
Phys. Rev. D \textbf{99} (2019) no.2, 026012
doi:10.1103/PhysRevD.99.026012
[arXiv:1810.03976 [physics.gen-ph]].


\bibitem{Schumayer:2011yp}
D.~Schumayer and D.~A.~W.~Hutchinson,
``Physics of the Riemann Hypothesis,''
Rev. Mod. Phys. \textbf{83}, 307-330 (2011)
doi:10.1103/RevModPhys.83.307
[arXiv:1101.3116 [math-ph]].


\bibitem{Sands:2022}
A.~Sands,
{``Automorphic Hamiltonians, Epstein Zeta Functions, and Kronecker Limit Formulas,''}
[arXiv:2208.02082 [math-NT]].

\bibitem{Remmen:2021zmc}
G.~N.~Remmen,
``Amplitudes and the Riemann Zeta Function,''
Phys. Rev. Lett. \textbf{127}, no.24, 241602 (2021)
doi:10.1103/PhysRevLett.127.241602
[arXiv:2108.07820 [hep-th]].

\bibitem{Bianchi:2022mhs}
M.~Bianchi, M.~Firrotta, J.~Sonnenschein and D.~Weissman,
``A measure for chaotic scattering amplitudes,''
[arXiv:2207.13112 [hep-th]].

\bibitem{Benjamin:2022pnx}
N.~Benjamin and C.~H.~Chang,
``Scalar Modular Bootstrap and Zeros of the Riemann Zeta Function,''
[arXiv:2208.02259 [hep-th]].

\bibitem{Honda:2022hvy}
M.~Honda and T.~Yoda,
``String theory, $\mathcal{N}=4$ SYM and Riemann hypothesis,''
[arXiv:2203.17091 [hep-th]].


\bibitem{Sonnenschein:2018aqf}
J.~Sonnenschein and D.~Weissman,
``Quantizing the rotating string with massive endpoints,''
JHEP \textbf{06}, 148 (2018)
doi:10.1007/JHEP06(2018)148
[arXiv:1801.00798 [hep-th]].

\bibitem{Sonnenschein:2020jbe}
J.~Sonnenschein and D.~Weissman,
``On the quantization of folded strings in non-critical dimensions,''
JHEP \textbf{12}, 120 (2020)
doi:10.1007/JHEP12(2020)120
[arXiv:2006.14634 [hep-th]].

\bibitem{Rzadkowski:2012}
G.~Rz\c{a}dkowski,
``On some expansions for the Euler Gamma function and the Riemann Zeta function,''
J. Comp. Appl. Math. {\bf 236} (2012) 3710-3719 [arXiv:1009.1955 [math.CA]].


\bibitem{Ball:2017}
K.~M.~Ball,
``Rational approximations to the zeta function,''
Proc. Roy. Soc. A {\bf 475} (2019) 2225 [arXiv:1706.07998].

\bibitem{Ball:2019}
K.~M.~Ball,
``Rational approximations to the zeta function II,''
arXiv:1810.01613


\bibitem{Pidduck}
F.~B.~Pidduck, 
{\it ``On the propagation of a disturbance in a fluid under
gravity,''} Proc.~Roy.~Soc. (London), {\bf A83} (1910), 347-356.

\bibitem{Lomont:2001}
J.~S.~Lomont and J.~Brillhart,
{\it Elliptic Polynomials,}
Taylor \& Francis, 2001.


\bibitem{Hardy}
G.~H.~Hardy, J.~E.~Littlewood, and G.~P\'olya, 
{\it Inequalities,} Cambridge, at the University Press,
1952, 1967 reprint of the 1952 2nd ed. MR0046395

\bibitem{NIST}
%Riemann Zeta Function,
{\it Digitial Library of Mathematical Functions,}
National Institute of Technology,
{\texttt{https://dlmf.nist.gov/25.5\#E8}}


\bibitem{Salem:1953}
R.~Salem, Sur une proposition \'equivalente \`a l'hipoth\`ese de Riemann, C. R. Acad. Sci. Paris 236 (1953), 1127–1128.


\end{thebibliography}

\end{document}